\documentclass{article}%
\usepackage{amsfonts}
\usepackage{graphicx}
\usepackage{amsmath}
\usepackage{amssymb}%
\providecommand{\U}[1]{\protect\rule{.1in}{.1in}}
\newtheorem{theorem}{Theorem}

\newtheorem{corollary}[theorem]{Corollary}

\newtheorem{lemma}[theorem]{Lemma}
\newtheorem{notation}[theorem]{Notation}

\newtheorem{proposition}[theorem]{Proposition}
\newtheorem{remark}[theorem]{Remark}

\newenvironment{proof}[1][Proof]{\textbf{#1.} }{\ \rule{0.5em}{0.5em}}
\begin{document}

\title{Higher-Dimensional General Jacobi Identities I}
\author{Hirokazu NISHIMURA\\Institute of Mathematics, University of Tsukuba\\Tsukuba, Ibaraki 305-8571\\Japan}
\maketitle

\begin{abstract}
It was shown by the author [International Journal of Theoretical Physics 36
(1997), 1099-1131] that what is called the general Jacobi identity, obtaining
in microcubes, underlies the Jacobi identity of vector fields. It is well
known in the theory of Lie algebras that a plethora of higher-dimensional
generalizations of the Jacobi identity hold, though they are usually
established not as a derivation on the nose from the axioms of Lie algebras
but by making an appeal to the so-called Poincar\'{e}-Birkhoff-Witt theorem
and the like. The general Jacobi identity was rediscovered by Kirill Mackenzie
in the second decade of this century [Geometric Methods in Physics, 357-366,
Birkh\"{a}user/Springer 2013]. The principal objective in this paper is to
investigate a four-dimensional generalization of the general Jacobi identity
in detail. In a subsequent paper we will propose a uniform method for
establishing a bevy of higher-dimensional generalizations of the general
Jacobi identity under a single umbrella.

\end{abstract}

\section{Introduction}

It is known in synthetic differential geometry (cf. \cite{ko} and \cite{la})
that vector fields on a microlinear space $M$ forms a Lie algebra, for which
the following antisymmetry holds:%

\begin{equation}
\left[  X_{1},X_{2}\right]  +\left[  X_{2},X_{1}\right]  =0 \label{1.1}%
\end{equation}
It was shown in \cite{kola}\ that a bit deeper theorem in the following
underlies the above identity.

\begin{theorem}
Let $M$\ be a microlinear space. Given microsquares $\gamma_{12},\gamma
_{21}:D^{2}\rightarrow M$ with $\gamma_{12}\mid D\left(  2\right)
=\gamma_{21}\mid D\left(  2\right)  $, we have
\begin{equation}
\left(  \gamma_{12}\overset{\cdot}{-}\gamma_{21}\right)  +\left(  \gamma
_{21}\overset{\cdot}{-}\gamma_{12}\right)  =0 \label{1.2}%
\end{equation}

\end{theorem}

Now we consider the famous Jacobi identity.
\begin{equation}
\left[  X_{1},\left[  X_{2},X_{3}\right]  \right]  +\left[  X_{2},\left[
X_{3},X_{1}\right]  \right]  +\left[  X_{3},\left[  X_{1},X_{2}\right]
\right]  =0\label{1.3}%
\end{equation}
It claims that the sum of $\left[  X_{1},\left[  X_{2},X_{3}\right]  \right]
$'s\ with the three cyclic permutations of $\left\{  1,2,3\right\}  $\ applied
vanishes. We note in passing that the three cyclic permutations of $\left\{
1,2,3\right\}  $\ are no other than the three even permutations of $\left\{
1,2,3\right\}  $. It has been demonstrated in \cite{ni1}, \cite{ni2},
\cite{ni3} and \cite{ni5} that the following deeper theorem underlies the
above identity.

\begin{theorem}
(General Jacobi Identity) Let $M$\ be a microlinear space. Given microcubes
$\gamma_{123},\gamma_{132},\gamma_{213},\gamma_{231},\gamma_{312},\gamma
_{321}:D^{3}\rightarrow M$ with
\begin{align*}
\gamma_{123}  &  \mid\left\{  \left(  d_{1},d_{2},d_{3}\right)  \in D^{3}\mid
d_{2}d_{3}=0\right\}  =\gamma_{132}\mid\left\{  \left(  d_{1},d_{2}%
,d_{3}\right)  \in D^{3}\mid d_{2}d_{3}=0\right\} \\
\gamma_{231}  &  \mid\left\{  \left(  d_{1},d_{2},d_{3}\right)  \in D^{3}\mid
d_{2}d_{3}=0\right\}  =\gamma_{321}\mid\left\{  \left(  d_{1},d_{2}%
,d_{3}\right)  \in D^{3}\mid d_{2}d_{3}=0\right\} \\
\gamma_{231}  &  \mid\left\{  \left(  d_{1},d_{2},d_{3}\right)  \in D^{3}\mid
d_{1}d_{3}=0\right\}  =\gamma_{213}\mid\left\{  \left(  d_{1},d_{2}%
,d_{3}\right)  \in D^{3}\mid d_{1}d_{3}=0\right\} \\
\gamma_{312}  &  \mid\left\{  \left(  d_{1},d_{2},d_{3}\right)  \in D^{3}\mid
d_{1}d_{3}=0\right\}  =\gamma_{132}\mid\left\{  \left(  d_{1},d_{2}%
,d_{3}\right)  \in D^{3}\mid d_{1}d_{3}=0\right\} \\
\gamma_{312}  &  \mid\left\{  \left(  d_{1},d_{2},d_{3}\right)  \in D^{3}\mid
d_{1}d_{2}=0\right\}  =\gamma_{321}\mid\left\{  \left(  d_{1},d_{2}%
,d_{3}\right)  \in D^{3}\mid d_{1}d_{2}=0\right\} \\
\gamma_{123}  &  \mid\left\{  \left(  d_{1},d_{2},d_{3}\right)  \in D^{3}\mid
d_{1}d_{2}=0\right\}  =\gamma_{213}\mid\left\{  \left(  d_{1},d_{2}%
,d_{3}\right)  \in D^{3}\mid d_{1}d_{2}=0\right\}
\end{align*}
we have
\begin{align}
&  \left(  \left(  \gamma_{123}\underset{1}{\overset{\cdot}{-}}\gamma
_{132}\right)  \overset{\cdot}{-}\left(  \gamma_{231}\underset{1}%
{\overset{\cdot}{-}}\gamma_{321}\right)  \right)  +\nonumber\\
&  \left(  \left(  \gamma_{231}\underset{2}{\overset{\cdot}{-}}\gamma
_{213}\right)  \overset{\cdot}{-}\left(  \gamma_{312}\underset{2}%
{\overset{\cdot}{-}}\gamma_{132}\right)  \right)  +\nonumber\\
&  \left(  \left(  \gamma_{312}\underset{3}{\overset{\cdot}{-}}\gamma
_{321}\right)  \overset{\cdot}{-}\left(  \gamma_{123}\underset{3}%
{\overset{\cdot}{-}}\gamma_{213}\right)  \right) \nonumber\\
&  =0 \label{1.4}%
\end{align}

\end{theorem}

The general Jacobi identity was rediscovered by Kirill Mackenzie \cite{ma} in
a somewhat different context. We add that the general Jacobi identity plays a
fundamental role in a combinatorial or geometric proof of Jacobi-like
identities in so-called Fr\"{o}licher-Nijenhuis calculus (cf. \cite{ni4}).

Now we consider the following four-dimensional analogue of the Jacobi
identity.
\begin{align}
&  \left[  X_{1},\left[  X_{2},\left[  X_{3},X_{4}\right]  \right]  \right]
+\left[  X_{1},\left[  X_{3},\left[  X_{4},X_{2}\right]  \right]  \right]
+\left[  X_{1},\left[  X_{4},\left[  X_{2},X_{3}\right]  \right]  \right]
+\nonumber\\
&  \left[  X_{2},\left[  X_{1},\left[  X_{4},X_{3}\right]  \right]  \right]
+\left[  X_{2},\left[  X_{3},\left[  X_{1},X_{4}\right]  \right]  \right]
+\left[  X_{2},\left[  X_{4},\left[  X_{3},X_{1}\right]  \right]  \right]
+\nonumber\\
&  \left[  X_{3},\left[  X_{1},\left[  X_{2},X_{4}\right]  \right]  \right]
+\left[  X_{3},\left[  X_{2},\left[  X_{4},X_{1}\right]  \right]  \right]
+\left[  X_{3},\left[  X_{4},\left[  X_{1},X_{2}\right]  \right]  \right]
+\nonumber\\
&  \left[  X_{4},\left[  X_{1},\left[  X_{3},X_{2}\right]  \right]  \right]
+\left[  X_{4},\left[  X_{2},\left[  X_{1},X_{3}\right]  \right]  \right]
+\left[  X_{4},\left[  X_{3},\left[  X_{2},X_{1}\right]  \right]  \right]
\nonumber\\
&  =0 \label{1.5}%
\end{align}
It claims that the sum of $\left[  X_{1},\left[  X_{2},\left[  X_{3}%
,X_{4}\right]  \right]  \right]  $'s\ with the twelve even permutations of
$\left\{  1,2,3,4\right\}  $\ applied vanishes.

The principal objective in this paper is to establish a four-dimensional
version of the general Jacobi identity underpinning the above identity
(\ref{1.5}). In a subsequent paper we will discuss a slew of
higher-dimensional general Jacobi identities underlying the higher-dimensional
Jacobi identities discussed in \cite{bh} and \cite{we} (the former called them
\textit{generalized Jacobi identities}) from a coherent standpoint. For a good
introduction to generalized Jacobi identities, the reader is referred to
Chapter 8 of \cite{re}. We know well that various higher-dimensional Jacobi
identities are logical consequences of the three-dimensional Jacobi identity,
but we guess that higher-dimensional general Jacobi identities are by no means
logical consequences of the three-dimensional general Jacobi identity. We
assume the reader to be familiar with \cite{la} up to Chapter 3.

\section{Strong Differences}

First we introduce the notion of a simplicial small object after \cite{ni1},
though in a somewhat generalized form.

\begin{notation}
(\underline{Simplicial small objects}) Let $n$ be a natural number. Given a
subset $\mathfrak{p}$\ of
\[
\left\{  \left(  i,j\right)  \in\mathbb{N}\times\mathbb{N}\mid1\leq i\leq
n,1\leq j\leq n,i\neq j\right\}
\]
and a subset $\xi$ of%
\[
\left\{  i\in\mathbb{N}\mid1\leq i\leq n\right\}
\]
$D^{n}\left\{  \mathfrak{p,}\xi\right\}  $ denotes the set
\[
\left\{  \left(  d_{1},...,d_{n}\right)  :D^{n}\mid d_{i}d_{j}=0\text{ for any
}\left(  i,j\right)  \in\mathfrak{p}\text{, }d_{i}=0\text{ for any }i\in
\xi\right\}
\]
which is surely a small object. By way of example, we have
\begin{align*}
D\left(  2\right)   &  =D^{2}\left\{  \left(  1,2\right)  \right\} \\
D(3)  &  =D^{3}\left\{  \left(  1,2\right)  ,\left(  1,3\right)  ,\left(
2,3\right)  \right\}
\end{align*}
and $D^{3}\left\{  1,3\right\}  $ can be identified with $D$ via the canonical
isomorphism%
\[
d\in D\mapsto\left(  0,d,0\right)  \in D^{3}\left\{  1,3\right\}
\]

\end{notation}

The notion of strong difference in synthetic differential geometry is based
upon the following lemma.

\begin{lemma}
\label{l2.1}(cf. the first Lemma in \S 3.4 of \cite{la}) The diagram
\begin{equation}%
\begin{array}
[c]{ccccc}
&  & D^{3}\left\{  \left(  1,3\right)  ,\left(  2,3\right)  \right\}  &  & \\
& \nearrow &  & \nwarrow & \\
D^{2} &  &  &  & D^{2}\\
& \nwarrow &  & \nearrow & \\
&  & D^{2}\left\{  \left(  1,2\right)  \right\}  &  &
\end{array}
\label{l2.1.1}%
\end{equation}
with the lower two arrows being the canonical injections and the upper two
arrows being%
\begin{align*}
j_{1}^{2}  &  :\left(  d_{1},d_{2}\right)  \in D^{2}\mapsto\left(  d_{1}%
,d_{2},d_{1}d_{2}\right)  \in D^{3}\left\{  \left(  1,3\right)  ,\left(
2,3\right)  \right\} \\
j_{2}^{2}  &  :\left(  d_{1},d_{2}\right)  \in D^{2}\mapsto\left(  d_{1}%
,d_{2},0\right)  \in D^{3}\left\{  \left(  1,3\right)  ,\left(  2,3\right)
\right\}
\end{align*}
from left to right is a quasi-colimit diagram.
\end{lemma}

\begin{corollary}
\label{cl2.1}Let $M$\ be a microlinear space with two microsquares $\gamma
_{1},\gamma_{2}:D^{2}\rightarrow M$ abiding by%
\[
\gamma_{1}\mid D^{2}\left\{  \left(  1,2\right)  \right\}  =\gamma_{2}\mid
D^{2}\left\{  \left(  1,2\right)  \right\}
\]
Then there exists a unique mapping%
\[
\mathfrak{n}_{\left(  \gamma_{1},\gamma_{2}\right)  }^{2}:D^{3}\left\{
\left(  1,3\right)  ,\left(  2,3\right)  \right\}  \rightarrow M
\]
such that $\mathfrak{n}_{\left(  \gamma_{1},\gamma_{2}\right)  }^{2}\circ
j_{1}^{2}=\gamma_{1}$ and $\mathfrak{n}_{\left(  \gamma_{1},\gamma_{2}\right)
}^{2}\circ j_{2}^{2}=\gamma_{2}$.
\end{corollary}

\begin{notation}
In the above notation in Corollary \ref{cl2.1} we write $\gamma_{1}%
\overset{\cdot}{-}\gamma_{2}$ for the mapping%
\[
d\in D\mapsto\mathfrak{n}_{\left(  \gamma_{1},\gamma_{2}\right)  }^{2}\left(
0,0,d\right)
\]

\end{notation}

The notion of strong difference can easily be relativized.

\begin{lemma}
\label{l2.2}Let $n$\ be a natural number. The diagram
\begin{equation}%
\begin{array}
[c]{ccccc}
&  & D^{n+3}\left\{  \left(  n+1,n+3\right)  ,\left(  n+2,n+3\right)  \right\}
&  & \\
& \nearrow &  & \nwarrow & \\
D^{n+2} &  &  &  & D^{n+2}\\
& \nwarrow &  & \nearrow & \\
&  & D^{n+2}\left\{  \left(  n+1,n+2\right)  \right\}  &  &
\end{array}
\label{l2.2.1}%
\end{equation}
with the lower two arrows being the canonical injections and the upper two
arrows being%
\begin{align*}
j_{1}^{n+2}  &  :\left(  d_{1},...,d_{n},d_{n+1},d_{n+2}\right)  \in
D^{n+2}\mapsto\left(  d_{1},...,d_{n},d_{n+1},d_{n+2},d_{n+1}d_{n+2}\right)
\in D^{n+3}\left\{  \left(  n+1,n+3\right)  ,\left(  n+2,n+3\right)  \right\}
\\
j_{2}^{n+2}  &  :\left(  d_{1},...,d_{n},d_{n+1},d_{n+2}\right)  \in
D^{n+2}\mapsto\left(  d_{1},...,d_{n},d_{n+1},d_{n+2},0\right)  \in
D^{n+3}\left\{  \left(  n+1,n+3\right)  ,\left(  n+2,n+3\right)  \right\}
\end{align*}
from left to right is a quasi-colimit diagram.
\end{lemma}

\begin{corollary}
\label{cl2.2}Let $n$\ be a natural number. Let $M$\ be a microlinear space
with two mappings $\gamma_{1},\gamma_{2}:D^{n+2}\rightarrow M$ abiding by%
\[
\gamma_{1}\mid D^{n+2}\left\{  \left(  n+1,n+2\right)  \right\}  =\gamma
_{2}\mid D^{n+2}\left\{  \left(  n+1,n+2\right)  \right\}
\]
Then there exists a unique mapping%
\[
\mathfrak{n}_{\left(  \gamma_{1},\gamma_{2}\right)  }^{n+2}:D^{n+3}\left\{
\left(  n+1,n+3\right)  ,\left(  n+2,n+3\right)  \right\}  \rightarrow M
\]
such that $\mathfrak{n}_{\left(  \gamma_{1},\gamma_{2}\right)  }^{n+2}\circ
j_{1}^{n+2}=\gamma_{1}$ and $\mathfrak{n}_{\left(  \gamma_{1},\gamma
_{2}\right)  }^{n+2}\circ j_{2}^{3}=\gamma_{2}$.
\end{corollary}

\begin{notation}
Let $n$\ be a natural number. Let $M$\ be a microlinear space. Given
$\gamma:D^{n}\rightarrow M$ and a permutation $\sigma$ of $\left\{
1,...,n\right\}  $, we write $\gamma^{\sigma}$ for the mapping%
\[
\left(  d_{1},...,d_{n}\right)  \in D^{n}\mapsto\gamma\left(  d_{\sigma
^{-1}\left(  1\right)  },...,d_{\sigma^{-1}\left(  n\right)  }\right)  \in M
\]

\end{notation}

\begin{notation}
Let $n$\ be a natural number. Let $M$\ be a microlinear space.

\begin{enumerate}
\item Given $\gamma_{1},\gamma_{2}:D^{n+2}\rightarrow M$ with%
\[
\gamma_{1}\mid D^{n+2}\left\{  \left(  n+1,n+2\right)  \right\}  =\gamma
_{2}\mid D^{n+2}\left\{  \left(  n+1,n+2\right)  \right\}
\]
we write $\gamma_{1}\overset{\cdot}{\underset{1...n}{-}}\gamma_{2}$ for the
mapping%
\[
\left(  d_{1},...,d_{n},d_{n+1}\right)  \in D^{n+1}\mapsto\mathfrak{n}%
_{\left(  \gamma_{1},\gamma_{2}\right)  }^{n+2}\left(  d_{1},...,d_{n}%
,0,0,d_{n+1}\right)  \in M
\]

\item Given a permutation $\sigma$ of $\left\{  1,...,n,n+1,n+2\right\}  $ and
$\gamma_{1},\gamma_{2}:D^{n+2}\rightarrow M$ with%
\[
\gamma_{1}\mid D^{n+2}\left\{  \left(  \sigma\left(  n+1\right)
,\sigma\left(  n+2\right)  \right)  \right\}  =\gamma_{2}\mid D^{n+2}\left\{
\left(  \sigma\left(  n+1\right)  ,\sigma\left(  n+2\right)  \right)
\right\}
\]
we write $\gamma_{1}\overset{\cdot}{\underset{\sigma\left(  1\right)
...\sigma\left(  n\right)  }{-}}\gamma_{2}$ for $\left(  \gamma_{1}\right)
^{\sigma}\overset{\cdot}{\underset{1...n}{-}}\left(  \gamma_{2}\right)
^{\sigma}$.
\end{enumerate}
\end{notation}

The following result is well known.

\begin{lemma}
\label{l2.3}(cf. Proposition 6 in \S 2.2 of \cite{la}) The diagram%
\begin{equation}%
\begin{array}
[c]{ccccc}
&  & D^{2}\left\{  \left(  1,2\right)  \right\}  &  & \\
& \nearrow &  & \nwarrow & \\
D^{2}\left\{  1\right\}  &  &  &  & D^{2}\left\{  2\right\} \\
& \nwarrow &  & \nearrow & \\
&  & 1 &  &
\end{array}
\label{l2.3.1}%
\end{equation}
with the four arrows being the canonical injections is a quasi-colimit diagram.
\end{lemma}

\begin{corollary}
\label{cl2.3}Let $M$\ be a microlinear space. Given $\gamma_{1},\gamma
_{2}:D^{2}\rightarrow M$,
\[
\gamma_{1}\mid D^{2}\left\{  \left(  1,2\right)  \right\}  =\gamma_{2}\mid
D^{2}\left\{  \left(  1,2\right)  \right\}
\]
obtains iff both%
\[
\gamma_{1}\mid D^{2}\left\{  1\right\}  =\gamma_{2}\mid D^{2}\left\{
1\right\}
\]
and%
\[
\gamma_{1}\mid D^{2}\left\{  2\right\}  =\gamma_{2}\mid D^{2}\left\{
2\right\}
\]
obtain.
\end{corollary}

It can readily be relativized and generalized.

\begin{lemma}
\label{l2.4}The diagram%
\begin{equation}%
\begin{array}
[c]{ccccc}
&  & D^{n+m_{1}+m_{2}}\left\{
\begin{array}
[c]{c}%
\left(  n+i,n+m_{1}+j\right)  \mid\\
1\leq i\leq m_{1},1\leq j\leq m_{2}%
\end{array}
\right\}  &  & \\
& \nearrow &  & \nwarrow & \\
D^{n+m_{1}+m_{2}}\left\{
\begin{array}
[c]{c}%
n+1,...,\\
n+m_{1}%
\end{array}
\right\}  &  &  &  & D^{n+m_{1}+m_{2}}\left\{
\begin{array}
[c]{c}%
n+m_{1}+1,...,\\
n+m_{1}+m_{2}%
\end{array}
\right\} \\
& \nwarrow &  & \nearrow & \\
&  & D^{n} &  &
\end{array}
\label{l2.4.1}%
\end{equation}
with the four arrows being the canonical injections is a quasi-colimit diagram.
\end{lemma}

\begin{corollary}
\label{cl2.4}Let $M$\ be a microlinear space. Given $\gamma_{1},\gamma
_{2}:D^{n+m_{1}+m_{2}}\rightarrow M$,
\begin{align*}
\gamma_{1}  &  \mid D^{n+m_{1}+m_{2}}\left\{  \left(  n+i,n+m_{1}+j\right)
\mid1\leq i\leq m_{1},1\leq j\leq m_{2}\right\} \\
&  =\gamma_{2}\mid D^{n+m_{1}+m_{2}}\left\{  \left(  n+i,n+m_{1}+j\right)
\mid1\leq i\leq m_{1},1\leq j\leq m_{2}\right\}
\end{align*}
obtains iff both%
\[
\gamma_{1}\mid D^{n+m_{1}+m_{2}}\left\{  n+1,...,n+m_{1}\right\}  =\gamma
_{2}\mid D^{n+m_{1}+m_{2}}\left\{  n+1,...,n+m_{1}\right\}
\]
and%
\[
\gamma_{1}\mid D^{n+m_{1}+m_{2}}\left\{  n+m_{1}+1,...,n+m_{1}+m_{2}\right\}
=\gamma_{2}\mid D^{n+m_{1}+m_{2}}\left\{  n+m_{1}+1,...,n+m_{1}+m_{2}\right\}
\]
obtain.
\end{corollary}

\begin{proposition}
\label{p2.1}Let $M$\ be a microlinear space. Then we have the following two statements:

\begin{enumerate}
\item Given
\[
\gamma_{1}:D^{4}\rightarrow M,\gamma_{2}:D^{4}\rightarrow M,\gamma_{3}%
:D^{4}\rightarrow M,\gamma_{4}:D^{4}\rightarrow M
\]
if it holds that
\begin{align}
\gamma_{1}  &  \mid D^{4}\left\{  (3,4)\right\}  =\gamma_{2}\mid D^{4}\left\{
(3,4)\right\} \nonumber\\
\gamma_{3}  &  \mid D^{4}\left\{  (3,4)\right\}  =\gamma_{4}\mid D^{4}\left\{
(3,4)\right\} \nonumber\\
\gamma_{1}  &  \mid D^{4}\left\{  \left(  2,3\right)  ,\left(  2,4\right)
\right\}  =\gamma_{3}\mid D^{4}\left\{  \left(  2,3\right)  ,\left(
2,4\right)  \right\} \label{p2.1.0.1}\\
\gamma_{2}  &  \mid D^{4}\left\{  \left(  2,3\right)  ,\left(  2,4\right)
\right\}  =\gamma_{4}\mid D^{4}\left\{  \left(  2,3\right)  ,\left(
2,4\right)  \right\}  \label{p2.1.0.2}%
\end{align}
then all of
\begin{align*}
&  \gamma_{1}\overset{\cdot}{\underset{12}{-}}\gamma_{2}\\
&  \gamma_{3}\overset{\cdot}{\underset{12}{-}}\gamma_{4}\\
&  \left(  \gamma_{1}\overset{\cdot}{\underset{12}{-}}\gamma_{2}\right)
\overset{\cdot}{\underset{1}{-}}\left(  \gamma_{3}\overset{\cdot}%
{\underset{12}{-}}\gamma_{4}\right)
\end{align*}
are well defined.

\item Given
\begin{align*}
\gamma_{1}  &  :D^{4}\rightarrow M,\gamma_{2}:D^{4}\rightarrow M,\gamma
_{3}:D^{4}\rightarrow M,\gamma_{4}:D^{4}\rightarrow M,\\
\gamma_{5}  &  :D^{4}\rightarrow M,\gamma_{6}:D^{4}\rightarrow M,\gamma
_{7}:D^{4}\rightarrow M,\gamma_{8}:D^{4}\rightarrow M
\end{align*}
if it holds that
\begin{align}
\gamma_{1}  &  \mid D^{4}\left\{  (3,4)\right\}  =\gamma_{2}\mid D^{4}\left\{
(3,4)\right\} \nonumber\\
\gamma_{3}  &  \mid D^{4}\left\{  (3,4)\right\}  =\gamma_{4}\mid D^{4}\left\{
(3,4)\right\} \nonumber\\
\gamma_{5}  &  \mid D^{4}\left\{  (3,4)\right\}  =\gamma_{6}\mid D^{4}\left\{
(3,4)\right\} \nonumber\\
\gamma_{7}  &  \mid D^{4}\left\{  (3,4)\right\}  =\gamma_{8}\mid D^{4}\left\{
(3,4)\right\} \nonumber\\
\gamma_{1}  &  \mid D^{4}\left\{  \left(  2,3\right)  ,\left(  2,4\right)
\right\}  =\gamma_{3}\mid D^{4}\left\{  \left(  2,3\right)  ,\left(
2,4\right)  \right\} \nonumber\\
\gamma_{2}  &  \mid D^{4}\left\{  \left(  2,3\right)  ,\left(  2,4\right)
\right\}  =\gamma_{4}\mid D^{4}\left\{  \left(  2,3\right)  ,\left(
2,4\right)  \right\} \nonumber\\
\gamma_{5}  &  \mid D^{4}\left\{  \left(  2,3\right)  ,\left(  2,4\right)
\right\}  =\gamma_{7}\mid D^{4}\left\{  \left(  2,3\right)  ,\left(
2,4\right)  \right\} \nonumber\\
\gamma_{6}  &  \mid D^{4}\left\{  \left(  2,3\right)  ,\left(  2,4\right)
\right\}  =\gamma_{8}\mid D^{4}\left\{  \left(  2,3\right)  ,\left(
2,4\right)  \right\} \nonumber\\
\gamma_{1}  &  \mid D^{4}\left\{  \left(  1,2\right)  ,\left(  1,3\right)
,\left(  1,4\right)  \right\}  =\gamma_{5}\mid D^{4}\left\{  \left(
1,2\right)  ,\left(  1,3\right)  ,\left(  1,4\right)  \right\}
\label{p2.1.0.3}\\
\gamma_{2}  &  \mid D^{4}\left\{  \left(  1,2\right)  ,\left(  1,3\right)
,\left(  1,4\right)  \right\}  =\gamma_{6}\mid D^{4}\left\{  \left(
1,2\right)  ,\left(  1,3\right)  ,\left(  1,4\right)  \right\}
\label{p2.1.0.4}\\
\gamma_{3}  &  \mid D^{4}\left\{  \left(  1,2\right)  ,\left(  1,3\right)
,\left(  1,4\right)  \right\}  =\gamma_{7}\mid D^{4}\left\{  \left(
1,2\right)  ,\left(  1,3\right)  ,\left(  1,4\right)  \right\}
\label{p2.1.0.5}\\
\gamma_{4}  &  \mid D^{4}\left\{  \left(  1,2\right)  ,\left(  1,3\right)
,\left(  1,4\right)  \right\}  =\gamma_{8}\mid D^{4}\left\{  \left(
1,2\right)  ,\left(  1,3\right)  ,\left(  1,4\right)  \right\}
\label{p2.1.0.6}%
\end{align}
then all of
\begin{align*}
&  \gamma_{1}\overset{\cdot}{\underset{12}{-}}\gamma_{2}\\
&  \gamma_{3}\overset{\cdot}{\underset{12}{-}}\gamma_{4}\\
&  \gamma_{5}\overset{\cdot}{\underset{12}{-}}\gamma_{6}\\
&  \gamma_{7}\overset{\cdot}{\underset{12}{-}}\gamma_{8}\\
&  \left(  \gamma_{1}\overset{\cdot}{\underset{12}{-}}\gamma_{2}\right)
\overset{\cdot}{\underset{1}{-}}\left(  \gamma_{3}\overset{\cdot}%
{\underset{12}{-}}\gamma_{4}\right) \\
&  \left(  \gamma_{5}\overset{\cdot}{\underset{12}{-}}\gamma_{6}\right)
\overset{\cdot}{\underset{1}{-}}\left(  \gamma_{7}\overset{\cdot}%
{\underset{12}{-}}\gamma_{8}\right) \\
&  \left(  \left(  \gamma_{1}\overset{\cdot}{\underset{12}{-}}\gamma
_{2}\right)  \overset{\cdot}{\underset{1}{-}}\left(  \gamma_{3}\overset{\cdot
}{\underset{12}{-}}\gamma_{4}\right)  \right)  \overset{\cdot}{-}\left(
\left(  \gamma_{5}\overset{\cdot}{\underset{12}{-}}\gamma_{6}\right)
\overset{\cdot}{\underset{1}{-}}\left(  \gamma_{7}\overset{\cdot}%
{\underset{12}{-}}\gamma_{8}\right)  \right)
\end{align*}
are well defined.
\end{enumerate}
\end{proposition}

\begin{proof}
We deal with the above two statements in order.

\begin{enumerate}
\item For the first statement, we have to show that
\[
\left(  \gamma_{1}\overset{\cdot}{\underset{12}{-}}\gamma_{2}\right)  \mid
D^{3}\left\{  \left(  2,3\right)  \right\}  =\left(  \gamma_{3}\overset{\cdot
}{\underset{12}{-}}\gamma_{4}\right)  \mid D^{3}\left\{  \left(  2,3\right)
\right\}
\]
which is, by dint of Corollary \ref{cl2.4}, tantamout to showing that
\begin{align}
\left(  \gamma_{1}\overset{\cdot}{\underset{12}{-}}\gamma_{2}\right)   &  \mid
D^{3}\left\{  2\right\}  =\left(  \gamma_{3}\overset{\cdot}{\underset{12}{-}%
}\gamma_{4}\right)  \mid D^{3}\left\{  2\right\} \label{p2.1.1}\\
\left(  \gamma_{1}\overset{\cdot}{\underset{12}{-}}\gamma_{2}\right)   &  \mid
D^{3}\left\{  3\right\}  =\left(  \gamma_{3}\overset{\cdot}{\underset{12}{-}%
}\gamma_{4}\right)  \mid D^{3}\left\{  3\right\}  \label{p2.1.2}%
\end{align}
because of the quasi-colimit diagram%
\[%
\begin{array}
[c]{ccccc}
&  & D^{3}\left\{  \left(  2,3\right)  \right\}  &  & \\
& \nearrow &  & \nwarrow & \\
D^{3}\left\{  2\right\}  &  &  &  & D^{3}\left\{  3\right\} \\
& \nwarrow &  & \nearrow & \\
&  & D^{3}\left\{  2,3\right\}  &  &
\end{array}
\]
with the four arrows being the canonical injections (Lemma \ref{l2.4} with
$n=m_{1}=m_{2}=1$). Due to the quasi-colimit diagram
\[%
\begin{array}
[c]{ccccc}
&  & D^{4}\left\{  \left(  2,3\right)  ,\left(  2,4\right)  \right\}  &  & \\
& \nearrow &  & \nwarrow & \\
D^{4}\left\{  2\right\}  &  &  &  & D^{4}\left\{  3,4\right\} \\
& \nwarrow &  & \nearrow & \\
&  & D^{4}\left\{  2,3,4\right\}  &  &
\end{array}
\]
with the four arrows being the canonical injections (Lemma \ref{l2.4} with
$n=m_{1}=1$ and $m_{2}=2$), the condition (\ref{p2.1.0.1}) is equivalent to
the conditions
\begin{align}
\gamma_{1}  &  \mid D^{4}\left\{  2\right\}  =\gamma_{3}\mid D^{4}\left\{
2\right\} \label{p2.1.3}\\
\gamma_{1}  &  \mid D^{4}\left\{  3,4\right\}  =\gamma_{3}\mid D^{4}\left\{
3,4\right\}  \label{p2.1.4}%
\end{align}
while the condition (\ref{p2.1.0.2}) is equivalent to the conditions
\begin{align}
\gamma_{2}  &  \mid D^{4}\left\{  2\right\}  =\gamma_{4}\mid D^{4}\left\{
2\right\} \label{p2.1.5}\\
\gamma_{2}  &  \mid D^{4}\left\{  3,4\right\}  =\gamma_{4}\mid D^{4}\left\{
3,4\right\}  \label{p2.1.6}%
\end{align}
In order to show that (\ref{p2.1.1}) obtains, we note that the quasi-colimit
diagram in (\ref{l2.2.1}) with $n=2$ is to be restricted to the quasi-colimit
diagram
\[%
\begin{array}
[c]{ccccc}
&  & D^{5}\left\{  2,\left(  3,5\right)  ,\left(  4,5\right)  \right\}  &  &
\\
& \nearrow &  & \nwarrow & \\
D^{4}\left\{  2\right\}  &  &  &  & D^{4}\left\{  2\right\} \\
& \nwarrow &  & \nearrow & \\
&  & D^{4}\left\{  2,\left(  3,4\right)  \right\}  &  &
\end{array}
\]
so that the conditions (\ref{p2.1.3}) and (\ref{p2.1.5}) imply (\ref{p2.1.1}).
It is easy to see that
\begin{align*}
\left(  \left(  \gamma_{1}\overset{\cdot}{\underset{12}{-}}\gamma_{2}\right)
\mid D^{3}\left\{  3\right\}  \right)  \circ i_{2}^{3}  &  =\left(  \gamma
_{1}\mid D^{4}\left\{  3,4\right\}  \right)  \circ i_{2}^{4}=\left(
\gamma_{2}\mid D^{4}\left\{  3,4\right\}  \right)  \circ i_{2}^{4}\\
\left(  \left(  \gamma_{3}\overset{\cdot}{\underset{12}{-}}\gamma_{4}\right)
\mid D^{3}\left\{  3\right\}  \right)  \circ i_{2}^{3}  &  =\left(  \gamma
_{3}\mid D^{4}\left\{  3,4\right\}  \right)  \circ i_{2}^{4}=\left(
\gamma_{4}\mid D^{4}\left\{  3,4\right\}  \right)  \circ i_{2}^{4}%
\end{align*}
obtain with
\begin{align*}
i_{2}^{3}  &  :\left(  d_{1},d_{2}\right)  \in D^{2}\mapsto\left(  d_{1}%
,d_{2},0\right)  \in D^{3}\\
i_{2}^{4}  &  :\left(  d_{1},d_{2}\right)  \in D^{2}\mapsto\left(  d_{1}%
,d_{2},0,0\right)  \in D^{4}%
\end{align*}
so that (\ref{p2.1.4}) and (\ref{p2.1.6}) imply (\ref{p2.1.2}).

\item For the second statement, we have to show that
\[
\left(  \left(  \gamma_{1}\overset{\cdot}{\underset{12}{-}}\gamma_{2}\right)
\overset{\cdot}{\underset{1}{-}}\left(  \gamma_{3}\overset{\cdot}%
{\underset{12}{-}}\gamma_{4}\right)  \right)  \mid D^{2}\left\{  \left(
1,2\right)  \right\}  =\left(  \left(  \gamma_{5}\overset{\cdot}{\underset
{12}{-}}\gamma_{6}\right)  \overset{\cdot}{\underset{1}{-}}\left(  \gamma
_{7}\overset{\cdot}{\underset{12}{-}}\gamma_{8}\right)  \right)  \mid
D^{2}\left\{  \left(  1,2\right)  \right\}
\]
which is tantamout to showing that
\begin{align}
\left(  \left(  \gamma_{1}\overset{\cdot}{\underset{12}{-}}\gamma_{2}\right)
\overset{\cdot}{\underset{1}{-}}\left(  \gamma_{3}\overset{\cdot}%
{\underset{12}{-}}\gamma_{4}\right)  \right)   &  \mid D^{2}\left\{
1\right\}  =\left(  \left(  \gamma_{5}\overset{\cdot}{\underset{12}{-}}%
\gamma_{6}\right)  \overset{\cdot}{\underset{1}{-}}\left(  \gamma_{7}%
\overset{\cdot}{\underset{12}{-}}\gamma_{8}\right)  \right)  \mid
D^{2}\left\{  1\right\} \label{p2.1.7}\\
\left(  \left(  \gamma_{1}\overset{\cdot}{\underset{12}{-}}\gamma_{2}\right)
\overset{\cdot}{\underset{1}{-}}\left(  \gamma_{3}\overset{\cdot}%
{\underset{12}{-}}\gamma_{4}\right)  \right)   &  \mid D^{2}\left\{
2\right\}  =\left(  \left(  \gamma_{5}\overset{\cdot}{\underset{12}{-}}%
\gamma_{6}\right)  \overset{\cdot}{\underset{1}{-}}\left(  \gamma_{7}%
\overset{\cdot}{\underset{12}{-}}\gamma_{8}\right)  \right)  \mid
D^{2}\left\{  2\right\}  \label{p2.1.8}%
\end{align}
because of the quasi-colimit diagram%
\[%
\begin{array}
[c]{ccccc}
&  & D^{2}\left\{  \left(  1,2\right)  \right\}  &  & \\
& \nearrow &  & \nwarrow & \\
D^{2}\left\{  1\right\}  &  &  &  & D^{2}\left\{  2\right\} \\
& \nwarrow &  & \nearrow & \\
&  & D^{2}\left\{  1,2\right\}  &  &
\end{array}
\]
with the four arrows being the canonical injections (Lemma \ref{l2.3}). Due to
the quasi-colimit diagram
\[%
\begin{array}
[c]{ccccc}
&  & D^{4}\left\{  (1,2),(1,3),(1,4)\right\}  &  & \\
& \nearrow &  & \nwarrow & \\
D^{4}\left\{  1\right\}  &  &  &  & D^{4}\left\{  2,3,4\right\} \\
& \nwarrow &  & \nearrow & \\
&  & D^{4}\left\{  1,2,3,4\right\}  &  &
\end{array}
\]
with the four arrows being the canonical injections (Lemma \ref{l2.4} with
$n=0$, $m_{1}=1$ and $m_{2}=3$), we have

\begin{itemize}
\item the condition (\ref{p2.1.0.3}) is equivalent to the conditions
\begin{align}
\gamma_{1}  &  \mid D^{4}\left\{  1\right\}  =\gamma_{5}\mid D^{4}\left\{
1\right\} \label{p2.1.9}\\
\gamma_{1}  &  \mid D^{4}\left\{  2,3,4\right\}  =\gamma_{5}\mid D^{4}\left\{
2,3,4\right\}  \label{p2.1.10}%
\end{align}

\item the condition (\ref{p2.1.0.4}) is equivalent to the conditions
\begin{align}
\gamma_{2}  &  \mid D^{4}\left\{  1\right\}  =\gamma_{6}\mid D^{4}\left\{
1\right\} \label{p2.1.11}\\
\gamma_{2}  &  \mid D^{4}\left\{  2,3,4\right\}  =\gamma_{6}\mid D^{4}\left\{
2,3,4\right\}  \label{p2.1.12}%
\end{align}

\item the condition (\ref{p2.1.0.5}) is equivalent to the conditions
\begin{align}
\gamma_{3}  &  \mid D^{4}\left\{  1\right\}  =\gamma_{7}\mid D^{4}\left\{
1\right\} \label{p2.1.13}\\
\gamma_{3}  &  \mid D^{4}\left\{  2,3,4\right\}  =\gamma_{7}\mid D^{4}\left\{
2,3,4\right\}  \label{p2.1.14}%
\end{align}

\item the condition (\ref{p2.1.0.6}) is equivalent to the conditions
\begin{align}
\gamma_{4}  &  \mid D^{4}\left\{  1\right\}  =\gamma_{8}\mid D^{4}\left\{
1\right\} \label{p2.1.15}\\
\gamma_{4}  &  \mid D^{4}\left\{  2,3,4\right\}  =\gamma_{8}\mid D^{4}\left\{
2,3,4\right\}  \label{p2.1.16}%
\end{align}

\end{itemize}

In order to show that (\ref{p2.1.7}) obtains, we note first that the
quasi-colimit diagram in \ref{l2.2.1} with $n=2$ is to be restricted to the
quasi-colimit diagram
\[%
\begin{array}
[c]{ccccc}
&  & D^{5}\left\{  1,\left(  3,5\right)  ,\left(  4,5\right)  \right\}  &  &
\\
& \nearrow &  & \nwarrow & \\
D^{4}\left\{  1\right\}  &  &  &  & D^{4}\left\{  1\right\} \\
& \nwarrow &  & \nearrow & \\
&  & D^{4}\left\{  1,\left(  3,4\right)  \right\}  &  &
\end{array}
\]
so that we have

\begin{itemize}
\item the conditions (\ref{p2.1.9}) and (\ref{p2.1.11}) imply
\begin{equation}
\left(  \gamma_{1}\overset{\cdot}{\underset{12}{-}}\gamma_{2}\right)  \mid
D^{3}\left\{  1\right\}  =\left(  \gamma_{5}\overset{\cdot}{\underset{12}{-}%
}\gamma_{6}\right)  \mid D^{3}\left\{  1\right\}  \label{p2.1.17}%
\end{equation}

\item the conditions (\ref{p2.1.13}) and (\ref{p2.1.15}) imply
\begin{equation}
\left(  \gamma_{3}\overset{\cdot}{\underset{12}{-}}\gamma_{4}\right)  \mid
D^{3}\left\{  1\right\}  =\left(  \gamma_{7}\overset{\cdot}{\underset{12}{-}%
}\gamma_{8}\right)  \mid D^{3}\left\{  1\right\}  \label{p2.1.18}%
\end{equation}

\end{itemize}

We note also that the quasi-colimit diagram in (\ref{l2.2.1}) with $n=1$\ is
to be restricted to the quasi-colimit diagram
\[%
\begin{array}
[c]{ccccc}
&  & D^{4}\left\{  1,\left(  2,4\right)  ,\left(  3,4\right)  \right\}  &  &
\\
& \nearrow &  & \nwarrow & \\
D^{3}\left\{  1\right\}  &  &  &  & D^{3}\left\{  1\right\} \\
& \nwarrow &  & \nearrow & \\
&  & D^{3}\left\{  1,\left(  2,3\right)  \right\}  &  &
\end{array}
\]
so that the conditions (\ref{p2.1.17}) and (\ref{p2.1,18}) imply the condition
(\ref{p2.1.7}). It is easy to see that
\begin{align*}
&  \left(  \left(  \left(  \gamma_{1}\overset{\cdot}{\underset{12}{-}}%
\gamma_{2}\right)  \overset{\cdot}{\underset{1}{-}}\left(  \gamma_{3}%
\overset{\cdot}{\underset{12}{-}}\gamma_{4}\right)  \right)  \mid
D^{2}\left\{  2\right\}  \right)  \circ i^{2}\\
&  =\left(  \left(  \gamma_{1}\overset{\cdot}{\underset{12}{-}}\gamma
_{2}\right)  \mid D^{3}\left\{  2,3\right\}  \right)  \circ i^{3}=\left(
\left(  \gamma_{3}\overset{\cdot}{\underset{12}{-}}\gamma_{4}\right)  \mid
D^{3}\left\{  2,3\right\}  \right)  \circ i^{3}\\
&  =\left(  \gamma_{1}\mid D^{4}\left\{  2,3,4\right\}  \right)  \circ
i^{4}=\left(  \gamma_{2}\mid D^{4}\left\{  2,3,4\right\}  \right)  \circ
i^{4}\\
&  =\left(  \gamma_{3}\mid D^{4}\left\{  2,3,4\right\}  \right)  \circ
i^{4}=\left(  \gamma_{4}\mid D^{4}\left\{  2,3,4\right\}  \right)  \circ i^{4}%
\end{align*}
and
\begin{align*}
&  \left(  \left(  \left(  \gamma_{5}\overset{\cdot}{\underset{12}{-}}%
\gamma_{6}\right)  \overset{\cdot}{\underset{1}{-}}\left(  \gamma_{7}%
\overset{\cdot}{\underset{12}{-}}\gamma_{8}\right)  \right)  \mid
D^{2}\left\{  2\right\}  \right)  \circ i^{2}\\
&  =\left(  \left(  \gamma_{5}\overset{\cdot}{\underset{12}{-}}\gamma
_{6}\right)  \mid D^{3}\left\{  2,3\right\}  \right)  \circ i^{3}=\left(
\left(  \gamma_{7}\overset{\cdot}{\underset{12}{-}}\gamma_{8}\right)  \mid
D^{3}\left\{  2,3\right\}  \right)  \circ i^{3}\\
&  =\left(  \gamma_{5}\mid D^{4}\left\{  2,3,4\right\}  \right)  \circ
i^{4}=\left(  \gamma_{6}\mid D^{4}\left\{  2,3,4\right\}  \right)  \circ
i^{4}\\
&  =\left(  \gamma_{7}\mid D^{4}\left\{  2,3,4\right\}  \right)  \circ
i^{4}=\left(  \gamma_{8}\mid D^{4}\left\{  2,3,4\right\}  \right)  \circ i^{4}%
\end{align*}
obtain with
\begin{align*}
i^{2}  &  :d\in D\mapsto\left(  d,0\right)  \in D^{2}\left\{  2\right\} \\
i^{3}  &  :d\in D\mapsto\left(  d,0,0\right)  \in D^{3}\left\{  2,3\right\} \\
i^{4}  &  :d\in D\mapsto\left(  d,0,0,0\right)  \in D^{4}\left\{
2,3,4\right\}
\end{align*}
so that (\ref{p2.1.10}), (\ref{p2.1.12}), (\ref{p2.1.14}) and (\ref{p2.1.16})
imply (\ref{p2.1.8}).
\end{enumerate}
\end{proof}

\begin{notation}
Let $M$\ be a microlinear space.

\begin{enumerate}
\item We denote by $\mathfrak{X}\left(  M\right)  $\ the totality of vector
fields on $M$. It forms a Lie algebra. We take the third viewpoint of a vector
field in the essentially equivalent three discussed in \S 3.2 of \cite{la}.
Namely, a vector field $X$\ on $M$\ is a mapping $d\in D\mapsto X_{d}\in
M^{M}$ with $X_{0}=\mathrm{id}_{M}$.

\item Given $X,...,X_{n}\in\mathfrak{X}\left(  M\right)  $, we denote by
$X_{n}\ast...\ast X_{1}$\ the mapping $\left(  d_{1},...,d_{n}\right)  \in
D^{n}\mapsto\left(  X_{n}\right)  _{d_{n}}\circ...\circ\left(  X_{1}\right)
_{d_{1}}\in M^{M}$.

\item We recall (cf. Proposition 8 in \S 3.4 of \cite{la}) that%
\[
\left[  X_{1},X_{2}\right]  =X_{2}\ast X_{1}\overset{\cdot}{-}\left(
X_{1}\ast X_{2}\right)  ^{\sigma_{21}}%
\]
with%
\[
\sigma_{21}=\left(
\begin{array}
[c]{c}%
12\\
21
\end{array}
\right)
\]
We recall (cf. Prposition 2.7 of \cite{ni1}) that%
\begin{align*}
&  \left[  X_{1},\left[  X_{2},X_{3}\right]  \right] \\
&  =\left(  X_{3}\ast X_{2}\ast X_{1}\underset{1}{\overset{\cdot}{-}}\left(
X_{2}\ast X_{3}\ast X_{1}\right)  ^{\sigma_{132}}\right)  \overset{\cdot}%
{-}\left(  \left(  X_{1}\ast X_{3}\ast X_{2}\right)  ^{\sigma_{231}}%
\underset{1}{\overset{\cdot}{-}}\left(  X_{1}\ast X_{2}\ast X_{3}\right)
^{\sigma_{321}}\right)
\end{align*}
with%
\[
\sigma_{132}=\left(
\begin{array}
[c]{c}%
123\\
132
\end{array}
\right)  ,\sigma_{231}=\left(
\begin{array}
[c]{c}%
123\\
312
\end{array}
\right)  ,\sigma_{321}=\left(
\begin{array}
[c]{c}%
123\\
321
\end{array}
\right)
\]
We note that%
\begin{align*}
&  \left[  X_{1},\left[  X_{2},\left[  X_{3},X_{4}\right]  \right]  \right] \\
&  =\left(
\begin{array}
[c]{c}%
\left(  X_{4}\ast X_{3}\ast X_{2}\ast X_{1}\underset{12}{\overset{\cdot}{-}%
}\left(  X_{3}\ast X_{4}\ast X_{2}\ast X_{1}\right)  ^{\sigma_{1243}}\right)
\underset{1}{\overset{\cdot}{-}}\\
\left(  \left(  X_{2}\ast X_{4}\ast X_{3}\ast X_{1}\right)  ^{\sigma_{1342}%
}\underset{12}{\overset{\cdot}{-}}\left(  X_{2}\ast X_{3}\ast X_{4}\ast
X_{1}\right)  ^{\sigma_{1432}}\right)
\end{array}
\right)  \overset{\cdot}{-}\\
&  \left(
\begin{array}
[c]{c}%
\left(  \left(  X_{1}\ast X_{4}\ast X_{3}\ast X_{2}\right)  ^{\sigma_{2341}%
}\underset{12}{\overset{\cdot}{-}}\left(  X_{1}\ast X_{3}\ast X_{4}\ast
X_{2}\right)  ^{\sigma_{2431}}\right)  \underset{1}{\overset{\cdot}{-}}\\
\left(  \left(  X_{1}\ast X_{2}\ast X_{4}\ast X_{3}\right)  ^{\sigma_{3421}%
}\underset{12}{\overset{\cdot}{-}}\left(  X_{1}\ast X_{2}\ast X_{3}\ast
X_{4}\right)  ^{\sigma_{4321}}\right)
\end{array}
\right)
\end{align*}
with%
\begin{align*}
\sigma_{1243}  &  =\left(
\begin{array}
[c]{c}%
1234\\
1243
\end{array}
\right)  ,\sigma_{1342}=\left(
\begin{array}
[c]{c}%
1234\\
1423
\end{array}
\right)  ,\sigma_{1432}=\left(
\begin{array}
[c]{c}%
1234\\
1432
\end{array}
\right)  ,\sigma_{2341}=\left(
\begin{array}
[c]{c}%
1234\\
4123
\end{array}
\right)  ,\\
\sigma_{2431}  &  =\left(
\begin{array}
[c]{c}%
1234\\
4132
\end{array}
\right)  ,\sigma_{3421}=\left(
\begin{array}
[c]{c}%
1234\\
4312
\end{array}
\right)  ,\sigma_{4321}=\left(
\begin{array}
[c]{c}%
1234\\
4321
\end{array}
\right)
\end{align*}

\end{enumerate}
\end{notation}

\section{\label{s3}A Four-Dimensional General Jacobi Identity}

\begin{theorem}
\label{t3.1}The diagram whose underlying directed graph consists of vertices
\begin{align*}
&  P,Q^{1234},Q^{1243},Q^{1324},Q^{1342},Q^{1423},Q^{1432},Q^{2134}%
,Q^{2143},Q^{2314},Q^{2341},Q^{2413},Q^{2431},\\
&  Q^{3124},Q^{3142},Q^{3214},Q^{3241},Q^{3412},Q^{3421},Q^{4123}%
,Q^{4132},Q^{4213},Q^{4231},Q^{4312},Q^{4321},\\
&  R_{12}^{1234,1243},R_{12}^{1342,1432},R_{12}^{2341,2431},R_{12}%
^{3421,4321},R_{12}^{2134,2143},R_{12}^{3412,4312},R_{13}^{1324,1342}%
,R_{13}^{1243,1423}\\
&  R_{13}^{3241,3421},R_{13}^{2431,4231},R_{13}^{3124,3142},R_{13}%
^{2413,4213},R_{14}^{1423,1432},R_{14}^{1234,1324},R_{14}^{4231,4321}%
,R_{14}^{2341,3241},\\
&  R_{14}^{4123,4132},R_{14}^{2314,3214},R_{23}^{2314,2341},R_{23}%
^{2143,2413},R_{23}^{3142,3412},R_{23}^{1432,4132},R_{23}^{3214,3241}%
,R_{23}^{1423,4123},\\
&  R_{24}^{2413,2431},R_{24}^{2134,2314},R_{24}^{4132,4312},R_{24}%
^{1342,3142},R_{24}^{4213,4231},R_{24}^{1324,3124},R_{34}^{3412,3421}%
,R_{34}^{3124,3214},\\
&  R_{34}^{4123,4213},R_{34}^{1243,2143},R_{34}^{4312,4321},R_{34}^{1234,2134}%
\end{align*}
and edges
\begin{align*}
f_{1234} &  :Q^{1234}\rightarrow P,f_{1243}:Q^{1243}\rightarrow P,f_{1324}%
:Q^{1324}\rightarrow P,f_{1342}:Q^{1342}\rightarrow P,\\
f_{1423} &  :Q^{1423}\rightarrow P,f_{1432}:Q^{1432}\rightarrow P,f_{2134}%
:Q^{2134}\rightarrow P,f_{2143}:Q^{2143}\rightarrow P,\\
f_{2314} &  :Q^{2314}\rightarrow P,f_{2341}:Q^{2341}\rightarrow P,f_{2413}%
:Q^{2413}\rightarrow P,f_{2431}:Q^{2431}\rightarrow P,\\
f_{3124} &  :Q^{3124}\rightarrow P,f_{3142}:Q^{3142}\rightarrow P,f_{3214}%
:Q^{3214}\rightarrow P,f_{3241}:Q^{3241}\rightarrow P,\\
f_{3412} &  :Q^{3412}\rightarrow P,f_{3421}:Q^{3421}\rightarrow P,f_{4123}%
:Q^{4123}\rightarrow P,f_{4132}:Q^{4132}\rightarrow P,\\
f_{4213} &  :Q^{4213}\rightarrow P,f_{4231}:Q^{4231}\rightarrow P,f_{4312}%
:Q^{4312}\rightarrow P,f_{4321}:Q^{4321}\rightarrow P,
\end{align*}%
\begin{align*}
g_{12}^{1234,1243} &  :R_{12}^{1234,1243}\rightarrow Q^{1234},h_{12}%
^{1234,1243}:R_{12}^{1234,1243}\rightarrow Q^{1243},\\
g_{12}^{1342,1432} &  :R_{12}^{1342,1432}\rightarrow Q^{1342},h_{12}%
^{1342,1432}:R_{12}^{1342,1432}\rightarrow Q^{1432},\\
g_{12}^{2341,2431} &  :R_{12}^{2341,2431}\rightarrow Q^{2341},h_{12}%
^{2341,2431}:R_{12}^{2341,2431}\rightarrow Q^{2431},\\
g_{12}^{3421,4321} &  :R_{12}^{3421,4321}\rightarrow Q^{3421},h_{12}%
^{3421,4321}:R_{12}^{3421,4321}\rightarrow Q^{4321},\\
g_{12}^{2134,2143} &  :R_{12}^{2134,2143}\rightarrow Q^{2134},h_{12}%
^{2134,2143}:R_{12}^{2134,2143}\rightarrow Q^{2143},\\
g_{12}^{3412,4312} &  :R_{12}^{3412,4312}\rightarrow Q^{3412},h_{12}%
^{3412,4312}:R_{12}^{3412,4312}\rightarrow Q^{4312},
\end{align*}%
\begin{align*}
g_{13}^{1324,1342} &  :R_{13}^{1324,1342}\rightarrow Q^{1324},h_{13}%
^{1324,1342}:R_{13}^{1324,1342}\rightarrow Q^{1342},\\
g_{13}^{1243,1423} &  :R_{13}^{1243,1423}\rightarrow Q^{1243},h_{13}%
^{1243,1423}:R_{13}^{1243,1423}\rightarrow Q^{1423},\\
g_{13}^{3241,3421} &  :R_{13}^{3241,3421}\rightarrow Q^{3241},h_{13}%
^{3241,3421}:R_{13}^{3241,3421}\rightarrow Q^{3421},\\
g_{13}^{2431,4231} &  :R_{13}^{2431,4231}\rightarrow Q^{2431},h_{13}%
^{2431,4231}:R_{13}^{2431,4231}\rightarrow Q^{4231},\\
g_{13}^{3124,3142} &  :R_{13}^{3124,3142}\rightarrow Q^{3124},h_{13}%
^{3124,3142}:R_{13}^{3124,3142}\rightarrow Q^{3142},\\
g_{13}^{2413,4213} &  :R_{13}^{2413,4213}\rightarrow Q^{2413},h_{13}%
^{2413,4213}:R_{13}^{2413,4213}\rightarrow Q^{4213},
\end{align*}%
\begin{align*}
g_{14}^{1423,1432} &  :R_{14}^{1423,1432}\rightarrow Q^{1423},h_{14}%
^{1423,1432}:R_{14}^{1423,1432}\rightarrow Q^{1432},\\
g_{14}^{1234,1324} &  :R_{14}^{1234,1324}\rightarrow Q^{1234},h_{14}%
^{1234,1324}:R_{14}^{1234,1324}\rightarrow Q^{1324},\\
g_{14}^{4231,4321} &  :R_{14}^{4231,4321}\rightarrow Q^{4231},h_{14}%
^{4231,4321}:R_{14}^{4231,4321}\rightarrow Q^{4321},\\
g_{14}^{2341,3241} &  :R_{14}^{2341,3241}\rightarrow Q^{2341},h_{14}%
^{2341,3241}:R_{14}^{2341,3241}\rightarrow Q^{3241},\\
g_{14}^{4123,4132} &  :R_{14}^{4123,4132}\rightarrow Q^{4123},h_{14}%
^{4123,4132}:R_{14}^{4123,4132}\rightarrow Q^{4132},\\
g_{14}^{2314,3214} &  :R_{14}^{2314,3214}\rightarrow Q^{2314},h_{14}%
^{2314,3214}:R_{14}^{2314,3214}\rightarrow Q^{3214},
\end{align*}%
\begin{align*}
g_{23}^{2314,2341} &  :R_{23}^{2314,2341}\rightarrow Q^{2314},h_{23}%
^{2314,2341}:R_{23}^{2314,2341}\rightarrow Q^{2341},\\
g_{23}^{2143,2413} &  :R_{23}^{2143,2413}\rightarrow Q^{2143},h_{23}%
^{2143,2413}:R_{23}^{2143,2413}\rightarrow Q^{2413},\\
g_{23}^{3142,3412} &  :R_{23}^{3142,3412}\rightarrow Q^{3142},h_{23}%
^{3142,3412}:R_{23}^{3142,3412}\rightarrow Q^{3412},\\
g_{23}^{1432,4132} &  :R_{23}^{1432,4132}\rightarrow Q^{1432},h_{23}%
^{1432,4132}:R_{23}^{1432,4132}\rightarrow Q^{4132},\\
g_{23}^{3214,3241} &  :R_{23}^{3214,3241}\rightarrow Q^{3214},h_{23}%
^{3214,3241}:R_{23}^{3214,3241}\rightarrow Q^{3241},\\
g_{23}^{1423,4123} &  :R_{23}^{1423,4123}\rightarrow Q^{1423},h_{23}%
^{1423,4123}:R_{23}^{1423,4123}\rightarrow Q^{4123},
\end{align*}%
\begin{align*}
g_{24}^{2413,2431} &  :R_{24}^{2413,2431}\rightarrow Q^{2413},h_{24}%
^{2413,2431}:R_{24}^{2413,2431}\rightarrow Q^{2431},\\
g_{24}^{2134,2314} &  :R_{24}^{2134,2314}\rightarrow Q^{2134},h_{24}%
^{2134,2314}:R_{24}^{2134,2314}\rightarrow Q^{2314},\\
g_{24}^{4132,4312} &  :R_{24}^{4132,4312}\rightarrow Q^{4132},h_{24}%
^{4132,4312}:R_{24}^{4132,4312}\rightarrow Q^{4312},\\
g_{24}^{1342,3142} &  :R_{24}^{1342,3142}\rightarrow Q^{1342},h_{24}%
^{1342,3142}:R_{24}^{1342,3142}\rightarrow Q^{3142},\\
g_{24}^{4213,4231} &  :R_{24}^{4213,4231}\rightarrow Q^{4213},h_{24}%
^{4213,4231}:R_{24}^{4213,4231}\rightarrow Q^{4231},\\
g_{24}^{1324,3124} &  :R_{24}^{1324,3124}\rightarrow Q^{1324},h_{24}%
^{1324,3124}:R_{24}^{1324,3124}\rightarrow Q^{3124},
\end{align*}%
\begin{align*}
g_{34}^{3412,3421} &  :R_{34}^{3412,3421}\rightarrow Q^{3412},h_{34}%
^{3412,3421}:R_{34}^{3412,3421}\rightarrow Q^{3421},\\
g_{34}^{3124,3214} &  :R_{34}^{3124,3214}\rightarrow Q^{3124},h_{34}%
^{3124,3214}:R_{34}^{3124,3214}\rightarrow Q^{3214},\\
g_{34}^{4123,4213} &  :R_{34}^{4123,4213}\rightarrow Q^{4123},h_{34}%
^{4123,4213}:R_{34}^{4123,4213}\rightarrow Q^{4213},\\
g_{34}^{1243,2143} &  :R_{34}^{1243,2143}\rightarrow Q^{1243},h_{34}%
^{1243,2143}:R_{34}^{1243,2143}\rightarrow Q^{2143},\\
g_{34}^{4312,4321} &  :R_{34}^{4312,4321}\rightarrow Q^{4312},h_{34}%
^{4312,4321}:R_{34}^{4312,4321}\rightarrow Q^{4321},\\
g_{34}^{1234,2134} &  :R_{34}^{1234,2134}\rightarrow Q^{1234},h_{34}%
^{1234,2134}:R_{34}^{1234,2134}\rightarrow Q^{2134}%
\end{align*}
with $P$\ being labelled
\[
D^{53}\left\{
\begin{array}
[c]{c}%
\left(  1,5\right)  ,\left(  2,5\right)  ,\left(  1,6\right)  ,\left(
3,6\right)  ,\left(  1,7\right)  ,\left(  4,7\right)  ,\left(  2,8\right)
,\left(  3,8\right)  ,\left(  2,9\right)  ,\left(  4,9\right)  ,\\
\left(  3,10\right)  ,\left(  4,10\right)  ,\left(  5,6\right)  ,\left(
5,7\right)  ,\left(  5,8\right)  ,\left(  5,9\right)  ,\left(  6,7\right)
,\left(  6,8\right)  ,\left(  6,10\right)  ,\left(  7,9\right)  ,\\
\left(  7,10\right)  ,\left(  8,9\right)  ,\left(  8,10\right)  ,\left(
9,10\right)  ,\left(  1,i_{11,15}\right)  ,\left(  2,i_{11,15}\right)
,\left(  3,i_{11,15}\right)  ,\\
\left(  5,i_{11,15}\right)  ,\left(  6,i_{11,15}\right)  ,\left(
7,i_{11,15}\right)  ,\left(  8,i_{11,15}\right)  ,\left(  9,i_{11,15}\right)
,\left(  10,i_{11,15}\right)  ,\\
\left(  1,i_{16,20}\right)  ,\left(  2,i_{16,20}\right)  ,\left(
4,i_{16,20}\right)  ,\left(  5,i_{16,20}\right)  ,\left(  6,i_{16,20}\right)
,\left(  7,i_{16,20}\right)  ,\\
\left(  8,i_{16,20}\right)  ,\left(  9,i_{16,20}\right)  ,\left(
10,i_{16,20}\right)  ,\left(  1,i_{21,25}\right)  \left(  3,i_{21,25}\right)
,\left(  4,i_{21,25}\right)  ,\\
\left(  5,i_{21,25}\right)  ,\left(  6,i_{21,25}\right)  ,\left(
7,i_{21,25}\right)  ,\left(  8,i_{21,25}\right)  ,\left(  9,i_{21,25}\right)
,\left(  10,i_{21,25}\right)  ,\\
\left(  2,i_{26,30}\right)  ,\left(  3,i_{26,30}\right)  ,\left(
4,i_{26,30}\right)  ,\left(  5,i_{26,30}\right)  ,\left(  6,i_{26,30}\right)
,\left(  7,i_{26,30}\right)  ,\\
\left(  8,i_{26,30}\right)  ,\left(  9,i_{26,30}\right)  ,\left(
10,i_{26,30}\right)  ,\left(  i_{11,15},i_{11,15}^{\prime}\right)  ,\left(
i_{11,15},i_{16,20}\right)  ,\\
\left(  i_{11,15},i_{21,25}\right)  ,\left(  i_{11,15},i_{26,30}\right)
,\left(  i_{16,20},i_{16,20}^{\prime}\right)  ,\left(  i_{16,20}%
,i_{21,25}\right)  ,\\
\left(  i_{16,20},i_{26,30}\right)  ,\left(  i_{21,25},i_{21,25}^{\prime
}\right)  ,\left(  i_{21,25},i_{26,30}\right)  ,\left(  i_{26,30}%
,i_{26,30}^{\prime}\right)  ,\\
\left(  1,i_{31,53}\right)  ,\left(  2,i_{31,53}\right)  ,\left(
3,i_{31,53}\right)  ,\left(  4,i_{31,53}\right)  ,\left(  5,i_{31,53}\right)
,\left(  6,i_{31,53}\right)  ,\\
\left(  7,i_{31,53}\right)  ,\left(  8,i_{31,53}\right)  ,\left(
9,i_{31,53}\right)  ,\left(  10,i_{31,53}\right)  ,\left(  i_{11,15}%
,i_{31,53}\right)  ,\\
\left(  i_{16,20},i_{31,53}\right)  ,\left(  i_{21,25},i_{31,53}\right)
,\left(  i_{26,30},i_{31,53}\right)  ,\left(  i_{31,53},i_{31,53}^{\prime
}\right)  \\
\mid1i_{11,15},i_{11,15}^{\prime},i_{16,20},i_{16,20}^{\prime},i_{21,25}%
,i_{21,25}^{\prime},i_{26,30},i_{26,30}^{\prime},i_{31,53},i_{31,53}^{\prime
}\in\mathbb{N},\\
1\leq i_{11,15}\leq15,11\leq i_{11,15}^{\prime}\leq15,16\leq i_{16,20}%
\leq20,\\
16\leq i_{16,20}^{\prime}\leq20,21\leq i_{21,25}\leq25,21\leq i_{21,25}%
^{\prime}\leq25,\\
26\leq i_{26,30}\leq30,26\leq i_{26,30}^{\prime}\leq30,31\leq i_{31,53}%
\leq53,\\
31\leq i_{31,53}^{\prime}\leq53
\end{array}
\right\}
\]
all of
\begin{align*}
&  Q^{1234},Q^{1243},Q^{1324},Q^{1342},Q^{1423},Q^{1432},Q^{2134}%
,Q^{2143},Q^{2314},Q^{2341},Q^{2413},Q^{2431},\\
&  Q^{3124},Q^{3142},Q^{3214},Q^{3241},Q^{3412},Q^{3421},Q^{4123}%
,Q^{4132},Q^{4213},Q^{4231},Q^{4312},Q^{4321}%
\end{align*}
being laballed $D^{4}$, all of
\[
R_{12}^{1234,1243},R_{12}^{1342,1432},R_{12}^{2341,2431},R_{12}^{3421,4321}%
,R_{12}^{2134,2143},R_{12}^{3412,4312}%
\]
being labelled $D^{4}\left\{  \left(  3,4\right)  \right\}  $, all of%
\[
R_{13}^{1324,1342},R_{13}^{1243,1423},R_{13}^{3241,3421},R_{13}^{2431,4231}%
,R_{13}^{3124,3142},R_{13}^{2413,4213}%
\]
being labelled $D^{4}\left\{  \left(  2,4\right)  \right\}  $, all of%
\[
R_{14}^{1423,1432},R_{14}^{1234,1324},R_{14}^{4231,4321},R_{14}^{2341,3241}%
,R_{14}^{4123,4132},R_{14}^{2314,3214}%
\]
being labelled $D^{4}\left\{  \left(  2,3\right)  \right\}  $, all of%
\[
R_{23}^{2314,2341},R_{23}^{2143,2413},R_{23}^{3142,3412},R_{23}^{1432,4132}%
,R_{23}^{3214,3241},R_{23}^{1423,4123}%
\]
being labelled $D^{4}\left\{  \left(  1,4\right)  \right\}  $, all of%
\[
R_{24}^{2413,2431},R_{24}^{2134,2314},R_{24}^{4132,4312},R_{24}^{1342,3142}%
,R_{24}^{4213,4231},R_{24}^{1324,3124}%
\]
being labelled $D^{4}\left\{  \left(  1,3\right)  \right\}  $%
\[
R_{34}^{3412,3421},R_{34}^{3124,3214},R_{34}^{4123,4213},R_{34}^{1243,2143}%
,R_{34}^{4312,4321},R_{34}^{1234,2134}%
\]
being labelled $D^{4}\left\{  \left(  1,2\right)  \right\}  $, the edges%
\begin{align*}
& f_{1234},f_{1243},f_{1324},f_{1342},f_{1423},f_{1432},f_{2134}%
,f_{2143},f_{2314},f_{2341},f_{2413},f_{2431},\\
& f_{3124},f_{3142},f_{3214},f_{3241},f_{3412},f_{3421},f_{4123}%
,f_{4132},f_{4213},f_{4231},f_{4312},f_{4321}%
\end{align*}
standing for mappings
\begin{align*}
&  f_{1234}\left(  d_{1},d_{2},d_{3},d_{4}\right)  \\
&  =\left(  d_{1},d_{2},d_{3},d_{4},\underset{5}{0},...,\underset{53}%
{0}\right)
\end{align*}%
\begin{align*}
&  f_{1243}\left(  d_{1},d_{2},d_{3},d_{4}\right)  \\
&  =\left(
\begin{array}
[c]{c}%
d_{1},d_{2},d_{3},d_{4},\underset{5}{0},...,\underset{9}{0},\underset
{10}{d_{3}d_{4}},\underset{11}{0},...,\underset{20}{0},d_{1}d_{3}%
d_{4},\underset{22}{0},...,\underset{25}{0},d_{2}d_{3}d_{4},\underset{27}%
{0},...,\underset{30}{0},\\
d_{1}d_{2}d_{3}d_{4},\underset{32}{0},...,\underset{53}{0}%
\end{array}
\right)
\end{align*}%
\begin{align*}
&  f_{1324}\left(  d_{1},d_{2},d_{3},d_{4}\right)  \\
&  =\left(
\begin{array}
[c]{c}%
d_{1},d_{2},d_{3},d_{4},\underset{5}{0},...,\underset{7}{0},d_{2}%
d_{3},\underset{9}{0},\underset{10}{0},d_{1}d_{2}d_{3},\underset{12}%
{0},...,\underset{26}{0},d_{2}d_{3}d_{4},\underset{28}{0},...,\underset{31}%
{0},\\
d_{1}d_{2}d_{3}d_{4},\underset{33}{0},...,\underset{53}{0}%
\end{array}
\right)
\end{align*}%
\begin{align*}
&  f_{1342}\left(  d_{1},d_{2},d_{3},d_{4}\right)  \\
&  =\left(
\begin{array}
[c]{c}%
d_{1},d_{2},d_{3},d_{4},\underset{5}{0},...,\underset{7}{0},d_{2}d_{3}%
,d_{2}d_{4},\underset{10}{0},d_{1}d_{2}d_{3},\underset{12}{0},...,\underset
{15}{0},d_{1}d_{2}d_{4},\underset{17}{0},...,\underset{27}{0},\\
d_{2}d_{3}d_{4},\underset{29}{0},...,\underset{32}{0},d_{1}d_{2}d_{3}%
d_{4},\underset{34}{0},...,\underset{53}{0}%
\end{array}
\right)
\end{align*}%
\begin{align*}
&  f_{1423}\left(  d_{1},d_{2},d_{3},d_{4}\right)  \\
&  =\left(
\begin{array}
[c]{c}%
d_{1},d_{2},d_{3},d_{4},\underset{5}{0},...,\underset{8}{0},d_{2}d_{4}%
,d_{3}d_{4},\underset{11}{0},...,\underset{15}{0},d_{1}d_{2}d_{4}%
,\underset{17}{0},...,\underset{20}{0},d_{1}d_{3}d_{4},\\
\underset{22}{0},...,\underset{28}{0},d_{2}d_{3}d_{4},\underset{30}%
{0},...,\underset{33}{0},d_{1}d_{2}d_{3}d_{4},\underset{35}{0},...,\underset
{53}{0}%
\end{array}
\right)
\end{align*}%
\begin{align*}
&  f_{1432}\left(  d_{1},d_{2},d_{3},d_{4}\right)  \\
&  =\left(
\begin{array}
[c]{c}%
d_{1},d_{2},d_{3},d_{4},\underset{5}{0},...,\underset{7}{0},d_{2}d_{3}%
,d_{2}d_{4},d_{3}d_{4},d_{1}d_{2}d_{3},\underset{12}{0},...,\underset{15}%
{0},d_{1}d_{2}d_{4},\underset{16}{0},...,\underset{20}{0},\\
d_{1}d_{3}d_{4},\underset{22}{0},...,\underset{29}{0},d_{2}d_{3}%
d_{4},\underset{31}{0},...,\underset{34}{0},d_{1}d_{2}d_{3}d_{4},\underset
{36}{0},...,\underset{53}{0}%
\end{array}
\right)
\end{align*}%
\begin{align*}
&  f_{2134}\left(  d_{1},d_{2},d_{3},d_{4}\right)  \\
&  =\left(  d_{1},d_{2},d_{3},d_{4},d_{1}d_{2},\underset{6}{0},...,\underset
{11}{0},d_{1}d_{2}d_{3},\underset{13}{0},...,\underset{16}{0},d_{1}d_{2}%
d_{4},\underset{18}{0},...,\underset{35}{0},d_{1}d_{2}d_{3}d_{4},\underset
{37}{0},...,\underset{53}{0}\right)
\end{align*}%
\begin{align*}
&  f_{2143}\left(  d_{1},d_{2},d_{3},d_{4}\right)  \\
&  =\left(
\begin{array}
[c]{c}%
d_{1},d_{2},d_{3},d_{4},d_{1}d_{2},\underset{6}{0},...,\underset{9}{0}%
,d_{3}d_{4},\underset{11}{0},d_{1}d_{2}d_{3},\underset{13}{0},...,\underset
{16}{0},d_{1}d_{2}d_{4},\underset{18}{0},...,\underset{20}{0},d_{1}d_{3}%
d_{4},\\
\underset{22}{0},...,\underset{25}{0},d_{2}d_{3}d_{4},\underset{27}%
{0},...,\underset{36}{0},d_{1}d_{2}d_{3}d_{4},\underset{38}{0},...,\underset
{53}{0}%
\end{array}
\right)
\end{align*}%
\begin{align*}
&  f_{2314}\left(  d_{1},d_{2},d_{3},d_{4}\right)  \\
&  =\left(
\begin{array}
[c]{c}%
d_{1},d_{2},d_{3},d_{4},d_{1}d_{2},d_{1}d_{3},\underset{7}{0},...,\underset
{12}{0},d_{1}d_{2}d_{3},\underset{14}{0},...,\underset{16}{0},d_{1}d_{2}%
d_{4},\underset{18}{0},...,\underset{21}{0},d_{1}d_{3}d_{4},\\
\underset{23}{0},...,\underset{37}{0},d_{1}d_{2}d_{3}d_{4},\underset{39}%
{0},...,\underset{53}{0}%
\end{array}
\right)
\end{align*}%
\begin{align*}
&  f_{2341}\left(  d_{1},d_{2},d_{3},d_{4}\right)  \\
&  =\left(
\begin{array}
[c]{c}%
d_{1},d_{2},d_{3},d_{4},d_{1}d_{2},d_{1}d_{3},d_{1}d_{4},\underset{8}%
{0},...,\underset{12}{0},d_{1}d_{2}d_{3},\underset{14}{0},...,\underset{17}%
{0},d_{1}d_{2}d_{4},\underset{19}{0},...,\underset{22}{0},\\
d_{1}d_{3}d_{4},\underset{24}{0},...,\underset{38}{0},d_{1}d_{2}d_{3}%
d_{4},\underset{40}{0},...,\underset{53}{0}%
\end{array}
\right)
\end{align*}%
\begin{align*}
&  f_{2413}\left(  d_{1},d_{2},d_{3},d_{4}\right)  \\
&  =\left(
\begin{array}
[c]{c}%
d_{1},d_{2},d_{3},d_{4},d_{1}d_{2},\underset{6}{0},d_{1}d_{4},\underset{8}%
{0},\underset{9}{0},d_{3}d_{4},\underset{11}{0},d_{1}d_{2}d_{3},\underset
{13}{0},...,\underset{17}{0},d_{1}d_{2}d_{4},\underset{19}{0},...,\underset
{23}{0},\\
d_{1}d_{3}d_{4},\underset{25}{0},d_{2}d_{3}d_{4},\underset{27}{0}%
,...,\underset{39}{0},d_{1}d_{2}d_{3}d_{4},\underset{41}{0},...,\underset
{53}{0}%
\end{array}
\right)
\end{align*}%
\begin{align*}
&  f_{2431}\left(  d_{1},d_{2},d_{3},d_{4}\right)  \\
&  =\left(
\begin{array}
[c]{c}%
d_{1},d_{2},d_{3},d_{4},d_{1}d_{2},d_{1}d_{3},d_{1}d_{4},\underset{8}%
{0},\underset{9}{0},d_{3}d_{4},\underset{11}{0},\underset{12}{0},d_{1}%
d_{2}d_{3},\underset{14}{0},...,\underset{17}{0},d_{1}d_{2}d_{4},\\
\underset{19}{0},...,\underset{24}{0},d_{1}d_{3}d_{4},d_{2}d_{3}%
d_{4},\underset{27}{0},...,\underset{40}{0},d_{1}d_{2}d_{3}d_{4},\underset
{42}{0},...,\underset{53}{0}%
\end{array}
\right)
\end{align*}%
\begin{align*}
&  f_{3124}\left(  d_{1},d_{2},d_{3},d_{4}\right)  \\
&  =\left(
\begin{array}
[c]{c}%
d_{1},d_{2},d_{3},d_{4},\underset{5}{0},d_{1}d_{3},\underset{7}{0},d_{2}%
d_{3},\underset{9}{0},...,\underset{13}{0},d_{1}d_{2}d_{3},\underset{15}%
{0},...,\underset{21}{0},d_{1}d_{3}d_{4},\underset{23}{0},...,\underset{26}%
{0},\\
d_{2}d_{3}d_{4},\underset{28}{0},...,\underset{41}{0},d_{1}d_{2}d_{3}%
d_{4},\underset{43}{0},...,\underset{53}{0}%
\end{array}
\right)
\end{align*}%
\begin{align*}
&  f_{3142}\left(  d_{1},d_{2},d_{3},d_{4}\right)  \\
&  =\left(
\begin{array}
[c]{c}%
d_{1},d_{2},d_{3},d_{4},\underset{5}{0},d_{1}d_{3},\underset{7}{0},d_{2}%
d_{3},d_{2}d_{4},\underset{10}{0},...,\underset{13}{0},d_{1}d_{2}%
d_{3},\underset{15}{0},d_{1}d_{2}d_{4},\underset{17}{0},...,\underset{21}%
{0},d_{1}d_{3}d_{4},\\
\underset{23}{0},...,\underset{27}{0},d_{2}d_{3}d_{4},\underset{29}%
{0},...,\underset{42}{0},d_{1}d_{2}d_{3}d_{4},\underset{44}{0},...,\underset
{53}{0}%
\end{array}
\right)
\end{align*}%
\begin{align*}
&  f_{3214}\left(  d_{1},d_{2},d_{3},d_{4}\right)  \\
&  =\left(
\begin{array}
[c]{c}%
d_{1},d_{2},d_{3},d_{4},d_{1}d_{2},d_{1}d_{3},\underset{7}{0},d_{2}%
d_{3},\underset{9}{0},...,\underset{14}{0},d_{1}d_{2}d_{3},\underset{16}%
{0},d_{1}d_{2}d_{4},\underset{18}{0},...,\underset{21}{0},d_{1}d_{3}d_{4},\\
\underset{23}{0},...,\underset{26}{0},d_{2}d_{3}d_{4},\underset{28}%
{0},...,\underset{43}{0},d_{1}d_{2}d_{3}d_{4},\underset{45}{0},...,\underset
{53}{0}%
\end{array}
\right)
\end{align*}%
\begin{align*}
&  f_{3241}\left(  d_{1},d_{2},d_{3},d_{4}\right)  \\
&  =\left(
\begin{array}
[c]{c}%
d_{1},d_{2},d_{3},d_{4},d_{1}d_{2},d_{1}d_{3},d_{1}d_{4},d_{2}d_{3}%
,\underset{9}{0},...,\underset{14}{0},d_{1}d_{2}d_{3},\underset{16}%
{0},\underset{17}{0},d_{1}d_{2}d_{4},\underset{19}{0},...,\underset{22}{0},\\
d_{1}d_{3}d_{4},\underset{24}{0},...,\underset{26}{0},d_{2}d_{3}%
d_{4},\underset{28}{0},...,\underset{44}{0},d_{1}d_{2}d_{3}d_{4},\underset
{46}{0},...,\underset{53}{0}%
\end{array}
\right)
\end{align*}%
\begin{align*}
&  f_{3412}\left(  d_{1},d_{2},d_{3},d_{4}\right)  \\
&  =\left(
\begin{array}
[c]{c}%
d_{1},d_{2},d_{3},d_{4},\underset{5}{0},d_{1}d_{3},d_{1}d_{4},d_{2}d_{3}%
,d_{2}d_{4},\underset{10}{0},...,\underset{13}{0},d_{1}d_{2}d_{3}%
,\underset{15}{0},...,\underset{18}{0},d_{1}d_{2}d_{4},\\
\underset{20}{0},...,\underset{22}{0},d_{1}d_{3}d_{4},\underset{24}%
{0},...,\underset{27}{0},d_{2}d_{3}d_{4},\underset{29}{0},...,\underset{45}%
{0},d_{1}d_{2}d_{3}d_{4},\underset{47}{0},...,\underset{53}{0}%
\end{array}
\right)
\end{align*}%
\begin{align*}
&  f_{3421}\left(  d_{1},d_{2},d_{3},d_{4}\right)  \\
&  =\left(
\begin{array}
[c]{c}%
d_{1},d_{2},d_{3},d_{4},d_{1}d_{2},d_{1}d_{3},d_{1}d_{4},d_{2}d_{3},d_{2}%
d_{4},\underset{10}{0},...,\underset{14}{0},d_{1}d_{2}d_{3},\underset{16}%
{0},...,\underset{19}{0},d_{1}d_{2}d_{4},\\
\underset{21}{0},\underset{22}{0},d_{1}d_{3}d_{4},\underset{24}{0}%
,...,\underset{27}{0},d_{2}d_{3}d_{4},\underset{29}{0},...,\underset{46}%
{0},d_{1}d_{2}d_{3}d_{4},\underset{48}{0},...,\underset{53}{0}%
\end{array}
\right)
\end{align*}%
\begin{align*}
&  f_{4123}\left(  d_{1},d_{2},d_{3},d_{4}\right)  \\
&  =\left(
\begin{array}
[c]{c}%
d_{1},d_{2},d_{3},d_{4},\underset{5}{0},\underset{6}{0},d_{1}d_{4}%
,\underset{7}{0},\underset{7}{0},d_{2}d_{4},d_{3}d_{4},\underset{11}%
{0},...,\underset{18}{0},d_{1}d_{2}d_{4},\underset{20}{0},...,\underset
{23}{0,}d_{1}d_{3}d_{4},\\
\underset{25}{0},...,\underset{28}{0},d_{2}d_{3}d_{4},\underset{30}%
{0},...,\underset{47}{0},d_{1}d_{2}d_{3}d_{4},\underset{49}{0},...,\underset
{53}{0}%
\end{array}
\right)
\end{align*}%
\begin{align*}
&  f_{4132}\left(  d_{1},d_{2},d_{3},d_{4}\right)  \\
&  =\left(
\begin{array}
[c]{c}%
d_{1},d_{2},d_{3},d_{4},\underset{5}{0},\underset{6}{0},d_{1}d_{4},d_{2}%
d_{3},d_{2}d_{4},d_{3}d_{4},d_{1}d_{2}d_{3},\underset{12}{0},...,\underset
{18}{0},d_{1}d_{2}d_{4},\underset{20}{0},...,\underset{23}{0},\\
d_{1}d_{3}d_{4},\underset{25}{0},...,\underset{29}{0},d_{2}d_{3}%
d_{4},\underset{31}{0},...,\underset{48}{0},d_{1}d_{2}d_{3}d_{4},\underset
{50}{0},...,\underset{53}{0}%
\end{array}
\right)
\end{align*}%
\begin{align*}
&  f_{4213}\left(  d_{1},d_{2},d_{3},d_{4}\right)  \\
&  =\left(
\begin{array}
[c]{c}%
d_{1},d_{2},d_{3},d_{4},d_{1}d_{2},\underset{6}{0},d_{1}d_{4},\underset{8}%
{0},d_{2}d_{4},d_{3}d_{4},\underset{11}{0},d_{1}d_{2}d_{3},\underset{13}%
{0},...,\underset{19}{0},d_{1}d_{2}d_{4},\\
\underset{21}{0},...,\underset{23}{0},d_{1}d_{3}d_{4},\underset{25}%
{0},...,\underset{28}{0},d_{2}d_{3}d_{4},\underset{30}{0},...,\underset{49}%
{0},d_{1}d_{2}d_{3}d_{4},\underset{51}{0},...,\underset{53}{0}%
\end{array}
\right)
\end{align*}%
\begin{align*}
&  f_{4231}\left(  d_{1},d_{2},d_{3},d_{4}\right)  \\
&  =\left(
\begin{array}
[c]{c}%
d_{1},d_{2},d_{3},d_{4},d_{1}d_{2},d_{1}d_{3},d_{1}d_{4},\underset{8}{0}%
,d_{2}d_{4},d_{3}d_{4},\underset{11}{0},\underset{12}{0},d_{1}d_{2}%
d_{3},\underset{14}{0},...,\underset{19}{0},\\
d_{1}d_{2}d_{4},\underset{21}{0},...,\underset{24}{0},d_{1}d_{3}%
d_{4},\underset{26}{0},...,\underset{28}{0},d_{2}d_{3}d_{4},\underset{30}%
{0},...,\underset{50}{0},d_{1}d_{2}d_{3}d_{4},\underset{52}{0},\underset
{53}{0}%
\end{array}
\right)
\end{align*}%
\begin{align*}
&  f_{4312}\left(  d_{1},d_{2},d_{3},d_{4}\right)  \\
&  =\left(
\begin{array}
[c]{c}%
d_{1},d_{2},d_{3},d_{4},\underset{5}{0},d_{1}d_{3},d_{1}d_{4},d_{2}d_{3}%
,d_{2}d_{4},d_{3}d_{4},\underset{11}{0},...,\underset{13}{0},d_{1}d_{2}%
d_{3},\underset{15}{0},...,\underset{18}{0},\\
d_{1}d_{2}d_{4},\underset{20}{0},...,\underset{24}{0},d_{1}d_{3}%
d_{4},\underset{26}{0},...,\underset{29}{0},d_{2}d_{3}d_{4},\underset{31}%
{0},...,\underset{51}{0},d_{1}d_{2}d_{3}d_{4},\underset{53}{0}%
\end{array}
\right)
\end{align*}%
\begin{align*}
&  f_{4321}\left(  d_{1},d_{2},d_{3},d_{4}\right)  \\
&  =\left(
\begin{array}
[c]{c}%
d_{1},d_{2},d_{3},d_{4},d_{1}d_{2},d_{1}d_{3},d_{1}d_{4},d_{2}d_{3},d_{2}%
d_{4},d_{3}d_{4},\underset{11}{0},...,\underset{14}{0},d_{1}d_{2}%
d_{3},\underset{16}{0},...,\underset{19}{0},\\
d_{1}d_{2}d_{4},\underset{21}{0},...,\underset{24}{0},d_{1}d_{3}%
d_{4},\underset{26}{0},...,\underset{29}{0},d_{2}d_{3}d_{4},\underset{31}%
{0},...,\underset{52}{0},d_{1}d_{2}d_{3}d_{4}%
\end{array}
\right)
\end{align*}
respectively, and all the other edges standing for the identity mappings
(e.g., both $g_{12}^{1234,1243}$\ and $h_{12}^{1234,1243}$\ standing for the
mapping $\left(  d_{1},d_{2},d_{3},d_{4}\right)  \in D^{4}\left\{  \left(
3,4\right)  \right\}  \mapsto\left(  d_{1},d_{2},d_{3},d_{4}\right)  \in
D^{4}$ while both $g_{13}^{1324,1342}$\ and $h_{13}^{1324,1342}$  \ standing
for the mapping  $\left(  d_{1},d_{2},d_{3},d_{4}\right)  \in D^{4}\left\{
\left(  2,4\right)  \right\}  \mapsto\left(  d_{1},d_{2},d_{3},d_{4}\right)
\in D^{4}$) is a quasi-colimit diagram.
\end{theorem}

\begin{proof}
Let $\theta^{1234}:Q^{1234}\rightarrow\mathbb{R}$, $\theta^{1243}%
:Q^{1243}\rightarrow\mathbb{R}$, $\theta^{1324}:Q^{1324}\rightarrow\mathbb{R}%
$, $\theta^{1342}:Q^{1342}\rightarrow\mathbb{R}$, $\theta^{1423}%
:Q^{1423}\rightarrow\mathbb{R}$, $\theta^{1432}:Q^{1432}\rightarrow\mathbb{R}%
$, $\theta^{2134}:Q^{2134}\rightarrow\mathbb{R}$, $\theta^{2143}%
:Q^{2143}\rightarrow\mathbb{R}$, $\theta^{2314}:Q^{2314}\rightarrow\mathbb{R}%
$, $\theta^{2341}:Q^{2341}\rightarrow\mathbb{R}$, $\theta^{2413}%
:Q^{2413}\rightarrow\mathbb{R}$, $\theta^{2431}:Q^{2431}\rightarrow\mathbb{R}%
$, $\theta^{3124}:Q^{3124}\rightarrow\mathbb{R}$, $\theta^{3142}%
:Q^{3142}\rightarrow\mathbb{R}$, $\theta^{3214}:Q^{3214}\rightarrow\mathbb{R}%
$, $\theta^{3241}:Q^{3241}\rightarrow\mathbb{R}$, $\theta^{3412}%
:Q^{3412}\rightarrow\mathbb{R}$, $\theta^{3421}:Q^{3421}\rightarrow\mathbb{R}%
$, $\theta^{4123}:Q^{4123}\rightarrow\mathbb{R}$, $\theta^{4132}%
:Q^{4132}\rightarrow\mathbb{R}$, $\theta^{4213}:Q^{4213}\rightarrow\mathbb{R}%
$, $\theta^{4231}:Q^{4231}\rightarrow\mathbb{R}$, $\theta^{4312}%
:Q^{4312}\rightarrow\mathbb{R}$ and $\theta^{4321}:Q^{4321}\rightarrow
\mathbb{R}$ be mappings, which are to be of the following forms by dint of the
general Kock-Lawvere axiom (cf. \S 2.1.3 of \cite{la}):%
\begin{align*}
&  \theta^{1234}\left(  d_{1},d_{2},d_{3},d_{4}\right) \\
&  =a^{1234}+a_{1}^{1234}d_{1}+a_{2}^{1234}d_{2}+a_{3}^{1234}d_{3}%
+a_{4}^{1234}d_{1}+a_{12}^{1234}d_{1}d_{2}+a_{13}^{1234}d_{1}d_{3}%
+a_{14}^{1234}d_{1}d_{4}+\\
&  a_{23}^{1234}d_{2}d_{3}+a_{24}^{1234}d_{2}d_{4}+a_{34}^{1234}d_{3}%
d_{4}+a_{123}^{1234}d_{1}d_{2}d_{3}+a_{124}^{1234}d_{1}d_{2}d_{4}%
+a_{134}^{1234}d_{1}d_{3}d_{4}+\\
&  a_{234}^{1234}d_{2}d_{3}d_{4}+a_{1234}^{1234}d_{1}d_{2}d_{3}d_{4}%
\end{align*}%
\begin{align*}
&  \theta^{1243}\left(  d_{1},d_{2},d_{3},d_{4}\right) \\
&  =a^{1243}+a_{1}^{1243}d_{1}+a_{2}^{1243}d_{2}+a_{3}^{1243}d_{3}%
+a_{4}^{1243}d_{1}+a_{12}^{1243}d_{1}d_{2}+a_{13}^{1243}d_{1}d_{3}%
+a_{14}^{1243}d_{1}d_{4}+\\
&  a_{23}^{1243}d_{2}d_{3}+a_{24}^{1243}d_{2}d_{4}+a_{34}^{1243}d_{3}%
d_{4}+a_{123}^{1243}d_{1}d_{2}d_{3}+a_{124}^{1243}d_{1}d_{2}d_{4}%
+a_{134}^{1243}d_{1}d_{3}d_{4}+\\
&  a_{234}^{1243}d_{2}d_{3}d_{4}+a_{1234}^{1243}d_{1}d_{2}d_{3}d_{4}%
\end{align*}%
\begin{align*}
&  \theta^{1324}\left(  d_{1},d_{2},d_{3},d_{4}\right) \\
&  =a^{1324}+a_{1}^{1324}d_{1}+a_{2}^{1324}d_{2}+a_{3}^{1324}d_{3}%
+a_{4}^{1324}d_{1}+a_{12}^{1324}d_{1}d_{2}+a_{13}^{1324}d_{1}d_{3}%
+a_{14}^{1324}d_{1}d_{4}+\\
&  a_{23}^{1324}d_{2}d_{3}+a_{24}^{1324}d_{2}d_{4}+a_{34}^{1324}d_{3}%
d_{4}+a_{123}^{1324}d_{1}d_{2}d_{3}+a_{124}^{1324}d_{1}d_{2}d_{4}%
+a_{134}^{1324}d_{1}d_{3}d_{4}+\\
&  a_{234}^{1324}d_{2}d_{3}d_{4}+a_{1234}^{1324}d_{1}d_{2}d_{3}d_{4}%
\end{align*}%
\begin{align*}
&  \theta^{1342}\left(  d_{1},d_{2},d_{3},d_{4}\right) \\
&  =a^{1342}+a_{1}^{1342}d_{1}+a_{2}^{1342}d_{2}+a_{3}^{1342}d_{3}%
+a_{4}^{1342}d_{1}+a_{12}^{1342}d_{1}d_{2}+a_{13}^{1342}d_{1}d_{3}%
+a_{14}^{1342}d_{1}d_{4}+\\
&  a_{23}^{1342}d_{2}d_{3}+a_{24}^{1342}d_{2}d_{4}+a_{34}^{1342}d_{3}%
d_{4}+a_{123}^{1342}d_{1}d_{2}d_{3}+a_{124}^{1342}d_{1}d_{2}d_{4}%
+a_{134}^{1342}d_{1}d_{3}d_{4}+\\
&  a_{234}^{1342}d_{2}d_{3}d_{4}+a_{1234}^{1342}d_{1}d_{2}d_{3}d_{4}%
\end{align*}%
\begin{align*}
&  \theta^{1423}\left(  d_{1},d_{2},d_{3},d_{4}\right) \\
&  =a^{1423}+a_{1}^{1423}d_{1}+a_{2}^{1423}d_{2}+a_{3}^{1423}d_{3}%
+a_{4}^{1423}d_{1}+a_{12}^{1423}d_{1}d_{2}+a_{13}^{1423}d_{1}d_{3}%
+a_{14}^{1423}d_{1}d_{4}+\\
&  a_{23}^{1423}d_{2}d_{3}+a_{24}^{1423}d_{2}d_{4}+a_{34}^{1423}d_{3}%
d_{4}+a_{123}^{1423}d_{1}d_{2}d_{3}+a_{124}^{1423}d_{1}d_{2}d_{4}%
+a_{134}^{1423}d_{1}d_{3}d_{4}+\\
&  a_{234}^{1423}d_{2}d_{3}d_{4}+a_{1234}^{1423}d_{1}d_{2}d_{3}d_{4}%
\end{align*}%
\begin{align*}
&  \theta^{1432}\left(  d_{1},d_{2},d_{3},d_{4}\right) \\
&  =a^{1432}+a_{1}^{1432}d_{1}+a_{2}^{1432}d_{2}+a_{3}^{1432}d_{3}%
+a_{4}^{1432}d_{1}+a_{12}^{1432}d_{1}d_{2}+a_{13}^{1432}d_{1}d_{3}%
+a_{14}^{1432}d_{1}d_{4}+\\
&  a_{23}^{1432}d_{2}d_{3}+a_{24}^{1432}d_{2}d_{4}+a_{34}^{1432}d_{3}%
d_{4}+a_{123}^{1432}d_{1}d_{2}d_{3}+a_{124}^{1432}d_{1}d_{2}d_{4}%
+a_{134}^{1432}d_{1}d_{3}d_{4}+\\
&  a_{234}^{1432}d_{2}d_{3}d_{4}+a_{1234}^{1432}d_{1}d_{2}d_{3}d_{4}%
\end{align*}%
\begin{align*}
&  \theta^{2134}\left(  d_{1},d_{2},d_{3},d_{4}\right) \\
&  =a^{2134}+a_{1}^{2134}d_{1}+a_{2}^{2134}d_{2}+a_{3}^{2134}d_{3}%
+a_{4}^{2134}d_{1}+a_{12}^{2134}d_{1}d_{2}+a_{13}^{2134}d_{1}d_{3}%
+a_{14}^{2134}d_{1}d_{4}+\\
&  a_{23}^{2134}d_{2}d_{3}+a_{24}^{2134}d_{2}d_{4}+a_{34}^{2134}d_{3}%
d_{4}+a_{123}^{2134}d_{1}d_{2}d_{3}+a_{124}^{2134}d_{1}d_{2}d_{4}%
+a_{134}^{2134}d_{1}d_{3}d_{4}+\\
&  a_{234}^{2134}d_{2}d_{3}d_{4}+a_{1234}^{2134}d_{1}d_{2}d_{3}d_{4}%
\end{align*}%
\begin{align*}
&  \theta^{2143}\left(  d_{1},d_{2},d_{3},d_{4}\right) \\
&  =a^{2143}+a_{1}^{2143}d_{1}+a_{2}^{2143}d_{2}+a_{3}^{2143}d_{3}%
+a_{4}^{2143}d_{1}+a_{12}^{2143}d_{1}d_{2}+a_{13}^{2143}d_{1}d_{3}%
+a_{14}^{2143}d_{1}d_{4}+\\
&  a_{23}^{2143}d_{2}d_{3}+a_{24}^{2143}d_{2}d_{4}+a_{34}^{2143}d_{3}%
d_{4}+a_{123}^{2143}d_{1}d_{2}d_{3}+a_{124}^{2143}d_{1}d_{2}d_{4}%
+a_{134}^{2143}d_{1}d_{3}d_{4}+\\
&  a_{234}^{2143}d_{2}d_{3}d_{4}+a_{1234}^{2143}d_{1}d_{2}d_{3}d_{4}%
\end{align*}%
\begin{align*}
&  \theta^{2314}\left(  d_{1},d_{2},d_{3},d_{4}\right) \\
&  =a^{2314}+a_{1}^{2314}d_{1}+a_{2}^{2314}d_{2}+a_{3}^{2314}d_{3}%
+a_{4}^{2314}d_{1}+a_{12}^{2314}d_{1}d_{2}+a_{13}^{2314}d_{1}d_{3}%
+a_{14}^{2314}d_{1}d_{4}+\\
&  a_{23}^{2314}d_{2}d_{3}+a_{24}^{2314}d_{2}d_{4}+a_{34}^{2314}d_{3}%
d_{4}+a_{123}^{2314}d_{1}d_{2}d_{3}+a_{124}^{2314}d_{1}d_{2}d_{4}%
+a_{134}^{2314}d_{1}d_{3}d_{4}+\\
&  a_{234}^{2314}d_{2}d_{3}d_{4}+a_{1234}^{2314}d_{1}d_{2}d_{3}d_{4}%
\end{align*}%
\begin{align*}
&  \theta^{2341}\left(  d_{1},d_{2},d_{3},d_{4}\right) \\
&  =a^{2341}+a_{1}^{2341}d_{1}+a_{2}^{2341}d_{2}+a_{3}^{2341}d_{3}%
+a_{4}^{2341}d_{1}+a_{12}^{2341}d_{1}d_{2}+a_{13}^{2341}d_{1}d_{3}%
+a_{14}^{2341}d_{1}d_{4}+\\
&  a_{23}^{2341}d_{2}d_{3}+a_{24}^{2341}d_{2}d_{4}+a_{34}^{2341}d_{3}%
d_{4}+a_{123}^{2341}d_{1}d_{2}d_{3}+a_{124}^{2341}d_{1}d_{2}d_{4}%
+a_{134}^{2341}d_{1}d_{3}d_{4}+\\
&  a_{234}^{2341}d_{2}d_{3}d_{4}+a_{1234}^{2341}d_{1}d_{2}d_{3}d_{4}%
\end{align*}%
\begin{align*}
&  \theta^{2413}\left(  d_{1},d_{2},d_{3},d_{4}\right) \\
&  =a^{2413}+a_{1}^{2413}d_{1}+a_{2}^{2413}d_{2}+a_{3}^{2413}d_{3}%
+a_{4}^{2413}d_{1}+a_{12}^{2413}d_{1}d_{2}+a_{13}^{2413}d_{1}d_{3}%
+a_{14}^{2413}d_{1}d_{4}+\\
&  a_{23}^{2413}d_{2}d_{3}+a_{24}^{2413}d_{2}d_{4}+a_{34}^{2413}d_{3}%
d_{4}+a_{123}^{2413}d_{1}d_{2}d_{3}+a_{124}^{2413}d_{1}d_{2}d_{4}%
+a_{134}^{2413}d_{1}d_{3}d_{4}+\\
&  a_{234}^{2413}d_{2}d_{3}d_{4}+a_{1234}^{2413}d_{1}d_{2}d_{3}d_{4}%
\end{align*}%
\begin{align*}
&  \theta^{2431}\left(  d_{1},d_{2},d_{3},d_{4}\right) \\
&  =a^{2431}+a_{1}^{2431}d_{1}+a_{2}^{2431}d_{2}+a_{3}^{2431}d_{3}%
+a_{4}^{2431}d_{1}+a_{12}^{2431}d_{1}d_{2}+a_{13}^{2431}d_{1}d_{3}%
+a_{14}^{2431}d_{1}d_{4}+\\
&  a_{23}^{2431}d_{2}d_{3}+a_{24}^{2431}d_{2}d_{4}+a_{34}^{2431}d_{3}%
d_{4}+a_{123}^{2431}d_{1}d_{2}d_{3}+a_{124}^{2431}d_{1}d_{2}d_{4}%
+a_{134}^{2431}d_{1}d_{3}d_{4}+\\
&  a_{234}^{2431}d_{2}d_{3}d_{4}+a_{1234}^{2431}d_{1}d_{2}d_{3}d_{4}%
\end{align*}%
\begin{align*}
&  \theta^{3124}\left(  d_{1},d_{2},d_{3},d_{4}\right) \\
&  =a^{3124}+a_{1}^{3124}d_{1}+a_{2}^{3124}d_{2}+a_{3}^{3124}d_{3}%
+a_{4}^{3124}d_{1}+a_{12}^{3124}d_{1}d_{2}+a_{13}^{3124}d_{1}d_{3}%
+a_{14}^{3124}d_{1}d_{4}+\\
&  a_{23}^{3124}d_{2}d_{3}+a_{24}^{3124}d_{2}d_{4}+a_{34}^{3124}d_{3}%
d_{4}+a_{123}^{3124}d_{1}d_{2}d_{3}+a_{124}^{3124}d_{1}d_{2}d_{4}%
+a_{134}^{3124}d_{1}d_{3}d_{4}+\\
&  a_{234}^{3124}d_{2}d_{3}d_{4}+a_{1234}^{3124}d_{1}d_{2}d_{3}d_{4}%
\end{align*}%
\begin{align*}
&  \theta^{3142}\left(  d_{1},d_{2},d_{3},d_{4}\right) \\
&  =a^{3142}+a_{1}^{3142}d_{1}+a_{2}^{3142}d_{2}+a_{3}^{3142}d_{3}%
+a_{4}^{3142}d_{1}+a_{12}^{3142}d_{1}d_{2}+a_{13}^{3142}d_{1}d_{3}%
+a_{14}^{3142}d_{1}d_{4}+\\
&  a_{23}^{3142}d_{2}d_{3}+a_{24}^{3142}d_{2}d_{4}+a_{34}^{3142}d_{3}%
d_{4}+a_{123}^{3142}d_{1}d_{2}d_{3}+a_{124}^{3142}d_{1}d_{2}d_{4}%
+a_{134}^{3142}d_{1}d_{3}d_{4}+\\
&  a_{234}^{3142}d_{2}d_{3}d_{4}+a_{1234}^{3142}d_{1}d_{2}d_{3}d_{4}%
\end{align*}%
\begin{align*}
&  \theta^{3214}\left(  d_{1},d_{2},d_{3},d_{4}\right) \\
&  =a^{3214}+a_{1}^{3214}d_{1}+a_{2}^{3214}d_{2}+a_{3}^{3214}d_{3}%
+a_{4}^{3214}d_{1}+a_{12}^{3214}d_{1}d_{2}+a_{13}^{3214}d_{1}d_{3}%
+a_{14}^{3214}d_{1}d_{4}+\\
&  a_{23}^{3214}d_{2}d_{3}+a_{24}^{3214}d_{2}d_{4}+a_{34}^{3214}d_{3}%
d_{4}+a_{123}^{3214}d_{1}d_{2}d_{3}+a_{124}^{3214}d_{1}d_{2}d_{4}%
+a_{134}^{3214}d_{1}d_{3}d_{4}+\\
&  a_{234}^{3214}d_{2}d_{3}d_{4}+a_{1234}^{3214}d_{1}d_{2}d_{3}d_{4}%
\end{align*}%
\begin{align*}
&  \theta^{3241}\left(  d_{1},d_{2},d_{3},d_{4}\right) \\
&  =a^{3241}+a_{1}^{3241}d_{1}+a_{2}^{3241}d_{2}+a_{3}^{3241}d_{3}%
+a_{4}^{3241}d_{1}+a_{12}^{3241}d_{1}d_{2}+a_{13}^{3241}d_{1}d_{3}%
+a_{14}^{3241}d_{1}d_{4}+\\
&  a_{23}^{3241}d_{2}d_{3}+a_{24}^{3241}d_{2}d_{4}+a_{34}^{3241}d_{3}%
d_{4}+a_{123}^{3241}d_{1}d_{2}d_{3}+a_{124}^{3241}d_{1}d_{2}d_{4}%
+a_{134}^{3241}d_{1}d_{3}d_{4}+\\
&  a_{234}^{3241}d_{2}d_{3}d_{4}+a_{1234}^{3241}d_{1}d_{2}d_{3}d_{4}%
\end{align*}%
\begin{align*}
&  \theta^{3412}\left(  d_{1},d_{2},d_{3},d_{4}\right) \\
&  =a^{3412}+a_{1}^{3412}d_{1}+a_{2}^{3412}d_{2}+a_{3}^{3412}d_{3}%
+a_{4}^{3412}d_{1}+a_{12}^{3412}d_{1}d_{2}+a_{13}^{3412}d_{1}d_{3}%
+a_{14}^{3412}d_{1}d_{4}+\\
&  a_{23}^{3412}d_{2}d_{3}+a_{24}^{3412}d_{2}d_{4}+a_{34}^{3412}d_{3}%
d_{4}+a_{123}^{3412}d_{1}d_{2}d_{3}+a_{124}^{3412}d_{1}d_{2}d_{4}%
+a_{134}^{3412}d_{1}d_{3}d_{4}+\\
&  a_{234}^{3412}d_{2}d_{3}d_{4}+a_{1234}^{3412}d_{1}d_{2}d_{3}d_{4}%
\end{align*}%
\begin{align*}
&  \theta^{3421}\left(  d_{1},d_{2},d_{3},d_{4}\right) \\
&  =a^{3421}+a_{1}^{3421}d_{1}+a_{2}^{3421}d_{2}+a_{3}^{3421}d_{3}%
+a_{4}^{3421}d_{1}+a_{12}^{3421}d_{1}d_{2}+a_{13}^{3421}d_{1}d_{3}%
+a_{14}^{3421}d_{1}d_{4}+\\
&  a_{23}^{3421}d_{2}d_{3}+a_{24}^{3421}d_{2}d_{4}+a_{34}^{3421}d_{3}%
d_{4}+a_{123}^{3421}d_{1}d_{2}d_{3}+a_{124}^{3421}d_{1}d_{2}d_{4}%
+a_{134}^{3421}d_{1}d_{3}d_{4}+\\
&  a_{234}^{3421}d_{2}d_{3}d_{4}+a_{1234}^{3421}d_{1}d_{2}d_{3}d_{4}%
\end{align*}%
\begin{align*}
&  \theta^{4123}\left(  d_{1},d_{2},d_{3},d_{4}\right) \\
&  =a^{4123}+a_{1}^{4123}d_{1}+a_{2}^{4123}d_{2}+a_{3}^{4123}d_{3}%
+a_{4}^{4123}d_{1}+a_{12}^{4123}d_{1}d_{2}+a_{13}^{4123}d_{1}d_{3}%
+a_{14}^{4123}d_{1}d_{4}+\\
&  a_{23}^{4123}d_{2}d_{3}+a_{24}^{4123}d_{2}d_{4}+a_{34}^{4123}d_{3}%
d_{4}+a_{123}^{4123}d_{1}d_{2}d_{3}+a_{124}^{4123}d_{1}d_{2}d_{4}%
+a_{134}^{4123}d_{1}d_{3}d_{4}+\\
&  a_{234}^{4123}d_{2}d_{3}d_{4}+a_{1234}^{4123}d_{1}d_{2}d_{3}d_{4}%
\end{align*}%
\begin{align*}
&  \theta^{4132}\left(  d_{1},d_{2},d_{3},d_{4}\right) \\
&  =a^{4132}+a_{1}^{4132}d_{1}+a_{2}^{4132}d_{2}+a_{3}^{4132}d_{3}%
+a_{4}^{4132}d_{1}+a_{12}^{4132}d_{1}d_{2}+a_{13}^{4132}d_{1}d_{3}%
+a_{14}^{4132}d_{1}d_{4}+\\
&  a_{23}^{4132}d_{2}d_{3}+a_{24}^{4132}d_{2}d_{4}+a_{34}^{4132}d_{3}%
d_{4}+a_{123}^{4132}d_{1}d_{2}d_{3}+a_{124}^{4132}d_{1}d_{2}d_{4}%
+a_{134}^{4132}d_{1}d_{3}d_{4}+\\
&  a_{234}^{4132}d_{2}d_{3}d_{4}+a_{1234}^{4132}d_{1}d_{2}d_{3}d_{4}%
\end{align*}%
\begin{align*}
&  \theta^{4213}\left(  d_{1},d_{2},d_{3},d_{4}\right) \\
&  =a^{4213}+a_{1}^{4213}d_{1}+a_{2}^{4213}d_{2}+a_{3}^{4213}d_{3}%
+a_{4}^{4213}d_{1}+a_{12}^{4213}d_{1}d_{2}+a_{13}^{4213}d_{1}d_{3}%
+a_{14}^{4213}d_{1}d_{4}+\\
&  a_{23}^{4213}d_{2}d_{3}+a_{24}^{4213}d_{2}d_{4}+a_{34}^{4213}d_{3}%
d_{4}+a_{123}^{4213}d_{1}d_{2}d_{3}+a_{124}^{4213}d_{1}d_{2}d_{4}%
+a_{134}^{4213}d_{1}d_{3}d_{4}+\\
&  a_{234}^{4213}d_{2}d_{3}d_{4}+a_{1234}^{4213}d_{1}d_{2}d_{3}d_{4}%
\end{align*}%
\begin{align*}
&  \theta^{4231}\left(  d_{1},d_{2},d_{3},d_{4}\right) \\
&  =a^{4231}+a_{1}^{4231}d_{1}+a_{2}^{4231}d_{2}+a_{3}^{4231}d_{3}%
+a_{4}^{4231}d_{1}+a_{12}^{4231}d_{1}d_{2}+a_{13}^{4231}d_{1}d_{3}%
+a_{14}^{4231}d_{1}d_{4}+\\
&  a_{23}^{4231}d_{2}d_{3}+a_{24}^{4231}d_{2}d_{4}+a_{34}^{4231}d_{3}%
d_{4}+a_{123}^{4231}d_{1}d_{2}d_{3}+a_{124}^{4231}d_{1}d_{2}d_{4}%
+a_{134}^{4231}d_{1}d_{3}d_{4}+\\
&  a_{234}^{4231}d_{2}d_{3}d_{4}+a_{1234}^{4231}d_{1}d_{2}d_{3}d_{4}%
\end{align*}%
\begin{align*}
&  \theta^{4312}\left(  d_{1},d_{2},d_{3},d_{4}\right) \\
&  =a^{4312}+a_{1}^{4312}d_{1}+a_{2}^{4312}d_{2}+a_{3}^{4312}d_{3}%
+a_{4}^{4312}d_{1}+a_{12}^{4312}d_{1}d_{2}+a_{13}^{4312}d_{1}d_{3}%
+a_{14}^{4312}d_{1}d_{4}+\\
&  a_{23}^{4312}d_{2}d_{3}+a_{24}^{4312}d_{2}d_{4}+a_{34}^{4312}d_{3}%
d_{4}+a_{123}^{4312}d_{1}d_{2}d_{3}+a_{124}^{4312}d_{1}d_{2}d_{4}%
+a_{134}^{4312}d_{1}d_{3}d_{4}+\\
&  a_{234}^{4312}d_{2}d_{3}d_{4}+a_{1234}^{4312}d_{1}d_{2}d_{3}d_{4}%
\end{align*}%
\begin{align*}
&  \theta^{4321}\left(  d_{1},d_{2},d_{3},d_{4}\right) \\
&  =a^{4321}+a_{1}^{4321}d_{1}+a_{2}^{4321}d_{2}+a_{3}^{4321}d_{3}%
+a_{4}^{4321}d_{1}+a_{12}^{4321}d_{1}d_{2}+a_{13}^{4321}d_{1}d_{3}%
+a_{14}^{4321}d_{1}d_{4}+\\
&  a_{23}^{4321}d_{2}d_{3}+a_{24}^{4321}d_{2}d_{4}+a_{34}^{4321}d_{3}%
d_{4}+a_{123}^{4321}d_{1}d_{2}d_{3}+a_{124}^{4321}d_{1}d_{2}d_{4}%
+a_{134}^{4321}d_{1}d_{3}d_{4}+\\
&  a_{234}^{4321}d_{2}d_{3}d_{4}+a_{1234}^{4321}d_{1}d_{2}d_{3}d_{4}%
\end{align*}
The conditions%
\begin{align*}
\theta^{1234}\circ g_{12}^{1234,1243}  &  =\theta^{1243}\circ h_{12}%
^{1234,1243},\theta^{1342}\circ g_{12}^{1342,1432}=\theta^{1432}\circ
h_{12}^{1342,1432},\\
\theta^{2341}\circ g_{12}^{2341,2431}  &  =\theta^{2431}\circ h_{12}%
^{2341,2431},\theta^{3421}\circ g_{12}^{3421,4321}=\theta^{4321}\circ
h_{12}^{3421,4321},\\
\theta^{2134}\circ g_{12}^{2134,2143}  &  =\theta^{2143}\circ h_{12}%
^{2134,2143},\theta^{3412}\circ g_{12}^{3412,4312}=\theta^{4312}\circ
h_{12}^{3412,4312},\\
\theta^{1324}\circ g_{13}^{1324,1342}  &  =\theta^{1342}\circ h_{13}%
^{1324,1342},\theta^{1243}\circ g_{13}^{1243,1423}=\theta^{1423}\circ
h_{13}^{1243,1423},\\
\theta^{3241}\circ g_{13}^{3241,3421}  &  =\theta^{3421}\circ h_{13}%
^{3241,3421},\theta^{2431}\circ g_{13}^{2431,4231}=\theta^{4231}\circ
h_{13}^{2431,4231},\\
\theta^{3124}\circ g_{13}^{3124,3142}  &  =\theta^{3142}\circ h_{13}%
^{3124,3142},\theta^{2413}\circ g_{13}^{2413,4213}=\theta^{4213}\circ
h_{13}^{2413,4213},\\
\theta^{1423}\circ g_{14}^{1423,1432}  &  =\theta^{1432}\circ h_{14}%
^{1423,1432},\theta^{1234}\circ g_{14}^{1234,1324}=\theta^{1324}\circ
h_{14}^{1234,1324},\\
\theta^{4231}\circ g_{14}^{4231,4321}  &  =\theta^{4321}\circ h_{14}%
^{4231,4321},\theta^{2341}\circ g_{14}^{2341,3241}=\theta^{3241}\circ
h_{14}^{2341,3241},\\
\theta^{4123}\circ g_{14}^{4123,4132}  &  =\theta^{4132}\circ h_{14}%
^{4123,4132},\theta^{2314}\circ g_{14}^{2314,3214}=\theta^{3214}\circ
h_{14}^{2314,3214},\\
\theta^{2314}\circ g_{23}^{2314,2341}  &  =\theta^{2341}\circ h_{23}%
^{2314,2341},\theta^{2143}\circ g_{23}^{2143,2413}=\theta^{2413}\circ
h_{23}^{2143,2413},\\
\theta^{3142}\circ g_{23}^{3142,3412}  &  =\theta^{3412}\circ h_{23}%
^{3142,3412},\theta^{1432}\circ g_{23}^{1432,4132}=\theta^{4132}\circ
h_{23}^{1432,4132},\\
\theta^{3214}\circ g_{23}^{3214,3241}  &  =\theta^{3241}\circ h_{23}%
^{3214,3241},\theta^{1423}\circ g_{23}^{1423,4123}=\theta^{4123}\circ
h_{23}^{1423,4123},\\
\theta^{2413}\circ g_{24}^{2413,2431}  &  =\theta^{2431}\circ h_{24}%
^{2413,2431},\theta^{2134}\circ g_{24}^{2134,2314}=\theta^{2314}\circ
h_{24}^{2134,2314},\\
\theta^{4132}\circ g_{24}^{4132,4312}  &  =\theta^{4312}\circ h_{24}%
^{4132,4312},\theta^{1342}\circ g_{24}^{1342,3142}=\theta^{3142}\circ
h_{24}^{1342,3142},\\
\theta^{4213}\circ g_{24}^{4213,4231}  &  =\theta^{4231}\circ h_{24}%
^{4213,4231},\theta^{1324}\circ g_{24}^{1324,3124}=\theta^{3124}\circ
h_{24}^{1324,3124},\\
\theta^{3412}\circ g_{34}^{3412,3421}  &  =\theta^{3421}\circ h_{34}%
^{3412,3421},\theta^{3124}\circ g_{34}^{3124,3214}=\theta^{3214}\circ
h_{34}^{3124,3214},\\
\theta^{4123}\circ g_{34}^{4123,4213}  &  =\theta^{4213}\circ h_{34}%
^{4123,4213},\theta^{1243}\circ g_{34}^{1243,2143}=\theta^{2143}\circ
h_{34}^{1243,2143},\\
\theta^{4312}\circ g_{34}^{4312,4321}  &  =\theta^{4321}\circ h_{34}%
^{4312,4321},\theta^{1234}\circ g_{34}^{1234,2134}=\theta^{2134}\circ
h_{34}^{1234,2134}%
\end{align*}
are tantamount to the following conditions in terms of coefficients of the
polynomials%
\begin{align*}
a^{1342}  &  =a^{1432},a_{1}^{1342}=a_{1}^{1432},a_{2}^{1342}=a_{2}%
^{1432},a_{3}^{1342}=a_{3}^{1432},a_{4}^{1342}=a_{4}^{1432},\\
a_{12}^{1342}  &  =a_{12}^{1432},a_{13}^{1342}=a_{13}^{1432},a_{14}%
^{1342}=a_{14}^{1432},a_{23}^{1342}=a_{23}^{1432},a_{24}^{1342}=a_{24}%
^{1432},\\
a_{123}^{1342}  &  =a_{123}^{1432},a_{124}^{1342}=a_{124}^{1432}%
\end{align*}%
\begin{align*}
a^{2341}  &  =a^{2431},a_{1}^{2341}=a_{1}^{2431},a_{2}^{2341}=a_{2}%
^{2431},a_{3}^{2341}=a_{3}^{2431},a_{4}^{2341}=a_{4}^{2431},\\
a_{12}^{2341}  &  =a_{12}^{2431},a_{13}^{2341}=a_{13}^{2431},a_{14}%
^{2341}=a_{14}^{2431},a_{23}^{2341}=a_{23}^{2431},a_{24}^{2341}=a_{24}%
^{2431},\\
a_{123}^{2341}  &  =a_{123}^{2431},a_{124}^{2341}=a_{124}^{2431}%
\end{align*}%
\begin{align*}
a^{3421}  &  =a^{4321},a_{1}^{3421}=a_{1}^{4321},a_{2}^{3421}=a_{2}%
^{4321},a_{3}^{3421}=a_{3}^{4321},a_{4}^{3421}=a_{4}^{4321},\\
a_{12}^{3421}  &  =a_{12}^{4321},a_{13}^{3421}=a_{13}^{4321},a_{14}%
^{3421}=a_{14}^{4321},a_{23}^{3421}=a_{23}^{4321},a_{24}^{3421}=a_{24}%
^{4321},\\
a_{123}^{3421}  &  =a_{123}^{4321},a_{124}^{3421}=a_{124}^{4321}%
\end{align*}%
\begin{align*}
a^{2341}  &  =a^{2314},a_{1}^{2341}=a_{1}^{2314},a_{2}^{2341}=a_{2}%
^{2314},a_{3}^{2341}=a_{3}^{2314},a_{4}^{2341}=a_{4}^{2314},\\
a_{12}^{2341}  &  =a_{12}^{2314},a_{13}^{2341}=a_{13}^{2314},a_{14}%
^{2341}=a_{14}^{2314},a_{24}^{2341}=a_{24}^{2314},a_{34}^{2341}=a_{34}%
^{2314},\\
a_{124}^{2341}  &  =a_{124}^{2314},a_{134}^{2341}=a_{134}^{2314}%
\end{align*}%
\begin{align*}
a^{2413}  &  =a^{2143},a_{1}^{2413}=a_{1}^{2143},a_{2}^{2413}=a_{2}%
^{2143},a_{3}^{2413}=a_{3}^{2143},a_{4}^{2413}=a_{4}^{2143},\\
a_{12}^{2413}  &  =a_{12}^{2143},a_{13}^{2413}=a_{13}^{2143},a_{14}%
^{2413}=a_{14}^{2143},a_{24}^{2413}=a_{24}^{2143},a_{34}^{2413}=a_{34}%
^{2143},\\
a_{124}^{2413}  &  =a_{124}^{2143},a_{134}^{2413}=a_{134}^{2143}%
\end{align*}%
\begin{align*}
a^{3412}  &  =a^{3142},a_{1}^{3412}=a_{1}^{3142},a_{2}^{3412}=a_{2}%
^{3142},a_{3}^{3412}=a_{3}^{3142},a_{4}^{3412}=a_{4}^{3142},\\
a_{12}^{3412}  &  =a_{12}^{3142},a_{13}^{3412}=a_{13}^{3142},a_{14}%
^{3412}=a_{14}^{3142},a_{24}^{3412}=a_{24}^{3142},a_{34}^{3412}=a_{34}%
^{3142},\\
a_{124}^{3412}  &  =a_{124}^{3142},a_{134}^{3412}=a_{134}^{3142}%
\end{align*}%
\begin{align*}
a^{4132}  &  =a^{1432},a_{1}^{4132}=a_{1}^{1432},a_{2}^{4132}=a_{2}%
^{1432},a_{3}^{4132}=a_{3}^{1432},a_{4}^{4132}=a_{4}^{1432},\\
a_{12}^{4132}  &  =a_{12}^{1432},a_{13}^{4132}=a_{13}^{1432},a_{14}%
^{4132}=a_{14}^{1432},a_{24}^{4132}=a_{24}^{1432},a_{34}^{4132}=a_{34}%
^{1432},\\
a_{124}^{4132}  &  =a_{124}^{1432},a_{134}^{4132}=a_{134}^{1432}%
\end{align*}%
\begin{align*}
a^{3412}  &  =a^{3421},a_{1}^{3412}=a_{1}^{3421},a_{2}^{3412}=a_{2}%
^{3421},a_{3}^{3412}=a_{3}^{3421},a_{4}^{3412}=a_{4}^{3421},\\
a_{12}^{3412}  &  =a_{12}^{3421},a_{13}^{3412}=a_{13}^{3421},a_{14}%
^{3412}=a_{14}^{3421},a_{23}^{3412}=a_{23}^{3421},a_{24}^{3412}=a_{24}%
^{3421},\\
a_{123}^{3412}  &  =a_{123}^{3421},a_{124}^{3412}=a_{124}^{3421}%
\end{align*}%
\begin{align*}
a^{3124}  &  =a^{3214},a_{1}^{3124}=a_{1}^{3214},a_{2}^{3124}=a_{2}%
^{3214},a_{3}^{3124}=a_{3}^{3214},a_{4}^{3124}=a_{4}^{3214},\\
a_{12}^{3124}  &  =a_{12}^{3214},a_{13}^{3124}=a_{13}^{3214},a_{14}%
^{3124}=a_{14}^{3214},a_{23}^{3124}=a_{23}^{3214},a_{24}^{3124}=a_{24}%
^{3214},\\
a_{123}^{3124}  &  =a_{123}^{3214},a_{124}^{3124}=a_{124}^{3214}%
\end{align*}%
\begin{align*}
a^{4123}  &  =a^{4213},a_{1}^{4123}=a_{1}^{4213},a_{2}^{4123}=a_{2}%
^{4213},a_{3}^{4123}=a_{3}^{4213},a_{4}^{4123}=a_{4}^{4213},\\
a_{12}^{4123}  &  =a_{12}^{4213},a_{13}^{4123}=a_{13}^{4213},a_{14}%
^{4123}=a_{14}^{4213},a_{23}^{4123}=a_{23}^{4213},a_{24}^{4123}=a_{24}%
^{4213},\\
a_{123}^{4123}  &  =a_{123}^{4213},a_{124}^{4123}=a_{124}^{4213}%
\end{align*}%
\begin{align*}
a^{1243}  &  =a^{2143},a_{1}^{1243}=a_{1}^{2143},a_{2}^{1243}=a_{2}%
^{2143},a_{3}^{1243}=a_{3}^{2143},a_{4}^{1243}=a_{4}^{2143},\\
a_{12}^{1243}  &  =a_{12}^{2143},a_{13}^{1243}=a_{13}^{2143},a_{14}%
^{1243}=a_{14}^{2143},a_{23}^{1243}=a_{23}^{2143},a_{24}^{1243}=a_{24}%
^{2143},\\
a_{123}^{1243}  &  =a_{123}^{2143},a_{124}^{1243}=a_{124}^{2143}%
\end{align*}%
\begin{align*}
a^{4123}  &  =a^{4132},a_{1}^{4123}=a_{1}^{4132},a_{2}^{4123}=a_{2}%
^{4132},a_{3}^{4123}=a_{3}^{4132},a_{4}^{4123}=a_{4}^{4132},\\
a_{12}^{4123}  &  =a_{12}^{4132},a_{13}^{4123}=a_{13}^{4132},a_{23}%
^{4123}=a_{23}^{4132},a_{24}^{4123}=a_{24}^{4132},a_{34}^{4123}=a_{34}%
^{4132},\\
a_{123}^{4123}  &  =a_{123}^{4132},a_{234}^{4123}=a_{234}^{4132}%
\end{align*}%
\begin{align*}
a^{4231}  &  =a^{4321},a_{1}^{4231}=a_{1}^{4321},a_{2}^{4231}=a_{2}%
^{4321},a_{3}^{4231}=a_{3}^{4321},a_{4}^{4231}=a_{4}^{4321},\\
a_{12}^{4231}  &  =a_{12}^{4321},a_{13}^{4231}=a_{13}^{4321},a_{23}%
^{4231}=a_{23}^{4321},a_{24}^{4231}=a_{24}^{4321},a_{34}^{4231}=a_{34}%
^{4321},\\
a_{123}^{4231}  &  =a_{123}^{4321},a_{234}^{4231}=a_{234}^{4321}%
\end{align*}%
\begin{align*}
a^{1234}  &  =a^{1324},a_{1}^{1234}=a_{1}^{1324},a_{2}^{1234}=a_{2}%
^{1324},a_{3}^{1234}=a_{3}^{1324},a_{4}^{1234}=a_{4}^{1324},\\
a_{12}^{1234}  &  =a_{12}^{1324},a_{13}^{1234}=a_{13}^{1324},a_{23}%
^{1234}=a_{23}^{1324},a_{24}^{1234}=a_{24}^{1324},a_{34}^{1234}=a_{34}%
^{1324},\\
a_{123}^{1234}  &  =a_{123}^{1324},a_{234}^{1234}=a_{234}^{1324}%
\end{align*}%
\begin{align*}
a^{2314}  &  =a^{3214},a_{1}^{2314}=a_{1}^{3214},a_{2}^{2314}=a_{2}%
^{3214},a_{3}^{2314}=a_{3}^{3214},a_{4}^{2314}=a_{4}^{3214},\\
a_{12}^{2314}  &  =a_{12}^{3214},a_{13}^{2314}=a_{13}^{3214},a_{23}%
^{2314}=a_{23}^{3214},a_{24}^{2314}=a_{24}^{3214},a_{34}^{2314}=a_{34}%
^{3214},\\
a_{123}^{2314}  &  =a_{123}^{3214},a_{234}^{2314}=a_{234}^{3214}%
\end{align*}
which can succintly be summarized as%
\begin{align*}
a^{1234}  &  =a^{1243}=a^{1324}=a^{1342}=a^{1423}=a^{1432}=\\
a^{2134}  &  =a^{2143}=a^{2314}=a^{2341}=a^{2413}=a^{2431}=\\
a^{3124}  &  =a^{3142}=a^{3214}=a^{3241}=a^{3412}=a^{3421}=\\
a^{4123}  &  =a^{4132}=a^{4213}=a^{4231}=a^{4312}=a^{4321}%
\end{align*}%
\begin{align*}
a_{1}^{1234}  &  =a_{1}^{1243}=a_{1}^{1324}=a_{1}^{1342}=a_{1}^{1423}%
=a_{1}^{1432}=\\
a_{1}^{2134}  &  =a_{1}^{2143}=a_{1}^{2314}=a_{1}^{2341}=a_{1}^{2413}%
=a_{1}^{2431}=\\
a_{1}^{3124}  &  =a_{1}^{3142}=a_{1}^{3214}=a_{1}^{3241}=a_{1}^{3412}%
=a_{1}^{3421}=\\
a_{1}^{4123}  &  =a_{1}^{4132}=a_{1}^{4213}=a_{1}^{4231}=a_{1}^{4312}%
=a_{1}^{4321}%
\end{align*}%
\begin{align*}
a_{2}^{1234}  &  =a_{2}^{1243}=a_{2}^{1324}=a_{2}^{1342}=a_{2}^{1423}%
=a_{2}^{1432}=\\
a_{2}^{2134}  &  =a_{2}^{2143}=a_{2}^{2314}=a_{2}^{2341}=a_{2}^{2413}%
=a_{2}^{2431}=\\
a_{2}^{3124}  &  =a_{2}^{3142}=a_{2}^{3214}=a_{2}^{3241}=a_{2}^{3412}%
=a_{2}^{3421}=\\
a_{2}^{4123}  &  =a_{2}^{4132}=a_{2}^{4213}=a_{2}^{4231}=a_{2}^{4312}%
=a_{2}^{4321}%
\end{align*}%
\begin{align*}
a_{3}^{1234}  &  =a_{3}^{1243}=a_{3}^{1324}=a_{3}^{1342}=a_{3}^{1423}%
=a_{3}^{1432}=\\
a_{3}^{2134}  &  =a_{3}^{2143}=a_{3}^{2314}=a_{3}^{2341}=a_{3}^{2413}%
=a_{3}^{2431}=\\
a_{3}^{3124}  &  =a_{3}^{3142}=a_{3}^{3214}=a_{3}^{3241}=a_{3}^{3412}%
=a_{3}^{3421}=\\
a_{3}^{4123}  &  =a_{3}^{4132}=a_{3}^{4213}=a_{3}^{4231}=a_{3}^{4312}%
=a_{3}^{4321}%
\end{align*}%
\begin{align*}
a_{4}^{1234}  &  =a_{4}^{1243}=a_{4}^{1324}=a_{4}^{1342}=a_{4}^{1423}%
=a_{4}^{1432}=\\
a_{4}^{2134}  &  =a_{4}^{2143}=a_{4}^{2314}=a_{4}^{2341}=a_{4}^{2413}%
=a_{4}^{2431}=\\
a_{4}^{3124}  &  =a_{4}^{3142}=a_{4}^{3214}=a_{4}^{3241}=a_{4}^{3412}%
=a_{4}^{3421}=\\
a_{4}^{4123}  &  =a_{4}^{4132}=a_{4}^{4213}=a_{4}^{4231}=a_{4}^{4312}%
=a_{4}^{4321}%
\end{align*}%
\begin{align*}
a_{12}^{1234}  &  =a_{12}^{1243}=a_{12}^{1324}=a_{12}^{1342}=a_{12}%
^{1423}=a_{12}^{1432}=\\
a_{12}^{3124}  &  =a_{12}^{3142}=a_{12}^{3412}=a_{12}^{4123}=a_{12}%
^{4132}=a_{12}^{4312}%
\end{align*}%
\begin{align*}
a_{12}^{2134}  &  =a_{12}^{2143}=a_{12}^{2314}=a_{12}^{2341}=a_{12}%
^{2413}=a_{12}^{2431}=\\
a_{12}^{3214}  &  =a_{12}^{3241}=a_{12}^{3421}=a_{12}^{4213}=a_{12}%
^{4231}=a_{12}^{4321}%
\end{align*}%
\begin{align*}
a_{13}^{1342}  &  =a_{13}^{1324}=a_{13}^{1432}=a_{13}^{1423}=a_{13}%
^{1234}=a_{13}^{1243}=\\
a_{13}^{4132}  &  =a_{13}^{4123}=a_{13}^{4213}=a_{13}^{2134}=a_{13}%
^{2143}=a_{13}^{2413}%
\end{align*}%
\begin{align*}
a_{13}^{3142}  &  =a_{13}^{3124}=a_{13}^{3412}=a_{13}^{3421}=a_{13}%
^{3214}=a_{13}^{3241}=\\
a_{13}^{4312}  &  =a_{13}^{4321}=a_{13}^{4231}=a_{13}^{2314}=a_{13}%
^{2341}=a_{13}^{2431}%
\end{align*}%
\begin{align*}
a_{14}^{1423}  &  =a_{14}^{1432}=a_{14}^{1243}=a_{14}^{1234}=a_{14}%
^{1342}=a_{14}^{1324}=\\
a_{14}^{2143}  &  =a_{14}^{2134}=a_{14}^{2314}=a_{14}^{3142}=a_{14}%
^{3124}=a_{14}^{3214}%
\end{align*}%
\begin{align*}
a_{14}^{4123}  &  =a_{14}^{4132}=a_{14}^{4213}=a_{14}^{4231}=a_{14}%
^{4312}=a_{14}^{4321}=\\
a_{14}^{2413}  &  =a_{14}^{2431}=a_{14}^{2341}=a_{14}^{3412}=a_{14}%
^{3421}=a_{14}^{3241}%
\end{align*}%
\begin{align*}
a_{23}^{2314}  &  =a_{23}^{2341}=a_{23}^{2134}=a_{23}^{2143}=a_{23}%
^{2431}=a_{23}^{2413}=\\
a_{23}^{1234}  &  =a_{23}^{1243}=a_{23}^{1423}=a_{23}^{4231}=a_{23}%
^{4213}=a_{23}^{4123}%
\end{align*}%
\begin{align*}
a_{23}^{3214}  &  =a_{23}^{3241}=a_{23}^{3124}=a_{23}^{3142}=a_{23}%
^{3421}=a_{23}^{3412}=\\
a_{23}^{1324}  &  =a_{23}^{1342}=a_{23}^{1432}=a_{23}^{4321}=a_{23}%
^{4312}=a_{23}^{4132}%
\end{align*}%
\begin{align*}
a_{24}^{2413}  &  =a_{24}^{2431}=a_{24}^{2143}=a_{24}^{2134}=a_{24}%
^{2341}=a_{24}^{2314}=\\
a_{24}^{1243}  &  =a_{24}^{1234}=a_{24}^{1324}=a_{24}^{3241}=a_{24}%
^{3214}=a_{24}^{3124}%
\end{align*}%
\begin{align*}
a_{24}^{4213}  &  =a_{24}^{4231}=a_{24}^{4123}=a_{24}^{4132}=a_{24}%
^{4321}=a_{24}^{4312}=\\
a_{24}^{1423}  &  =a_{24}^{1432}=a_{24}^{1342}=a_{24}^{3421}=a_{24}%
^{3412}=a_{24}^{3142}%
\end{align*}%
\begin{align*}
a_{34}^{3412}  &  =a_{34}^{3421}=a_{34}^{3142}=a_{34}^{3124}=a_{34}%
^{3241}=a_{34}^{3214}=\\
a_{34}^{1342}  &  =a_{34}^{1324}=a_{34}^{1234}=a_{34}^{2341}=a_{34}%
^{2314}=a_{34}^{2134}%
\end{align*}%
\begin{align*}
a_{34}^{4312}  &  =a_{34}^{4321}=a_{34}^{4132}=a_{34}^{4123}=a_{34}%
^{4231}=a_{34}^{4213}=\\
a_{34}^{1432}  &  =a_{34}^{1423}=a_{34}^{1243}=a_{34}^{2431}=a_{34}%
^{2413}=a_{34}^{2143}%
\end{align*}%
\[
a_{123}^{1234}=a_{123}^{1243}=a_{123}^{1423}=a_{123}^{4123}%
\]%
\[
a_{123}^{1324}=a_{123}^{1342}=a_{123}^{1432}=a_{123}^{4132}%
\]%
\[
a_{123}^{2134}=a_{123}^{2143}=a_{123}^{2413}=a_{123}^{4213}%
\]%
\[
a_{123}^{2314}=a_{123}^{2341}=a_{123}^{2431}=a_{123}^{4231}%
\]%
\[
a_{123}^{3124}=a_{123}^{3142}=a_{123}^{3412}=a_{123}^{4312}%
\]%
\[
a_{123}^{3214}=a_{123}^{3241}=a_{123}^{3421}=a_{123}^{4321}%
\]%
\[
a_{124}^{1243}=a_{124}^{1234}=a_{124}^{1324}=a_{124}^{3124}%
\]%
\[
a_{124}^{1423}=a_{124}^{1432}=a_{124}^{1342}=a_{124}^{3142}%
\]%
\[
a_{124}^{2143}=a_{124}^{2134}=a_{124}^{2314}=a_{124}^{3214}%
\]%
\[
a_{124}^{2413}=a_{124}^{2431}=a_{124}^{2341}=a_{124}^{3241}%
\]%
\[
a_{124}^{4123}=a_{124}^{4132}=a_{124}^{4312}=a_{124}^{3412}%
\]%
\[
a_{124}^{4213}=a_{124}^{4231}=a_{124}^{4321}=a_{124}^{3421}%
\]%
\[
a_{134}^{1342}=a_{134}^{1324}=a_{134}^{1234}=a_{134}^{2134}%
\]%
\[
a_{134}^{1432}=a_{134}^{1423}=a_{134}^{1243}=a_{134}^{2143}%
\]%
\[
a_{134}^{3142}=a_{134}^{3124}=a_{134}^{3214}=a_{134}^{2314}%
\]%
\[
a_{134}^{3412}=a_{134}^{3421}=a_{134}^{3241}=a_{134}^{2341}%
\]%
\[
a_{134}^{4132}=a_{134}^{4123}=a_{134}^{4213}=a_{134}^{2413}%
\]%
\[
a_{134}^{4312}=a_{134}^{4321}=a_{134}^{4231}=a_{134}^{2431}%
\]%
\[
a_{234}^{2341}=a_{234}^{2314}=a_{234}^{2134}=a_{234}^{1234}%
\]%
\[
a_{234}^{2431}=a_{234}^{2413}=a_{234}^{2143}=a_{234}^{1243}%
\]%
\[
a_{234}^{3241}=a_{234}^{3214}=a_{234}^{3124}=a_{234}^{1324}%
\]%
\[
a_{234}^{3421}=a_{234}^{3412}=a_{234}^{3142}=a_{234}^{1342}%
\]%
\[
a_{234}^{4231}=a_{234}^{4213}=a_{234}^{4123}=a_{234}^{1423}%
\]%
\[
a_{234}^{4321}=a_{234}^{4312}=a_{234}^{4132}=a_{234}^{1432}%
\]
Therefore there exists a unique mapping $\theta:P\rightarrow\mathbb{R}$ such
that%
\begin{align*}
\theta^{1234}  &  =\theta\circ f_{1234},\theta^{1243}=\theta\circ
f_{1243},\theta^{1324}=\theta\circ f_{1324},\theta^{1342}=\theta\circ
f_{1342},\theta^{1423}=\theta\circ f_{1423},\theta^{1432}=\theta\circ
f_{1432},\\
\theta^{2134}  &  =\theta\circ f_{2134},\theta^{2143}=\theta\circ
f_{2143},\theta^{2314}=\theta\circ f_{2314},\theta^{2341}=\theta\circ
f_{2341},\theta^{2413}=\theta\circ f_{2413},\theta^{2431}=\theta\circ
f_{2431},\\
\theta^{3124}  &  =\theta\circ f_{3124},\theta^{3142}=\theta\circ
f_{3142},\theta^{3214}=\theta\circ f_{3214},\theta^{3241}=\theta\circ
f_{3241},\theta^{3412}=\theta\circ f_{3412},\theta^{3421}=\theta\circ
f_{3421},\\
\theta^{4123}  &  =\theta\circ f_{4123},\theta^{4132}=\theta\circ
f_{4132},\theta^{4213}=\theta\circ f_{4213},\theta^{4231}=\theta\circ
f_{4231},\theta^{4312}=\theta\circ f_{4312},\theta^{4321}=\theta\circ f_{4321}%
\end{align*}
The proof is now complete.
\end{proof}

\begin{remark}
For our convenience we display the positions $5-30$ of $f_{ijkl}\left(
d_{1},d_{2},d_{3},d_{4}\right)  $ as tables as follows:%
\begin{equation}%
\begin{array}
[c]{ccccccc}
& 5/d_{1}d_{2} & 6/d_{1}d_{3} & 7/d_{1}d_{4} & 8/d_{2}d_{3} & 9/d_{2}d_{4} &
10/d_{3}d_{4}\\
1234 & 0 & 0 & 0 & 0 & 0 & 0\\
1243 & 0 & 0 & 0 & 0 & 0 & d_{3}d_{4}\\
1324 & 0 & 0 & 0 & d_{2}d_{3} & 0 & 0\\
1342 & 0 & 0 & 0 & d_{2}d_{3} & d_{2}d_{4} & 0\\
1423 & 0 & 0 & 0 & 0 & d_{2}d_{4} & d_{3}d_{4}\\
1432 & 0 & 0 & 0 & d_{2}d_{3} & d_{2}d_{4} & d_{3}d_{4}\\
2134 & d_{1}d_{2} & 0 & 0 & 0 & 0 & 0\\
2143 & d_{1}d_{2} & 0 & 0 & 0 & 0 & d_{3}d_{4}\\
2314 & d_{1}d_{2} & d_{1}d_{3} & 0 & 0 & 0 & 0\\
2341 & d_{1}d_{2} & d_{1}d_{3} & d_{1}d_{4} & 0 & 0 & 0\\
2413 & d_{1}d_{2} & 0 & d_{1}d_{4} & 0 & 0 & d_{3}d_{4}\\
2431 & d_{1}d_{2} & d_{1}d_{3} & d_{1}d_{4} & 0 & 0 & d_{3}d_{4}\\
3124 & 0 & d_{1}d_{3} & 0 & d_{2}d_{3} & 0 & 0\\
3142 & 0 & d_{1}d_{3} & 0 & d_{2}d_{3} & d_{2}d_{4} & 0\\
3214 & d_{1}d_{2} & d_{1}d_{3} & 0 & d_{2}d_{3} & 0 & 0\\
3241 & d_{1}d_{2} & d_{1}d_{3} & d_{1}d_{4} & d_{2}d_{3} & 0 & 0\\
3412 & 0 & d_{1}d_{3} & d_{1}d_{4} & d_{2}d_{3} & d_{2}d_{4} & 0\\
3421 & d_{1}d_{2} & d_{1}d_{3} & d_{1}d_{4} & d_{2}d_{3} & d_{2}d_{4} & 0\\
4123 & 0 & 0 & d_{1}d_{4} & 0 & d_{2}d_{4} & d_{3}d_{4}\\
4132 & 0 & 0 & d_{1}d_{4} & d_{2}d_{3} & d_{2}d_{4} & d_{3}d_{4}\\
4213 & d_{1}d_{2} & 0 & d_{1}d_{4} & 0 & d_{2}d_{4} & d_{3}d_{4}\\
4231 & d_{1}d_{2} & d_{1}d_{3} & d_{1}d_{4} & 0 & d_{2}d_{4} & d_{3}d_{4}\\
4312 & 0 & d_{1}d_{3} & d_{1}d_{4} & d_{2}d_{3} & d_{2}d_{4} & d_{3}d_{4}\\
4321 & d_{1}d_{2} & d_{1}d_{3} & d_{1}d_{4} & d_{2}d_{3} & d_{2}d_{4} &
d_{3}d_{4}%
\end{array}
\label{Figure 1}%
\end{equation}%
\begin{equation}%
\begin{array}
[c]{ccccc}
& 11-15/d_{1}d_{2}d_{3} & 16-20/d_{1}d_{2}d_{4} & 21-25/d_{1}d_{3}d_{4} &
26-30/d_{2}d_{3}d_{4}\\
1234 &  &  &  & \\
1243 &  &  & 21 & 26\\
1324 & 11 &  &  & 27\\
1342 & 11 & 16 &  & 28\\
1423 &  & 16 & 21 & 29\\
1432 & 11 & 16 & 21 & 30\\
2134 & 12 & 17 &  & \\
2143 & 12 & 17 & 21 & 26\\
2314 & 13 & 17 & 22 & \\
2341 & 13 & 18 & 23 & \\
2413 & 12 & 18 & 24 & 26\\
2431 & 13 & 18 & 25 & 26\\
3124 & 14 &  & 22 & 27\\
3142 & 14 & 16 & 22 & 28\\
3214 & 15 & 17 & 22 & 27\\
3241 & 15 & 18 & 23 & 27\\
3412 & 14 & 19 & 23 & 28\\
3421 & 15 & 20 & 23 & 28\\
4123 &  & 19 & 24 & 29\\
4132 & 11 & 19 & 24 & 30\\
4213 & 12 & 20 & 24 & 29\\
4231 & 13 & 20 & 25 & 29\\
4312 & 14 & 19 & 25 & 30\\
4321 & 15 & 20 & 25 & 30
\end{array}
\label{Figure 2}%
\end{equation}

\end{remark}

\begin{corollary}
\label{c3.1}Let $M$\ be a microlinear space with mappings%
\begin{align*}
\gamma_{1234}  &  :Q^{1234}\rightarrow M,\gamma_{1243}:Q^{1243}\rightarrow
M,\gamma_{1324}:Q^{1324}\rightarrow M,\\
\gamma_{1342}  &  :Q^{1342}\rightarrow M,\gamma_{1423}:Q^{1423}\rightarrow
M,\gamma_{1432}:Q^{1432}\rightarrow M,\\
\gamma_{2134}  &  :Q^{2134}\rightarrow M,\gamma_{2143}:Q^{2143}\rightarrow
M,\gamma_{2314}:Q^{2314}\rightarrow M,\\
\gamma_{2341}  &  :Q^{2341}\rightarrow M,\gamma_{2413}:Q^{2413}\rightarrow
M,\gamma_{2431}:Q^{2431}\rightarrow M,\\
\gamma_{3124}  &  :Q^{3124}\rightarrow M,\gamma_{3142}:Q^{3142}\rightarrow
M,\gamma_{3214}:Q^{3214}\rightarrow M,\\
\gamma_{3241}  &  :Q^{3241}\rightarrow M,\gamma_{3412}:Q^{3412}\rightarrow
M,\gamma_{3421}:Q^{3421}\rightarrow M,\\
\gamma_{4123}  &  :Q^{4123}\rightarrow M,\gamma_{4132}:Q^{4132}\rightarrow
M,\gamma_{4213}:Q^{4213}\rightarrow M,\\
\gamma_{4231}  &  :Q^{4231}\rightarrow M,\gamma_{4312}:Q^{4312}\rightarrow
M,\gamma_{4321}:Q^{4321}\rightarrow M
\end{align*}
abiding by%
\begin{align*}
\gamma_{1234}\circ g_{12}^{1234,1243}  &  =\gamma_{1243}\circ h_{12}%
^{1234,1243},\gamma_{1342}\circ g_{12}^{1342,1432}=\gamma_{1432}\circ
h_{12}^{1342,1432},\\
\gamma_{2341}\circ g_{12}^{2341,2431}  &  =\gamma_{2431}\circ h_{12}%
^{2341,2431},\gamma_{3421}\circ g_{12}^{3421,4321}=\gamma_{4321}\circ
h_{12}^{3421,4321},\\
\gamma_{2134}\circ g_{12}^{2134,2143}  &  =\gamma_{2143}\circ h_{12}%
^{2134,2143},\gamma_{3412}\circ g_{12}^{3412,4312}=\gamma_{4312}\circ
h_{12}^{3412,4312},\\
\gamma_{1324}\circ g_{13}^{1324,1342}  &  =\gamma_{1342}\circ h_{13}%
^{1324,1342},\gamma_{1243}\circ g_{13}^{1243,1423}=\gamma_{1423}\circ
h_{13}^{1243,1423},\\
\gamma_{3241}\circ g_{13}^{3241,3421}  &  =\gamma_{3421}\circ h_{13}%
^{3241,3421},\gamma_{2431}\circ g_{13}^{2431,4231}=\gamma_{4231}\circ
h_{13}^{2431,4231},\\
\gamma_{3124}\circ g_{13}^{3124,3142}  &  =\gamma_{3142}\circ h_{13}%
^{3124,3142},\gamma_{2413}\circ g_{13}^{2413,4213}=\gamma_{4213}\circ
h_{13}^{2413,4213},\\
\gamma_{1423}\circ g_{14}^{1423,1432}  &  =\gamma_{1432}\circ h_{14}%
^{1423,1432},\gamma_{1234}\circ g_{14}^{1234,1324}=\gamma_{1324}\circ
h_{14}^{1234,1324},\\
\gamma_{4231}\circ g_{14}^{4231,4321}  &  =\gamma_{4321}\circ h_{14}%
^{4231,4321},\gamma_{2341}\circ g_{14}^{2341,3241}=\gamma_{3241}\circ
h_{14}^{2341,3241},\\
\gamma_{4123}\circ g_{14}^{4123,4132}  &  =\gamma_{4132}\circ h_{14}%
^{4123,4132},\gamma_{2314}\circ g_{14}^{2314,3214}=\gamma_{3214}\circ
h_{14}^{2314,3214},\\
\gamma_{2314}\circ g_{23}^{2314,2341}  &  =\gamma_{2341}\circ h_{23}%
^{2314,2341},\gamma_{2143}\circ g_{23}^{2143,2413}=\gamma_{2413}\circ
h_{23}^{2143,2413},\\
\gamma_{3142}\circ g_{23}^{3142,3412}  &  =\gamma_{3412}\circ h_{23}%
^{3142,3412},\gamma_{1432}\circ g_{23}^{1432,4132}=\gamma_{4132}\circ
h_{23}^{1432,4132},\\
\gamma_{3214}\circ g_{23}^{3214,3241}  &  =\gamma_{3241}\circ h_{23}%
^{3214,3241},\gamma_{1423}\circ g_{23}^{1423,4123}=\gamma_{4123}\circ
h_{23}^{1423,4123},\\
\gamma_{2413}\circ g_{24}^{2413,2431}  &  =\gamma_{2431}\circ h_{24}%
^{2413,2431},\gamma_{2134}\circ g_{24}^{2134,2314}=\gamma_{2314}\circ
h_{24}^{2134,2314},\\
\gamma_{4132}\circ g_{24}^{4132,4312}  &  =\gamma_{4312}\circ h_{24}%
^{4132,4312},\gamma_{1342}\circ g_{24}^{1342,3142}=\gamma_{3142}\circ
h_{24}^{1342,3142},\\
\gamma_{4213}\circ g_{24}^{4213,4231}  &  =\gamma_{4231}\circ h_{24}%
^{4213,4231},\gamma_{1324}\circ g_{24}^{1324,3124}=\gamma_{3124}\circ
h_{24}^{1324,3124},\\
\gamma_{3412}\circ g_{34}^{3412,3421}  &  =\gamma_{3421}\circ h_{34}%
^{3412,3421},\gamma_{3124}\circ g_{34}^{3124,3214}=\gamma_{3214}\circ
h_{34}^{3124,3214},\\
\gamma_{4123}\circ g_{34}^{4123,4213}  &  =\gamma_{4213}\circ h_{34}%
^{4123,4213},\gamma_{1243}\circ g_{34}^{1243,2143}=\gamma_{2143}\circ
h_{34}^{1243,2143},\\
\gamma_{4312}\circ g_{34}^{4312,4321}  &  =\gamma_{4321}\circ h_{34}%
^{4312,4321},\gamma_{1234}\circ g_{34}^{1234,2134}=\gamma_{2134}\circ
h_{34}^{1234,2134}%
\end{align*}
Then there exists a unique mapping%
\[
\mathfrak{m}:P\rightarrow M
\]
such that
\begin{align*}
\mathfrak{m}\circ f_{1234}  &  =\gamma_{1234},\mathfrak{m}\circ f_{1243}%
=\gamma_{1243},\mathfrak{m}\circ f_{1324}=\gamma_{1324},\mathfrak{m}\circ
f_{1342}=\gamma_{1342},\mathfrak{m}\circ f_{1423}=\gamma_{1423},\mathfrak{m}%
\circ f_{1432}=\gamma_{1432},\\
\mathfrak{m}\circ f_{2134}  &  =\gamma_{2134},\mathfrak{m}\circ f_{2143}%
=\gamma_{2143},\mathfrak{m}\circ f_{2314}=\gamma_{2314},\mathfrak{m}\circ
f_{2341}=\gamma_{2341},\mathfrak{m}\circ f_{2413}=\gamma_{2413},\mathfrak{m}%
\circ f_{2431}=\gamma_{2431},\\
\mathfrak{m}\circ f_{3124}  &  =\gamma_{3124},\mathfrak{m}\circ f_{3142}%
=\gamma_{3142},\mathfrak{m}\circ f_{3214}=\gamma_{3214},\mathfrak{m}\circ
f_{3241}=\gamma_{3241},\mathfrak{m}\circ f_{3412}=\gamma_{3412},\mathfrak{m}%
\circ f_{3421}=\gamma_{3421},\\
\mathfrak{m}\circ f_{4123}  &  =\gamma_{4123},\mathfrak{m}\circ f_{4132}%
=\gamma_{4132},\mathfrak{m}\circ f_{4213}=\gamma_{4213},\mathfrak{m}\circ
f_{4231}=\gamma_{4231},\mathfrak{m}\circ f_{4312}=\gamma_{4312},\mathfrak{m}%
\circ f_{4321}=\gamma_{4321}%
\end{align*}

\end{corollary}

\begin{theorem}
\label{t3.2}Let $M$\ be a microlinear space. Let%
\[%
\begin{array}
[c]{c}%
\gamma_{1234},\gamma_{1243},\gamma_{1324},\gamma_{1342},\gamma_{1423}%
,\gamma_{1432},\gamma_{2134},\gamma_{2143},\gamma_{2314},\gamma_{2341}%
,\gamma_{2413},\gamma_{2431},\\
\gamma_{3124},\gamma_{3142},\gamma_{3214},\gamma_{3241},\gamma_{3124}%
,\gamma_{3142},\gamma_{4123},\gamma_{4132},\gamma_{4213},\gamma_{4231}%
,\gamma_{4312},\gamma_{4321}%
\end{array}
:D^{4}\rightarrow M
\]
with%
\begin{align*}
\gamma_{1234}  &  \mid D^{4}\left\{  \left(  3,4\right)  \right\}
=\gamma_{1243}\mid D^{4}\left\{  \left(  3,4\right)  \right\}  ,\gamma
_{1342}\mid D^{4}\left\{  \left(  3,4\right)  \right\}  =\gamma_{1432}\mid
D^{4}\left\{  \left(  3,4\right)  \right\}  ,\\
\gamma_{2341}  &  \mid D^{4}\left\{  \left(  3,4\right)  \right\}
=\gamma_{2431}\mid D^{4}\left\{  \left(  3,4\right)  \right\}  ,\gamma
_{3421}\mid D^{4}\left\{  \left(  3,4\right)  \right\}  =\gamma_{4321}\mid
D^{4}\left\{  \left(  3,4\right)  \right\}  ,\\
\gamma_{2143}  &  \mid D^{4}\left\{  \left(  3,4\right)  \right\}
=\gamma_{2134}\mid D^{4}\left\{  \left(  3,4\right)  \right\}  ,\gamma
_{4312}\mid D^{4}\left\{  \left(  3,4\right)  \right\}  =\gamma_{3412}\mid
D^{4}\left\{  \left(  3,4\right)  \right\}  ,\\
\gamma_{1342}  &  \mid D^{4}\left\{  \left(  2,4\right)  \right\}
=\gamma_{1324}\mid D^{4}\left\{  \left(  2,4\right)  \right\}  ,\gamma
_{1423}\mid D^{4}\left\{  \left(  2,4\right)  \right\}  =\gamma_{1243}\mid
D^{4}\left\{  \left(  2,4\right)  \right\}  ,\\
\gamma_{3421}  &  \mid D^{4}\left\{  \left(  2,4\right)  \right\}
=\gamma_{3241}\mid D^{4}\left\{  \left(  2,4\right)  \right\}  ,\gamma
_{4231}\mid D^{4}\left\{  \left(  2,4\right)  \right\}  =\gamma_{2431}\mid
D^{4}\left\{  \left(  2,4\right)  \right\}  ,\\
\gamma_{3124}  &  \mid D^{4}\left\{  \left(  2,4\right)  \right\}
=\gamma_{3142}\mid D^{4}\left\{  \left(  2,4\right)  \right\}  ,\gamma
_{2413}\mid D^{4}\left\{  \left(  2,4\right)  \right\}  =\gamma_{4213}\mid
D^{4}\left\{  \left(  2,4\right)  \right\}  ,\\
\gamma_{1423}  &  \mid D^{4}\left\{  \left(  2,3\right)  \right\}
=\gamma_{1432}\mid D^{4}\left\{  \left(  2,3\right)  \right\}  ,\gamma
_{1234}\mid D^{4}\left\{  \left(  2,3\right)  \right\}  =\gamma_{1324}\mid
D^{4}\left\{  \left(  2,3\right)  \right\}  ,\\
\gamma_{4231}  &  \mid D^{4}\left\{  \left(  2,3\right)  \right\}
=\gamma_{4321}\mid D^{4}\left\{  \left(  2,3\right)  \right\}  ,\gamma
_{2341}\mid D^{4}\left\{  \left(  2,3\right)  \right\}  =\gamma_{3241}\mid
D^{4}\left\{  \left(  2,3\right)  \right\}  ,\\
\gamma_{4132}  &  \mid D^{4}\left\{  \left(  2,3\right)  \right\}
=\gamma_{4123}\mid D^{4}\left\{  \left(  2,3\right)  \right\}  ,\gamma
_{3214}\mid D^{4}\left\{  \left(  2,3\right)  \right\}  =\gamma_{2314}\mid
D^{4}\left\{  \left(  2,3\right)  \right\}  ,\\
\gamma_{2314}  &  \mid D^{4}\left\{  \left(  1,4\right)  \right\}
=\gamma_{2341}\mid D^{4}\left\{  \left(  1,4\right)  \right\}  ,\gamma
_{2143}\mid D^{4}\left\{  \left(  1,4\right)  \right\}  =\gamma_{2413}\mid
D^{4}\left\{  \left(  1,4\right)  \right\}  ,\\
\gamma_{3142}  &  \mid D^{4}\left\{  \left(  1,4\right)  \right\}
=\gamma_{3412}\mid D^{4}\left\{  \left(  1,4\right)  \right\}  ,\gamma
_{1432}\mid D^{4}\left\{  \left(  1,4\right)  \right\}  =\gamma_{4132}\mid
D^{4}\left\{  \left(  1,4\right)  \right\}  ,\\
\gamma_{3241}  &  \mid D^{4}\left\{  \left(  1,4\right)  \right\}
=\gamma_{3214}\mid D^{4}\left\{  \left(  1,4\right)  \right\}  ,\gamma
_{4123}\mid D^{4}\left\{  \left(  1,4\right)  \right\}  =\gamma_{1423}\mid
D^{4}\left\{  \left(  1,4\right)  \right\}  ,\\
\gamma_{2431}  &  \mid D^{4}\left\{  \left(  1,3\right)  \right\}
=\gamma_{2413}\mid D^{4}\left\{  \left(  1,3\right)  \right\}  ,\gamma
_{2314}\mid D^{4}\left\{  \left(  1,3\right)  \right\}  =\gamma_{2134}\mid
D^{4}\left\{  \left(  1,3\right)  \right\}  ,\\
\gamma_{4312}  &  \mid D^{4}\left\{  \left(  1,3\right)  \right\}
=\gamma_{4132}\mid D^{4}\left\{  \left(  1,3\right)  \right\}  ,\gamma
_{3142}\mid D^{4}\left\{  \left(  1,3\right)  \right\}  =\gamma_{1342}\mid
D^{4}\left\{  \left(  1,3\right)  \right\}  ,\\
\gamma_{4213}  &  \mid D^{4}\left\{  \left(  1,3\right)  \right\}
=\gamma_{4231}\mid D^{4}\left\{  \left(  1,3\right)  \right\}  ,\gamma
_{1324}\mid D^{4}\left\{  \left(  1,3\right)  \right\}  =\gamma_{3124}\mid
D^{4}\left\{  \left(  1,3\right)  \right\}  ,\\
\gamma_{3412}  &  \mid D^{4}\left\{  \left(  1,2\right)  \right\}
=\gamma_{3421}\mid D^{4}\left\{  \left(  1,2\right)  \right\}  ,\gamma
_{3124}\mid D^{4}\left\{  \left(  1,2\right)  \right\}  =\gamma_{3214}\mid
D^{4}\left\{  \left(  1,2\right)  \right\}  ,\\
\gamma_{4123}  &  \mid D^{4}\left\{  \left(  1,2\right)  \right\}
=\gamma_{4213}\mid D^{4}\left\{  \left(  1,2\right)  \right\}  ,\gamma
_{1243}\mid D^{4}\left\{  \left(  1,2\right)  \right\}  =\gamma_{2143}\mid
D^{4}\left\{  \left(  1,2\right)  \right\}  ,\\
\gamma_{4321}  &  \mid D^{4}\left\{  \left(  1,2\right)  \right\}
=\gamma_{4312}\mid D^{4}\left\{  \left(  1,2\right)  \right\}  ,\gamma
_{2134}\mid D^{4}\left\{  \left(  1,2\right)  \right\}  =\gamma_{1234}\mid
D^{4}\left\{  \left(  1,2\right)  \right\}
\end{align*}
Then we have%
\begin{align}
&  \left(  \left(  \left(  \gamma_{1234}\overset{\cdot}{\underset{12}{-}%
}\gamma_{1243}\right)  \overset{\cdot}{\underset{1}{-}}\left(  \gamma
_{1342}\overset{\cdot}{\underset{12}{-}}\gamma_{1432}\right)  \right)
\overset{\cdot}{-}\left(  \left(  \gamma_{2341}\overset{\cdot}{\underset
{12}{-}}\gamma_{2431}\right)  \overset{\cdot}{\underset{1}{-}}\left(
\gamma_{3421}\overset{\cdot}{\underset{12}{-}}\gamma_{4321}\right)  \right)
\right)  +\nonumber\\
&  \left(  \left(  \left(  \gamma_{1342}\overset{\cdot}{\underset{13}{-}%
}\gamma_{1324}\right)  \overset{\cdot}{\underset{1}{-}}\left(  \gamma
_{1423}\overset{\cdot}{\underset{13}{-}}\gamma_{1243}\right)  \right)
\overset{\cdot}{-}\left(  \left(  \gamma_{3421}\overset{\cdot}{\underset
{13}{-}}\gamma_{3241}\right)  \overset{\cdot}{\underset{1}{-}}\left(
\gamma_{4231}\overset{\cdot}{\underset{13}{-}}\gamma_{2431}\right)  \right)
\right)  +\nonumber\\
&  \left(  \left(  \left(  \gamma_{1423}\overset{\cdot}{\underset{14}{-}%
}\gamma_{1432}\right)  \overset{\cdot}{\underset{1}{-}}\left(  \gamma
_{1234}\overset{\cdot}{\underset{14}{-}}\gamma_{1324}\right)  \right)
\overset{\cdot}{-}\left(  \left(  \gamma_{4231}\overset{\cdot}{\underset
{14}{-}}\gamma_{4321}\right)  \overset{\cdot}{\underset{1}{-}}\left(
\gamma_{2341}\overset{\cdot}{\underset{14}{-}}\gamma_{3241}\right)  \right)
\right)  +\nonumber\\
&  \left(  \left(  \left(  \gamma_{2143}\overset{\cdot}{\underset{21}{-}%
}\gamma_{2134}\right)  \overset{\cdot}{\underset{1}{-}}\left(  \gamma
_{2431}\overset{\cdot}{\underset{21}{-}}\gamma_{2341}\right)  \right)
\overset{\cdot}{-}\left(  \left(  \gamma_{1432}\overset{\cdot}{\underset
{21}{-}}\gamma_{1342}\right)  \overset{\cdot}{\underset{1}{-}}\left(
\gamma_{4312}\overset{\cdot}{\underset{21}{-}}\gamma_{3412}\right)  \right)
\right)  +\nonumber\\
&  \left(  \left(  \left(  \gamma_{2314}\overset{\cdot}{\underset{23}{-}%
}\gamma_{2341}\right)  \overset{\cdot}{\underset{1}{-}}\left(  \gamma
_{2143}\overset{\cdot}{\underset{23}{-}}\gamma_{2413}\right)  \right)
\overset{\cdot}{-}\left(  \left(  \gamma_{3142}\overset{\cdot}{\underset
{23}{-}}\gamma_{3412}\right)  \overset{\cdot}{\underset{1}{-}}\left(
\gamma_{1432}\overset{\cdot}{\underset{23}{-}}\gamma_{4132}\right)  \right)
\right)  +\nonumber\\
&  \left(  \left(  \left(  \gamma_{2431}\overset{\cdot}{\underset{24}{-}%
}\gamma_{2413}\right)  \overset{\cdot}{\underset{1}{-}}\left(  \gamma
_{2314}\overset{\cdot}{\underset{24}{-}}\gamma_{2134}\right)  \right)
\overset{\cdot}{-}\left(  \left(  \gamma_{4312}\overset{\cdot}{\underset
{24}{-}}\gamma_{4132}\right)  \overset{\cdot}{\underset{1}{-}}\left(
\gamma_{3142}\overset{\cdot}{\underset{24}{-}}\gamma_{1342}\right)  \right)
\right)  +\nonumber\\
&  \left(  \left(  \left(  \gamma_{3124}\overset{\cdot}{\underset{31}{-}%
}\gamma_{3142}\right)  \overset{\cdot}{\underset{1}{-}}\left(  \gamma
_{3241}\overset{\cdot}{\underset{31}{-}}\gamma_{3421}\right)  \right)
\overset{\cdot}{-}\left(  \left(  \gamma_{1243}\overset{\cdot}{\underset
{31}{-}}\gamma_{1423}\right)  \overset{\cdot}{\underset{1}{-}}\left(
\gamma_{2413}\overset{\cdot}{\underset{31}{-}}\gamma_{4213}\right)  \right)
\right)  +\nonumber\\
&  \left(  \left(  \left(  \gamma_{3241}\overset{\cdot}{\underset{32}{-}%
}\gamma_{3214}\right)  \overset{\cdot}{\underset{1}{-}}\left(  \gamma
_{3412}\overset{\cdot}{\underset{32}{-}}\gamma_{3142}\right)  \right)
\overset{\cdot}{-}\left(  \left(  \gamma_{2413}\overset{\cdot}{\underset
{32}{-}}\gamma_{2143}\right)  \overset{\cdot}{\underset{1}{-}}\left(
\gamma_{4123}\overset{\cdot}{\underset{32}{-}}\gamma_{1423}\right)  \right)
\right)  +\nonumber\\
&  \left(  \left(  \left(  \gamma_{3412}\overset{\cdot}{\underset{34}{-}%
}\gamma_{3421}\right)  \overset{\cdot}{\underset{1}{-}}\left(  \gamma
_{3124}\overset{\cdot}{\underset{34}{-}}\gamma_{3214}\right)  \right)
\overset{\cdot}{-}\left(  \left(  \gamma_{4123}\overset{\cdot}{\underset
{34}{-}}\gamma_{4213}\right)  \overset{\cdot}{\underset{1}{-}}\left(
\gamma_{1243}\overset{\cdot}{\underset{34}{-}}\gamma_{2143}\right)  \right)
\right)  +\nonumber\\
&  \left(  \left(  \left(  \gamma_{4132}\overset{\cdot}{\underset{41}{-}%
}\gamma_{4123}\right)  \overset{\cdot}{\underset{1}{-}}\left(  \gamma
_{4321}\overset{\cdot}{\underset{41}{-}}\gamma_{4231}\right)  \right)
\overset{\cdot}{-}\left(  \left(  \gamma_{1324}\overset{\cdot}{\underset
{41}{-}}\gamma_{1234}\right)  \overset{\cdot}{\underset{1}{-}}\left(
\gamma_{3214}\overset{\cdot}{\underset{41}{-}}\gamma_{2314}\right)  \right)
\right)  +\nonumber\\
&  \left(  \left(  \left(  \gamma_{4213}\overset{\cdot}{\underset{42}{-}%
}\gamma_{4231}\right)  \overset{\cdot}{\underset{1}{-}}\left(  \gamma
_{4132}\overset{\cdot}{\underset{42}{-}}\gamma_{4312}\right)  \right)
\overset{\cdot}{-}\left(  \left(  \gamma_{2134}\overset{\cdot}{\underset
{42}{-}}\gamma_{2314}\right)  \overset{\cdot}{\underset{1}{-}}\left(
\gamma_{1324}\overset{\cdot}{\underset{42}{-}}\gamma_{3124}\right)  \right)
\right)  +\nonumber\\
&  \left(  \left(  \gamma_{4321}\overset{\cdot}{\underset{43}{-}}\gamma
_{4312}\right)  \overset{\cdot}{\underset{1}{-}}\left(  \gamma_{4213}%
\overset{\cdot}{\underset{43}{-}}\gamma_{4123}\right)  \right)  \overset
{\cdot}{-}\left(  \left(  \gamma_{3214}\overset{\cdot}{\underset{43}{-}}%
\gamma_{3124}\right)  \overset{\cdot}{\underset{1}{-}}\left(  \gamma
_{2134}\overset{\cdot}{\underset{43}{-}}\gamma_{1234}\right)  \right)
\nonumber\\
&  =0 \label{t3.2.0}%
\end{align}

\end{theorem}

\begin{proof}
The proof is divided into thirteen steps.

\begin{enumerate}
\item Since
\begin{align*}
&  \gamma_{1234}\left(  d_{1},d_{2},d_{3},d_{4}\right) \\
&  =\mathfrak{m}\left(  d_{1},d_{2},d_{3},d_{4},\underset{5}{0},...,\underset
{53}{0}\right)
\end{align*}
and
\begin{align*}
&  \gamma_{1243}\left(  d_{1},d_{2},d_{3},d_{4}\right) \\
&  =\mathfrak{m}\left(
\begin{array}
[c]{c}%
d_{1},d_{2},d_{3},d_{4},\underset{5}{0},...,\underset{9}{0},\underset
{10}{d_{3}d_{4}},\underset{11}{0},...,\underset{20}{0},d_{1}d_{3}%
d_{4},\underset{22}{0},...,\underset{25}{0},d_{2}d_{3}d_{4},\underset{27}%
{0},...,\underset{30}{0},\\
d_{1}d_{2}d_{3}d_{4},\underset{32}{0},...,\underset{53}{0}%
\end{array}
\right)
\end{align*}
we have
\begin{align}
&  \left(  \gamma_{1234}\overset{\cdot}{\underset{12}{-}}\gamma_{1243}\right)
\left(  d_{1},d_{2},d_{3}\right) \nonumber\\
&  =\mathfrak{m}\left(  d_{1},d_{2},\underset{3}{0},...,\underset{9}%
{0},\underset{10}{-d_{3}},\underset{11}{0},...,\underset{20}{0},-d_{1}%
d_{3},\underset{22}{0},...,\underset{25}{0},-d_{2}d_{3},\underset{27}%
{0},...,\underset{30}{0},-d_{1}d_{2}d_{3},\underset{32}{0},...,\underset
{53}{0}\right)  \label{t3.2.1}%
\end{align}
For we have%
\begin{align*}
&  \mathfrak{n}_{\left(  \gamma_{1234},\gamma_{1243}\right)  }^{4}\left(
d_{1},d_{2},d_{3},d_{4},d_{5}\right) \\
&  =\mathfrak{m}\left(
\begin{array}
[c]{c}%
d_{1},d_{2},d_{3},d_{4},\underset{5}{0},...,\underset{9}{0},\underset
{10}{d_{3}d_{4}-d_{5}},\underset{11}{0},...,\underset{20}{0},d_{1}d_{3}%
d_{4}-d_{1}d_{5},\underset{22}{0},...,\underset{25}{0},\\
d_{2}d_{3}d_{4}-d_{2}d_{5},\underset{27}{0},...,\underset{30}{0},d_{1}%
d_{2}d_{3}d_{4}-d_{1}d_{2}d_{5},\underset{32}{0},...,\underset{53}{0}%
\end{array}
\right)
\end{align*}
Since
\begin{align*}
&  \gamma_{1342}\left(  d_{1},d_{2},d_{3},d_{4}\right) \\
&  =\mathfrak{m}\left(
\begin{array}
[c]{c}%
d_{1},d_{2},d_{3},d_{4},\underset{5}{0},...,\underset{7}{0},d_{2}d_{3}%
,d_{2}d_{4},\underset{10}{0},d_{1}d_{2}d_{3},\underset{12}{0},...,\underset
{15}{0},d_{1}d_{2}d_{4},\underset{17}{0},...,\underset{27}{0},\\
d_{2}d_{3}d_{4},0,0,0,0,d_{1}d_{2}d_{3}d_{4},\underset{34}{0},...,\underset
{53}{0}%
\end{array}
\right)
\end{align*}
and
\begin{align*}
&  \gamma_{1432}\left(  d_{1},d_{2},d_{3},d_{4}\right) \\
&  =\mathfrak{m}\left(
\begin{array}
[c]{c}%
d_{1},d_{2},d_{3},d_{4},\underset{5}{0},...,\underset{7}{0},d_{2}d_{3}%
,d_{2}d_{4},d_{3}d_{4},d_{1}d_{2}d_{3},\underset{12}{0},...,\underset{15}%
{0},d_{1}d_{2}d_{4},\underset{17}{0},...,\underset{20}{0},\\
d_{1}d_{3}d_{4},\underset{22}{0},...,\underset{29}{0},d_{2}d_{3}%
d_{4},\underset{31}{0},...,\underset{34}{0},d_{1}d_{2}d_{3}d_{4},\underset
{36}{0},...,\underset{53}{0}%
\end{array}
\right)
\end{align*}
we have
\begin{align}
&  \left(  \gamma_{1342}\overset{\cdot}{\underset{12}{-}}\gamma_{1432}\right)
\left(  d_{1},d_{2},d_{3}\right) \nonumber\\
&  =\mathfrak{m}\left(
\begin{array}
[c]{c}%
d_{1},d_{2},\underset{3}{0},...,\underset{9}{0},-d_{3},\underset{11}%
{0},...,\underset{20}{0},-d_{1}d_{3},\underset{22}{0},...,\underset{27}%
{0},d_{2}d_{3},\underset{29}{0},-d_{2}d_{3},\underset{31}{0},\underset{32}%
{0},d_{1}d_{2}d_{3},\underset{34}{0},\\
-d_{1}d_{2}d_{3},\underset{36}{0},...,\underset{53}{0}%
\end{array}
\right)  \label{t3.2.2}%
\end{align}
For we have%
\begin{align*}
&  \mathfrak{n}_{\left(  \gamma_{1342},\gamma_{1432}\right)  }^{4}\left(
d_{1},d_{2},d_{3},d_{4},d_{5}\right) \\
&  =\mathfrak{m}\left(
\begin{array}
[c]{c}%
d_{1},d_{2},d_{3},d_{4},\underset{5}{0},...,\underset{9}{0},\underset
{10}{d_{3}d_{4}-d_{5}},\underset{11}{0},...,\underset{20}{0},d_{1}d_{3}%
d_{4}-d_{1}d_{5},\underset{22}{0},...,\underset{27}{0},\\
d_{2}d_{5},\underset{29}{0},d_{2}d_{3}d_{4}-d_{2}d_{5},\underset{31}%
{0},\underset{32}{0},d_{1}d_{2}d_{5},\underset{34}{0},d_{1}d_{2}d_{3}%
d_{4}-d_{1}d_{2}d_{5},\underset{36}{0},...,\underset{53}{0}%
\end{array}
\right)
\end{align*}
For the sake of the completeness of our proof, we have provided the reason why
(\ref{t3.2.1}) and (\ref{t3.2.2}) are derivable, but we will omit such
reasoning from now on. (\ref{t3.2.1}) and (\ref{t3.2.2}) imply that
\begin{align}
&  \left(  \left(  \gamma_{1234}\overset{\cdot}{\underset{12}{-}}\gamma
_{1243}\right)  \overset{\cdot}{\underset{1}{-}}\left(  \gamma_{1342}%
\overset{\cdot}{\underset{12}{-}}\gamma_{1432}\right)  \right)  \left(
d_{1},d_{2}\right) \nonumber\\
&  =\mathfrak{m}\left(  d_{1},\underset{2}{0},...,\underset{25}{0}%
,-d_{2},\underset{27}{0},-d_{2},\underset{29}{0},d_{2},-d_{1}d_{2}%
,\underset{32}{0},-d_{1}d_{2},\underset{34}{0},d_{1}d_{2},\underset{36}%
{0},...,\underset{53}{0}\right)  \label{t3.2.3}%
\end{align}
Since
\begin{align*}
&  \gamma_{2341}\left(  d_{1},d_{2},d_{3},d_{4}\right) \\
&  =\mathfrak{m}\left(
\begin{array}
[c]{c}%
d_{1},d_{2},d_{3},d_{4},d_{1}d_{2},d_{1}d_{3},d_{1}d_{4},\underset{8}%
{0},...,\underset{12}{0},d_{1}d_{2}d_{3},\underset{14}{0},...,\underset{17}%
{0},d_{1}d_{2}d_{4},\underset{19}{0},...,\underset{22}{0},\\
d_{1}d_{3}d_{4},\underset{24}{0},...,\underset{38}{0},d_{1}d_{2}d_{3}%
d_{4},\underset{40}{0},...,\underset{53}{0}%
\end{array}
\right)
\end{align*}
and
\begin{align*}
&  \gamma_{2431}\left(  d_{1},d_{2},d_{3},d_{4}\right) \\
&  =\mathfrak{m}\left(
\begin{array}
[c]{c}%
d_{1},d_{2},d_{3},d_{4},d_{1}d_{2},d_{1}d_{3},d_{1}d_{4},\underset{8}%
{0},\underset{9}{0},d_{3}d_{4},\underset{11}{0},\underset{12}{0},d_{1}%
d_{2}d_{3},\underset{14}{0},...,\underset{17}{0},d_{1}d_{2}d_{4},\\
\underset{19}{0},...,\underset{24}{0},d_{1}d_{3}d_{4},d_{2}d_{3}%
d_{4},\underset{27}{0},...,\underset{40}{0},d_{1}d_{2}d_{3}d_{4},\underset
{42}{0},...,\underset{53}{0}%
\end{array}
\right)
\end{align*}
we have
\begin{align}
&  \left(  \gamma_{2341}\overset{\cdot}{\underset{12}{-}}\gamma_{2431}\right)
\left(  d_{1},d_{2},d_{3}\right) \nonumber\\
&  =\mathfrak{m}\left(
\begin{array}
[c]{c}%
d_{1},d_{2},\underset{3}{0},...,\underset{9}{0},-d_{3},\underset{11}%
{0},...,\underset{22}{0},d_{1}d_{3},0,-d_{1}d_{3},-d_{2}d_{3},\\
\underset{27}{0},...,\underset{38}{0},d_{1}d_{2}d_{3},0,-d_{1}d_{2}%
d_{3},\underset{42}{0},...,\underset{53}{0}%
\end{array}
\right)  \label{t3.2.4}%
\end{align}
Since
\begin{align*}
&  \gamma_{3421}\left(  d_{1},d_{2},d_{3},d_{4}\right) \\
&  =\mathfrak{m}\left(
\begin{array}
[c]{c}%
d_{1},d_{2},d_{3},d_{4},d_{1}d_{2},d_{1}d_{3},d_{1}d_{4},d_{2}d_{3},d_{2}%
d_{4},\underset{10}{0},...,\underset{14}{0},d_{1}d_{2}d_{3},\underset{16}%
{0},...,\underset{19}{0},d_{1}d_{2}d_{4},\\
\underset{21}{0},\underset{22}{0},d_{1}d_{3}d_{4},\underset{24}{0}%
,...,\underset{27}{0},d_{2}d_{3}d_{4},\underset{29}{0},...,\underset{46}%
{0},d_{1}d_{2}d_{3}d_{4},\underset{48}{0},...,\underset{53}{0}%
\end{array}
\right)
\end{align*}
and
\begin{align*}
&  \gamma_{4321}\left(  d_{1},d_{2},d_{3},d_{4}\right) \\
&  =\mathfrak{m}\left(
\begin{array}
[c]{c}%
d_{1},d_{2},d_{3},d_{4},d_{1}d_{2},d_{1}d_{3},d_{1}d_{4},d_{2}d_{3},d_{2}%
d_{4},d_{3}d_{4},\underset{11}{0},...,\underset{14}{0},d_{1}d_{2}%
d_{3},\underset{16}{0},...,\underset{19}{0},\\
d_{1}d_{2}d_{4},\underset{21}{0},...,\underset{24}{0},d_{1}d_{3}%
d_{4},\underset{26}{0},...,\underset{29}{0},d_{2}d_{3}d_{4},\underset{31}%
{0},...,\underset{52}{0},d_{1}d_{2}d_{3}d_{4}%
\end{array}
\right)
\end{align*}
we have
\begin{align}
&  \left(  \gamma_{3421}\overset{\cdot}{\underset{12}{-}}\gamma_{4321}\right)
\left(  d_{1},d_{2},d_{3}\right) \nonumber\\
&  =\mathfrak{m}\left(
\begin{array}
[c]{c}%
d_{1},d_{2},\underset{3}{0},...,\underset{9}{0},-d_{3},\underset{11}%
{0},...,\underset{22}{0},d_{1}d_{3},\underset{24}{0},-d_{1}d_{3},\underset
{26}{0},\underset{27}{0},d_{2}d_{3},\underset{29}{0},-d_{2}d_{3},\\
\underset{31}{0},...,\underset{46}{0},d_{1}d_{2}d_{3},\underset{48}%
{0},...,\underset{52}{0},-d_{1}d_{2}d_{3}%
\end{array}
\right)  \label{t3.2.5}%
\end{align}
(\ref{t3.2.4}) and (\ref{t3.2.5}) imply that
\begin{align}
&  \left(  \left(  \gamma_{2341}\overset{\cdot}{\underset{12}{-}}\gamma
_{2431}\right)  \overset{\cdot}{\underset{1}{-}}\left(  \gamma_{3421}%
\overset{\cdot}{\underset{12}{-}}\gamma_{4321}\right)  \right)  \left(
d_{1},d_{2}\right) \nonumber\\
&  =\mathfrak{m}\left(
\begin{array}
[c]{c}%
d_{1},\underset{2}{0},...,\underset{25}{0},-d_{2},\underset{27}{0}%
,-d_{2},\underset{29}{0},d_{2},\underset{31}{0},...,\underset{38}{0}%
,d_{1}d_{2},\underset{40}{0},\\
-d_{1}d_{2},\underset{42}{0},...,\underset{46}{0},-d_{1}d_{2},\underset{48}%
{0},...,\underset{52}{0},d_{1}d_{2}%
\end{array}
\right)  \label{t3.2.6}%
\end{align}
(\ref{t3.2.3}) and (\ref{t3.2.6}) imply that
\begin{align}
&  \left(
\begin{array}
[c]{c}%
\left(  \left(  \gamma_{1234}\overset{\cdot}{\underset{12}{-}}\gamma
_{1243}\right)  \overset{\cdot}{\underset{1}{-}}\left(  \gamma_{1342}%
\overset{\cdot}{\underset{12}{-}}\gamma_{1432}\right)  \right)  \overset
{\cdot}{-}\\
\left(  \left(  \gamma_{2341}\overset{\cdot}{\underset{12}{-}}\gamma
_{2431}\right)  \overset{\cdot}{\underset{1}{-}}\left(  \gamma_{3421}%
\overset{\cdot}{\underset{12}{-}}\gamma_{4321}\right)  \right)
\end{array}
\right)  \left(  d\right) \nonumber\\
&  =\mathfrak{m}\left(  \underset{1}{0},...,\underset{30}{0},-d,\underset
{32}{0},-d,\underset{34}{0},d,\underset{36}{0},...,\underset{38}%
{0},-d,0,d,\underset{42}{0},...,\underset{46}{0},d,\underset{48}%
{0},...,\underset{52}{0},-d\right)  \label{t3.2.7}%
\end{align}

\item Since
\begin{align*}
&  \gamma_{1342}\left(  d_{1},d_{2},d_{3},d_{4}\right) \\
&  =\mathfrak{m}\left(
\begin{array}
[c]{c}%
d_{1},d_{2},d_{3},d_{4},\underset{5}{0},...,\underset{7}{0},d_{2}d_{3}%
,d_{2}d_{4},\underset{10}{0},d_{1}d_{2}d_{3},\underset{12}{0},...,\underset
{15}{0},d_{1}d_{2}d_{4},\underset{17}{0},...,\underset{27}{0},\\
d_{2}d_{3}d_{4},0,0,0,0,d_{1}d_{2}d_{3}d_{4},\underset{34}{0},...,\underset
{53}{0}%
\end{array}
\right)
\end{align*}
and
\begin{align*}
&  \gamma_{1324}\left(  d_{1},d_{2},d_{3},d_{4}\right) \\
&  =\mathfrak{m}\left(
\begin{array}
[c]{c}%
d_{1},d_{2},d_{3},d_{4},\underset{5}{0},...,\underset{7}{0},d_{2}%
d_{3},\underset{9}{0},\underset{10}{0},d_{1}d_{2}d_{3},\underset{12}%
{0},...,\underset{26}{0},d_{2}d_{3}d_{4},\underset{28}{0},...,\underset{31}%
{0},\\
d_{1}d_{2}d_{3}d_{4},\underset{33}{0},...,\underset{53}{0}%
\end{array}
\right)
\end{align*}
we have
\begin{align}
&  \left(  \gamma_{1342}\overset{\cdot}{\underset{13}{-}}\gamma_{1324}\right)
\left(  d_{1},d_{2},d_{3}\right) \nonumber\\
&  =\mathfrak{m}\left(
\begin{array}
[c]{c}%
d_{1},0,d_{2},\underset{4}{0},...,\underset{8}{0},d_{3},\underset{10}%
{0},...,\underset{15}{0},d_{1}d_{3},\underset{17}{0},...,\underset{26}%
{0},-d_{2}d_{3},d_{2}d_{3},\underset{29}{0},...,\underset{31}{0},\\
-d_{1}d_{2}d_{3},d_{1}d_{2}d_{3},\underset{34}{0},...,\underset{53}{0}%
\end{array}
\right)  \label{t3.2.8}%
\end{align}
Since
\begin{align*}
&  \gamma_{1423}\left(  d_{1},d_{2},d_{3},d_{4}\right) \\
&  =\mathfrak{m}\left(
\begin{array}
[c]{c}%
d_{1},d_{2},d_{3},d_{4},\underset{5}{0},...,\underset{8}{0},d_{2}d_{4}%
,d_{3}d_{4},\underset{11}{0},...,\underset{15}{0},d_{1}d_{2}d_{4}%
,\underset{17}{0},...,\underset{20}{0},d_{1}d_{3}d_{4},\\
\underset{22}{0},...,\underset{28}{0},d_{2}d_{3}d_{4},\underset{30}%
{0},...,\underset{33}{0},d_{1}d_{2}d_{3}d_{4},\underset{35}{0},...,\underset
{53}{0}%
\end{array}
\right)
\end{align*}
and
\begin{align*}
&  \gamma_{1243}\left(  d_{1},d_{2},d_{3},d_{4}\right) \\
&  =\left(
\begin{array}
[c]{c}%
d_{1},d_{2},d_{3},d_{4},\underset{5}{0},...,\underset{9}{0},\underset
{10}{d_{3}d_{4}},\underset{11}{0},...,\underset{20}{0},d_{1}d_{3}%
d_{4},\underset{22}{0},...,\underset{25}{0},d_{2}d_{3}d_{4},\underset{27}%
{0},...,\underset{30}{0},\\
d_{1}d_{2}d_{3}d_{4},\underset{32}{0},...,\underset{53}{0}%
\end{array}
\right)
\end{align*}
we have
\begin{align}
&  \left(  \gamma_{1423}\overset{\cdot}{\underset{13}{-}}\gamma_{1243}\right)
\left(  d_{1},d_{2},d_{3}\right) \nonumber\\
&  =\mathfrak{m}\left(
\begin{array}
[c]{c}%
d_{1},0,d_{2},\underset{4}{0},...,\underset{8}{0},d_{3},\underset{10}%
{0},...,\underset{15}{0},d_{1}d_{3},\underset{17}{0},...,\underset{25}%
{0},-d_{2}d_{3},\underset{27}{0},\underset{28}{0},d_{2}d_{3},\underset{30}%
{0},\\
-d_{1}d_{2}d_{3},\underset{32}{0},\underset{33}{0},d_{1}d_{2}d_{3}%
,\underset{35}{0},...,\underset{53}{0}%
\end{array}
\right)  \label{t3.2.9}%
\end{align}
(\ref{t3.2.8}) and (\ref{t3.2.9}) imply that%
\begin{align}
&  \left(  \left(  \gamma_{1324}\overset{\cdot}{\underset{13}{-}}\gamma
_{1342}\right)  \overset{\cdot}{\underset{1}{-}}\left(  \gamma_{1243}%
\overset{\cdot}{\underset{13}{-}}\gamma_{1423}\right)  \right)  \left(
d_{1},d_{2}\right) \nonumber\\
&  =\mathfrak{m}\left(  d_{1},\underset{2}{0},...,\underset{25}{0}%
,d_{2},-d_{2},d_{2},-d_{2},\underset{30}{0},d_{1}d_{2},-d_{1}d_{2},d_{1}%
d_{2},-d_{1}d_{2},\underset{35}{0},...,\underset{53}{0}\right)
\label{t3.2.10}%
\end{align}
Since
\begin{align*}
&  \gamma_{3421}\left(  d_{1},d_{2},d_{3},d_{4}\right) \\
&  =\mathfrak{m}\left(
\begin{array}
[c]{c}%
d_{1},d_{2},d_{3},d_{4},d_{1}d_{2},d_{1}d_{3},d_{1}d_{4},d_{2}d_{3},d_{2}%
d_{4},\underset{10}{0},...,\underset{14}{0},d_{1}d_{2}d_{3},\underset{16}%
{0},...,\underset{19}{0},d_{1}d_{2}d_{4},\\
0,0,d_{1}d_{3}d_{4},\underset{24}{0},...,\underset{27}{0},d_{2}d_{3}%
d_{4},\underset{29}{0},...,\underset{46}{0},d_{1}d_{2}d_{3}d_{4},\underset
{48}{0},...,\underset{53}{0}%
\end{array}
\right)
\end{align*}
and
\begin{align*}
&  \gamma_{3241}\left(  d_{1},d_{2},d_{3},d_{4}\right) \\
&  =\mathfrak{m}\left(
\begin{array}
[c]{c}%
d_{1},d_{2},d_{3},d_{4},d_{1}d_{2},d_{1}d_{3},d_{1}d_{4},d_{2}d_{3}%
,\underset{9}{0},...,\underset{14}{0},d_{1}d_{2}d_{3},\underset{16}%
{0},\underset{17}{0},d_{1}d_{2}d_{4},\underset{19}{0},...,\underset{22}{0},\\
d_{1}d_{3}d_{4},\underset{24}{0},...,\underset{26}{0},d_{2}d_{3}%
d_{4},\underset{28}{0},...,\underset{44}{0},d_{1}d_{2}d_{3}d_{4},\underset
{46}{0},...,\underset{53}{0}%
\end{array}
\right)
\end{align*}
we have
\begin{align}
&  \left(  \gamma_{3421}\overset{\cdot}{\underset{13}{-}}\gamma_{3241}\right)
\left(  d_{1},d_{2},d_{3}\right) \nonumber\\
&  =\mathfrak{m}\left(
\begin{array}
[c]{c}%
d_{1},0,d_{2},\underset{4}{0},...,\underset{8}{0},d_{3},\underset{10}%
{0},...,\underset{17}{0},-d_{1}d_{3},\underset{19}{0},d_{1}d_{3},\underset
{21}{0},...,\underset{26}{0},-d_{2}d_{3},d_{2}d_{3},\\
\underset{29}{0},...,\underset{44}{0},-d_{1}d_{2}d_{3},\underset{46}{0}%
,d_{1}d_{2}d_{3},\underset{48}{0},...,\underset{53}{0}%
\end{array}
\right)  \label{t3.2.11}%
\end{align}
Since
\begin{align*}
&  \gamma_{4231}\left(  d_{1},d_{2},d_{3},d_{4}\right) \\
&  =\mathfrak{m}\left(
\begin{array}
[c]{c}%
d_{1},d_{2},d_{3},d_{4},d_{1}d_{2},d_{1}d_{3},d_{1}d_{4},\underset{8}{0}%
,d_{2}d_{4},d_{3}d_{4},\underset{11}{0},\underset{12}{0},d_{1}d_{2}%
d_{3},\underset{14}{0},...,\underset{19}{0},\\
d_{1}d_{2}d_{4},\underset{21}{0},...,\underset{24}{0},d_{1}d_{3}%
d_{4},\underset{26}{0},...,\underset{28}{0},d_{2}d_{3}d_{4},\underset{30}%
{0},...,\underset{50}{0},d_{1}d_{2}d_{3}d_{4},0,\underset{53}{0}%
\end{array}
\right)
\end{align*}
and
\begin{align*}
&  \gamma_{2431}\left(  d_{1},d_{2},d_{3},d_{4}\right) \\
&  =\mathfrak{m}\left(
\begin{array}
[c]{c}%
d_{1},d_{2},d_{3},d_{4},d_{1}d_{2},d_{1}d_{3},d_{1}d_{4},\underset{8}%
{0},\underset{9}{0},d_{3}d_{4},\underset{11}{0},0,d_{1}d_{2}d_{3}%
,\underset{14}{0},...,\underset{17}{0},d_{1}d_{2}d_{4},\\
\underset{19}{0},...,\underset{24}{0},d_{1}d_{3}d_{4},d_{2}d_{3}%
d_{4},\underset{27}{0},...,\underset{40}{0},d_{1}d_{2}d_{3}d_{4},\underset
{42}{0},...,\underset{53}{0}%
\end{array}
\right)
\end{align*}
we have
\begin{align}
&  \left(  \gamma_{4231}\overset{\cdot}{\underset{13}{-}}\gamma_{2431}\right)
\left(  d_{1},d_{2},d_{3}\right) \nonumber\\
&  =\mathfrak{m}\left(
\begin{array}
[c]{c}%
d_{1},\underset{2}{0},d_{2},\underset{4}{0},...,\underset{8}{0},d_{3}%
,\underset{10}{0},...,\underset{17}{0},-d_{1}d_{3},\underset{19}{0},d_{1}%
d_{3},\underset{21}{0},...,\underset{25}{0},-d_{2}d_{3},\underset{27}%
{0},\underset{28}{0},d_{2}d_{3},\\
\underset{30}{0},...,\underset{40}{0},-d_{1}d_{2}d_{3},\underset{42}%
{0},...,\underset{50}{0},d_{1}d_{2}d_{3},\underset{52}{0},\underset{53}{0}%
\end{array}
\right)  \label{t3.2.12}%
\end{align}
(\ref{t3.2.11}) and (\ref{t3.2.12}) imply that
\begin{align}
&  \left(  \left(  \gamma_{3421}\overset{\cdot}{\underset{13}{-}}\gamma
_{3241}\right)  \overset{\cdot}{\underset{1}{-}}\left(  \gamma_{4231}%
\overset{\cdot}{\underset{13}{-}}\gamma_{2431}\right)  \right)  \left(
d_{1},d_{2}\right) \nonumber\\
&  =\mathfrak{m}\left(
\begin{array}
[c]{c}%
d_{1},\underset{2}{0},...,\underset{25}{0},d_{2},-d_{2},d_{2},-d_{2}%
,\underset{30}{0},...,\underset{40}{0},d_{1}d_{2},\underset{42}{0}%
,...,\underset{44}{0},\\
-d_{1}d_{2},\underset{46}{0},d_{1}d_{2},\underset{48}{0},...,\underset{50}%
{0},-d_{1}d_{2},\underset{52}{0},\underset{53}{0}%
\end{array}
\right)  \label{t3.2.13}%
\end{align}
(\ref{t3.2.10}) and (\ref{t3.2.13}) imply that
\begin{align}
&  \left(
\begin{array}
[c]{c}%
\left(  \left(  \gamma_{1324}\overset{\cdot}{\underset{13}{-}}\gamma
_{1342}\right)  \overset{\cdot}{\underset{1}{-}}\left(  \gamma_{1243}%
\overset{\cdot}{\underset{13}{-}}\gamma_{1423}\right)  \right)  \overset
{\cdot}{-}\\
\left(  \left(  \gamma_{3241}\overset{\cdot}{\underset{13}{-}}\gamma
_{3421}\right)  \overset{\cdot}{\underset{1}{-}}\left(  \gamma_{2431}%
\overset{\cdot}{\underset{13}{-}}\gamma_{4231}\right)  \right)
\end{array}
\right)  \left(  d\right) \nonumber\\
&  =\mathfrak{m}\left(  \underset{1}{0},...,\underset{30}{0}%
,d,-d,d,-d,\underset{35}{0},...,\underset{40}{0},-d,\underset{42}%
{0},...\underset{44}{0},d,\underset{46}{0},-d,\underset{48}{0},...,\underset
{50}{0},d,\underset{52}{0},\underset{53}{0}\right)  \label{t3.2.14}%
\end{align}

\item Since
\begin{align*}
&  \gamma_{1423}\left(  d_{1},d_{2},d_{3},d_{4}\right) \\
&  =\mathfrak{m}\left(
\begin{array}
[c]{c}%
d_{1},d_{2},d_{3},d_{4},\underset{5}{0},...,\underset{8}{0},d_{2}d_{4}%
,d_{3}d_{4},\underset{11}{0},...,\underset{15}{0},d_{1}d_{2}d_{4}%
,\underset{17}{0},...,\underset{20}{0},d_{1}d_{3}d_{4},\\
\underset{22}{0},...,\underset{28}{0},d_{2}d_{3}d_{4},\underset{30}%
{0},...,\underset{33}{0},d_{1}d_{2}d_{3}d_{4},\underset{35}{0},...,\underset
{53}{0}%
\end{array}
\right)
\end{align*}
and
\begin{align*}
&  \gamma_{1432}\left(  d_{1},d_{2},d_{3},d_{4}\right) \\
&  =\mathfrak{m}\left(
\begin{array}
[c]{c}%
d_{1},d_{2},d_{3},d_{4},\underset{5}{0},...,\underset{7}{0},d_{2}d_{3}%
,d_{2}d_{4},d_{3}d_{4},d_{1}d_{2}d_{3},\underset{12}{0},...,\underset{15}%
{0},d_{1}d_{2}d_{4},\underset{17}{0},...,\underset{20}{0},\\
d_{1}d_{3}d_{4},\underset{22}{0},...,\underset{29}{0},d_{2}d_{3}%
d_{4},\underset{31}{0},...,\underset{34}{0},d_{1}d_{2}d_{3}d_{4},\underset
{36}{0},...,\underset{53}{0}%
\end{array}
\right)
\end{align*}
we have
\begin{align}
&  \left(  \gamma_{1423}\overset{\cdot}{\underset{14}{-}}\gamma_{1432}\right)
\left(  d_{1},d_{2},d_{3}\right) \nonumber\\
&  =\mathfrak{m}\left(
\begin{array}
[c]{c}%
d_{1},\underset{2}{0},\underset{3}{0},d_{2},\underset{5}{0},...,\underset
{7}{0},-d_{3},\underset{9}{0},\underset{10}{0},-d_{1}d_{3},\underset{12}%
{0},...,\underset{28}{0},d_{2}d_{3},-d_{2}d_{3},\\
\underset{31}{0},...,\underset{33}{0},d_{1}d_{2}d_{3},-d_{1}d_{2}%
d_{3},\underset{36}{0},...,\underset{53}{0}%
\end{array}
\right)  \label{t3.2.15}%
\end{align}
Since
\begin{align*}
&  \gamma_{1234}\left(  d_{1},d_{2},d_{3},d_{4}\right) \\
&  =\mathfrak{m}\left(  d_{1},d_{2},d_{3},d_{4},\underset{5}{0},...,\underset
{53}{0}\right)
\end{align*}
and
\begin{align*}
&  \gamma_{1324}\left(  d_{1},d_{2},d_{3},d_{4}\right) \\
&  =\mathfrak{m}\left(
\begin{array}
[c]{c}%
d_{1},d_{2},d_{3},d_{4},\underset{5}{0},...,\underset{7}{0},d_{2}%
d_{3},\underset{9}{0},\underset{10}{0},d_{1}d_{2}d_{3},\underset{12}%
{0},...,\underset{26}{0},d_{2}d_{3}d_{4},\underset{28}{0},...,\underset{30}%
{0},\\
0,d_{1}d_{2}d_{3}d_{4},\underset{33}{0},...,\underset{53}{0}%
\end{array}
\right)
\end{align*}
we have
\begin{align}
&  \left(  \gamma_{1234}\overset{\cdot}{\underset{14}{-}}\gamma_{1324}\right)
\left(  d_{1},d_{2},d_{3}\right) \nonumber\\
&  =\mathfrak{m}\left(
\begin{array}
[c]{c}%
d_{1},\underset{2}{0},\underset{3}{0},d_{2},\underset{5}{0},...,\underset
{7}{0},-d_{3},\underset{9}{0},\underset{10}{0},-d_{1}d_{3},\underset{12}%
{0},...,\underset{26}{0},-d_{2}d_{3},\\
\underset{28}{0},...,\underset{31}{0},-d_{1}d_{2}d_{3},\underset{33}%
{0},...,\underset{53}{0}%
\end{array}
\right)  \label{t3.2.16}%
\end{align}
(\ref{t3.2.15}) and (\ref{t3.2.16}) imply that
\begin{align}
&  \left(  \left(  \gamma_{1423}\overset{\cdot}{\underset{14}{-}}\gamma
_{1432}\right)  \overset{\cdot}{\underset{1}{-}}\left(  \gamma_{1234}%
\overset{\cdot}{\underset{14}{-}}\gamma_{1324}\right)  \right)  \left(
d_{1},d_{2}\right) \nonumber\\
&  =\mathfrak{m}\left(  d_{1},\underset{2}{0},...,\underset{26}{0}%
,d_{2},\underset{28}{0},d_{2},-d_{2},\underset{31}{0},d_{1}d_{2},\underset
{33}{0},d_{1}d_{2},-d_{1}d_{2},\underset{36}{0},...,\underset{53}{0}\right)
\label{t3.2.17}%
\end{align}
Since
\begin{align*}
&  \gamma_{4231}\left(  d_{1},d_{2},d_{3},d_{4}\right) \\
&  =\mathfrak{m}\left(
\begin{array}
[c]{c}%
d_{1},d_{2},d_{3},d_{4},d_{1}d_{2},d_{1}d_{3},d_{1}d_{4},\underset{8}{0}%
,d_{2}d_{4},d_{3}d_{4},\underset{11}{0},\underset{12}{0},d_{1}d_{2}%
d_{3},\underset{14}{0},...,\underset{19}{0},\\
d_{1}d_{2}d_{4},\underset{21}{0},...,\underset{24}{0},d_{1}d_{3}%
d_{4},\underset{26}{0},...,\underset{28}{0},d_{2}d_{3}d_{4},\underset{30}%
{0},...,\underset{50}{0},d_{1}d_{2}d_{3}d_{4},\underset{52}{0},\underset
{53}{0}%
\end{array}
\right)
\end{align*}
and
\begin{align*}
&  \gamma_{4321}\left(  d_{1},d_{2},d_{3},d_{4}\right) \\
&  =\mathfrak{m}\left(
\begin{array}
[c]{c}%
d_{1},d_{2},d_{3},d_{4},d_{1}d_{2},d_{1}d_{3},d_{1}d_{4},d_{2}d_{3},d_{2}%
d_{4},d_{3}d_{4},\underset{11}{0},...,\underset{14}{0},d_{1}d_{2}%
d_{3},\underset{16}{0},...,\underset{19}{0},\\
d_{1}d_{2}d_{4},\underset{21}{0},...,\underset{24}{0},d_{1}d_{3}%
d_{4},\underset{26}{0},...,\underset{29}{0},d_{2}d_{3}d_{4},\underset{31}%
{0},...,\underset{52}{0},d_{1}d_{2}d_{3}d_{4}%
\end{array}
\right)
\end{align*}
we have
\begin{align}
&  \left(  \gamma_{4231}\overset{\cdot}{\underset{14}{-}}\gamma_{4321}\right)
\left(  d_{1},d_{2},d_{3}\right) \nonumber\\
&  =\mathfrak{m}\left(
\begin{array}
[c]{c}%
d_{1},\underset{2}{0},\underset{3}{0},d_{2},\underset{5}{0},...,\underset
{7}{0},-d_{3},\underset{9}{0},...,\underset{12}{0},d_{1}d_{3},\underset{14}%
{0},-d_{1}d_{3},\underset{16}{0},...,\underset{28}{0},d_{2}d_{3},-d_{2}%
d_{3},\\
\underset{31}{0},...,\underset{50}{0},d_{1}d_{2}d_{3},\underset{52}{0}%
,-d_{1}d_{2}d_{3}%
\end{array}
\right)  \label{t3.2.18}%
\end{align}
Since
\begin{align*}
&  \gamma_{2341}\left(  d_{1},d_{2},d_{3},d_{4}\right) \\
&  =\mathfrak{m}\left(
\begin{array}
[c]{c}%
d_{1},d_{2},d_{3},d_{4},d_{1}d_{2},d_{1}d_{3},d_{1}d_{4},\underset{8}%
{0},...,\underset{12}{0},d_{1}d_{2}d_{3},\underset{14}{0},...,\underset{17}%
{0},d_{1}d_{2}d_{4},\underset{19}{0},...,\underset{22}{0},\\
d_{1}d_{3}d_{4},\underset{24}{0},...,\underset{38}{0},d_{1}d_{2}d_{3}%
d_{4},\underset{40}{0},...,\underset{53}{0}%
\end{array}
\right)
\end{align*}
and
\begin{align*}
&  \gamma_{3241}\left(  d_{1},d_{2},d_{3},d_{4}\right) \\
&  =\mathfrak{m}\left(
\begin{array}
[c]{c}%
d_{1},d_{2},d_{3},d_{4},d_{1}d_{2},d_{1}d_{3},d_{1}d_{4},d_{2}d_{3}%
,\underset{9}{0},...,\underset{14}{0},d_{1}d_{2}d_{3},\underset{16}%
{0},\underset{17}{0},d_{1}d_{2}d_{4},\underset{19}{0},...,\underset{22}{0},\\
d_{1}d_{3}d_{4},\underset{24}{0},...,\underset{26}{0},d_{2}d_{3}%
d_{4},\underset{28}{0},...,\underset{44}{0},d_{1}d_{2}d_{3}d_{4},\underset
{46}{0},...,\underset{53}{0}%
\end{array}
\right)
\end{align*}
we have
\begin{align}
&  \left(  \gamma_{2341}\overset{\cdot}{\underset{14}{-}}\gamma_{3241}\right)
\left(  d_{1},d_{2},d_{3}\right) \nonumber\\
&  =\mathfrak{m}\left(
\begin{array}
[c]{c}%
d_{1},\underset{2}{0},\underset{3}{0},d_{2},\underset{5}{0},...,\underset
{7}{0},-d_{3},\underset{9}{0},...,\underset{12}{0},d_{1}d_{3},\underset{14}%
{0},-d_{1}d_{3},\underset{16}{0},...,\underset{26}{0},-d_{2}d_{3},\\
\underset{28}{0},...,\underset{38}{0},d_{1}d_{2}d_{3},\underset{40}%
{0},...,\underset{44}{0},-d_{1}d_{2}d_{3},\underset{46}{0},...,\underset
{53}{0}%
\end{array}
\right)  \label{t3.2.19}%
\end{align}
(\ref{t3.2.18}) and (\ref{t3.2.19}) imply that
\begin{align}
&  \left(  \left(  \gamma_{4231}\overset{\cdot}{\underset{14}{-}}\gamma
_{4321}\right)  \overset{\cdot}{\underset{1}{-}}\left(  \gamma_{2341}%
\overset{\cdot}{\underset{14}{-}}\gamma_{3241}\right)  \right)  \left(
d_{1},d_{2}\right) \nonumber\\
&  =\mathfrak{m}\left(
\begin{array}
[c]{c}%
d_{1},\underset{2}{0},...,\underset{26}{0},d_{2},\underset{28}{0},d_{2}%
,-d_{2},\underset{31}{0},...,\underset{38}{0},-d_{1}d_{2},\\
\underset{40}{0},...,\underset{44}{0},d_{1}d_{2},\underset{46}{0}%
,...,\underset{50}{0},d_{1}d_{2},\underset{52}{0},-d_{1}d_{2}%
\end{array}
\right)  \label{t3.2.20}%
\end{align}
(\ref{t3.2.17}) and (\ref{t3.2.20}) imply that
\begin{align}
&  \left(
\begin{array}
[c]{c}%
\left(  \left(  \gamma_{1423}\overset{\cdot}{\underset{14}{-}}\gamma
_{1432}\right)  \overset{\cdot}{\underset{1}{-}}\left(  \gamma_{1234}%
\overset{\cdot}{\underset{14}{-}}\gamma_{1324}\right)  \right)  \overset
{\cdot}{-}\\
\left(  \left(  \gamma_{4231}\overset{\cdot}{\underset{14}{-}}\gamma
_{4321}\right)  \overset{\cdot}{\underset{1}{-}}\left(  \gamma_{2341}%
\overset{\cdot}{\underset{14}{-}}\gamma_{3241}\right)  \right)
\end{array}
\right)  \left(  d\right) \nonumber\\
&  =\mathfrak{m}\left(  \underset{1}{0},...,\underset{31}{0},d,\underset
{33}{0},d,-d,\underset{36}{0},...,\underset{38}{0},d,\underset{40}%
{0},...,\underset{44}{0},-d,\underset{46}{0},...,\underset{50}{0}%
,-d,\underset{52}{0},d\right)  \label{t3.2.21}%
\end{align}

\item Since
\begin{align*}
&  \gamma_{2143}\left(  d_{1},d_{2},d_{3},d_{4}\right) \\
&  =\mathfrak{m}\left(
\begin{array}
[c]{c}%
d_{1},d_{2},d_{3},d_{4},d_{1}d_{2},\underset{6}{0},...,\underset{9}{0}%
,d_{3}d_{4},\underset{11}{0},d_{1}d_{2}d_{3},\underset{13}{0},...,\underset
{16}{0},d_{1}d_{2}d_{4},\underset{18}{0},...,\underset{20}{0},d_{1}d_{3}%
d_{4},\\
\underset{22}{0},...,\underset{25}{0},d_{2}d_{3}d_{4},\underset{27}%
{0},...,\underset{36}{0},d_{1}d_{2}d_{3}d_{4},\underset{38}{0},...,\underset
{53}{0}%
\end{array}
\right)
\end{align*}
and
\begin{align*}
&  \gamma_{2134}\left(  d_{1},d_{2},d_{3},d_{4}\right) \\
&  =\mathfrak{m}\left(  d_{1},d_{2},d_{3},d_{4},d_{1}d_{2},\underset{6}%
{0},...,\underset{11}{0},d_{1}d_{2}d_{3},\underset{13}{0},...,\underset{16}%
{0},d_{1}d_{2}d_{4},\underset{18}{0},...,\underset{35}{0},d_{1}d_{2}d_{3}%
d_{4},\underset{37}{0},...,\underset{53}{0}\right)
\end{align*}
we have
\begin{align}
&  \left(  \gamma_{2143}\overset{\cdot}{\underset{21}{-}}\gamma_{2134}\right)
\left(  d_{1},d_{2},d_{3}\right) \nonumber\\
&  =\mathfrak{m}\left(
\begin{array}
[c]{c}%
d_{2},d_{1},\underset{3}{0},\underset{4}{0},d_{1}d_{2},\underset{6}%
{0},...,\underset{9}{0},d_{3},\underset{11}{0},...,\underset{20}{0},d_{2}%
d_{3},\underset{22}{0},...,\underset{25}{0},d_{1}d_{3},\\
\underset{27}{0},...,\underset{35}{0},-d_{1}d_{2}d_{3},d_{1}d_{2}%
d_{3},\underset{38}{0},...,\underset{53}{0}%
\end{array}
\right)  \label{t3.2.22}%
\end{align}
Since
\begin{align*}
&  \gamma_{2431}\left(  d_{1},d_{2},d_{3},d_{4}\right) \\
&  =\mathfrak{m}\left(
\begin{array}
[c]{c}%
d_{1},d_{2},d_{3},d_{4},d_{1}d_{2},d_{1}d_{3},d_{1}d_{4},\underset{8}%
{0},\underset{9}{0},d_{3}d_{4},\underset{11}{0},\underset{12}{0},d_{1}%
d_{2}d_{3},\underset{14}{0},...,\underset{17}{0},d_{1}d_{2}d_{4},\\
\underset{19}{0},...,\underset{24}{0},d_{1}d_{3}d_{4},d_{2}d_{3}%
d_{4},\underset{27}{0},...,\underset{40}{0},d_{1}d_{2}d_{3}d_{4},\underset
{42}{0},...,\underset{53}{0}%
\end{array}
\right)
\end{align*}
and
\begin{align*}
&  \gamma_{2341}\left(  d_{1},d_{2},d_{3},d_{4}\right) \\
&  =\mathfrak{m}\left(
\begin{array}
[c]{c}%
d_{1},d_{2},d_{3},d_{4},d_{1}d_{2},d_{1}d_{3},d_{1}d_{4},\underset{8}%
{0},...,\underset{12}{0},d_{1}d_{2}d_{3},\underset{14}{0},...,\underset{17}%
{0},d_{1}d_{2}d_{4},\underset{19}{0},...,\underset{22}{0},\\
d_{1}d_{3}d_{4},\underset{24}{0},...,\underset{38}{0},d_{1}d_{2}d_{3}%
d_{4},\underset{40}{0},...,\underset{53}{0}%
\end{array}
\right)
\end{align*}
we have
\begin{align}
&  \left(  \gamma_{2431}\overset{\cdot}{\underset{21}{-}}\gamma_{2341}\right)
\left(  d_{1},d_{2},d_{3}\right) \nonumber\\
&  =\mathfrak{m}\left(
\begin{array}
[c]{c}%
d_{2},d_{1},\underset{3}{0},\underset{4}{0},d_{1}d_{2},\underset{6}%
{0},...,\underset{9}{0},d_{3},\underset{11}{0},...,\underset{22}{0}%
,-d_{2}d_{3},\underset{24}{0},d_{2}d_{3},d_{1}d_{3},\\
\underset{27}{0},...,\underset{38}{0},-d_{1}d_{2}d_{3},\underset{40}{0}%
,d_{1}d_{2}d_{3},\underset{42}{0},...,\underset{53}{0}%
\end{array}
\right)  \label{t3.2.23}%
\end{align}
(\ref{t3.2.22}) and (\ref{t3.2.23}) imply that
\begin{align}
&  \left(  \left(  \gamma_{2143}\overset{\cdot}{\underset{21}{-}}\gamma
_{2134}\right)  \overset{\cdot}{\underset{1}{-}}\left(  \gamma_{2431}%
\overset{\cdot}{\underset{21}{-}}\gamma_{2341}\right)  \right)  \left(
d_{1},d_{2}\right) \nonumber\\
&  =\mathfrak{m}\left(
\begin{array}
[c]{c}%
\underset{1}{0},d_{1},\underset{3}{0},...,\underset{20}{0},d_{2},\underset
{22}{0},d_{2},\underset{24}{0},-d_{2},\underset{26}{0},...,\underset{35}{0},\\
-d_{1}d_{2},d_{1}d_{2},\underset{38}{0},d_{1}d_{2},\underset{40}{0}%
,-d_{1}d_{2},\underset{42}{0},...,\underset{53}{0}%
\end{array}
\right)  \label{t3.2.24}%
\end{align}
Since
\begin{align*}
&  \gamma_{1432}\left(  d_{1},d_{2},d_{3},d_{4}\right) \\
&  =\mathfrak{m}\left(
\begin{array}
[c]{c}%
d_{1},d_{2},d_{3},d_{4},\underset{5}{0},...,\underset{7}{0},d_{2}d_{3}%
,d_{2}d_{4},d_{3}d_{4},d_{1}d_{2}d_{3},\underset{12}{0},...,\underset{15}%
{0},d_{1}d_{2}d_{4},\underset{17}{0},...,\underset{20}{0},\\
d_{1}d_{3}d_{4},\underset{22}{0},...,\underset{29}{0},d_{2}d_{3}%
d_{4},\underset{31}{0},...,\underset{34}{0},d_{1}d_{2}d_{3}d_{4},\underset
{36}{0},...,\underset{53}{0}%
\end{array}
\right)
\end{align*}
and
\begin{align*}
&  \gamma_{1342}\left(  d_{1},d_{2},d_{3},d_{4}\right) \\
&  =\mathfrak{m}\left(
\begin{array}
[c]{c}%
d_{1},d_{2},d_{3},d_{4},\underset{5}{0},...,\underset{7}{0},d_{2}d_{3}%
,d_{2}d_{4},\underset{10}{0},d_{1}d_{2}d_{3},\underset{12}{0},...,\underset
{15}{0},d_{1}d_{2}d_{4},\underset{17}{0},...,\underset{27}{0},\\
d_{2}d_{3}d_{4},0,0,0,0,d_{1}d_{2}d_{3}d_{4},\underset{34}{0},...,\underset
{53}{0}%
\end{array}
\right)
\end{align*}
we have
\begin{align}
&  \left(  \gamma_{1432}\overset{\cdot}{\underset{21}{-}}\gamma_{1342}\right)
\left(  d_{1},d_{2},d_{3}\right) \nonumber\\
&  =\mathfrak{m}\left(
\begin{array}
[c]{c}%
d_{2},d_{1},\underset{3}{0},\underset{4}{0},d_{1}d_{2},\underset{6}%
{0},...,\underset{9}{0},d_{3},\underset{11}{0},...,\underset{20}{0},d_{1}%
d_{3},\underset{22}{0},...,\underset{27}{0},-d_{2}d_{3},\underset{29}{0}%
,d_{2}d_{3},\\
\underset{31}{0},\underset{32}{0},-d_{1}d_{2}d_{3},\underset{34}{0},d_{1}%
d_{2}d_{3},\underset{36}{0},...,\underset{53}{0}%
\end{array}
\right)  \label{t3.2.25}%
\end{align}
Since
\begin{align*}
&  \gamma_{4312}\left(  d_{1},d_{2},d_{3},d_{4}\right) \\
&  =\mathfrak{m}\left(
\begin{array}
[c]{c}%
d_{1},d_{2},d_{3},d_{4},\underset{5}{0},d_{1}d_{3},d_{1}d_{4},d_{2}d_{3}%
,d_{2}d_{4},d_{3}d_{4},\underset{11}{0},...,\underset{13}{0},d_{1}d_{2}%
d_{3},\underset{15}{0},...,\underset{18}{0},\\
d_{1}d_{2}d_{4},\underset{20}{0},...,\underset{24}{0},d_{1}d_{3}%
d_{4},\underset{26}{0},...,\underset{29}{0},d_{2}d_{3}d_{4},\underset{31}%
{0},...,\underset{51}{0},d_{1}d_{2}d_{3}d_{4},\underset{53}{0}%
\end{array}
\right)
\end{align*}
and
\begin{align*}
&  \gamma_{3412}\left(  d_{1},d_{2},d_{3},d_{4}\right) \\
&  =\left(
\begin{array}
[c]{c}%
d_{1},d_{2},d_{3},d_{4},\underset{5}{0},d_{1}d_{3},d_{1}d_{4},d_{2}d_{3}%
,d_{2}d_{4},\underset{10}{0},...,\underset{13}{0},d_{1}d_{2}d_{3}%
,\underset{15}{0},...,\underset{18}{0},d_{1}d_{2}d_{4},\\
\underset{20}{0},...,\underset{22}{0},d_{1}d_{3}d_{4},\underset{24}%
{0},...,\underset{27}{0},d_{2}d_{3}d_{4},\underset{29}{0},...,\underset{45}%
{0},d_{1}d_{2}d_{3}d_{4},\underset{47}{0},...,\underset{53}{0}%
\end{array}
\right)
\end{align*}
we have
\begin{align}
&  \left(  \gamma_{4312}\overset{\cdot}{\underset{21}{-}}\gamma_{3412}\right)
\left(  d_{1},d_{2},d_{3}\right) \nonumber\\
&  =\mathfrak{m}\left(
\begin{array}
[c]{c}%
d_{2},d_{1},\underset{3}{0},\underset{4}{0},d_{1}d_{2},\underset{6}%
{0},...,\underset{9}{0},d_{3},\underset{11}{0},...,\underset{20}{0}%
,\underset{21}{0},\underset{22}{0},-d_{2}d_{3},\underset{24}{0},d_{2}%
d_{3},\underset{26}{0},\underset{27}{0},\\
-d_{1}d_{3},\underset{29}{0},d_{1}d_{3},\underset{31}{0},...,\underset{45}%
{0},-d_{1}d_{2}d_{3},\underset{47}{0},...,\underset{51}{0},d_{1}d_{2}%
d_{3},\underset{53}{0}%
\end{array}
\right)  \label{t3.2.26}%
\end{align}
(\ref{t3.2.25}) and (\ref{t3.2.26}) imply that
\begin{align}
&  \left(  \left(  \gamma_{1432}\overset{\cdot}{\underset{21}{-}}\gamma
_{1342}\right)  \overset{\cdot}{\underset{1}{-}}\left(  \gamma_{4312}%
\overset{\cdot}{\underset{21}{-}}\gamma_{3412}\right)  \right)  \left(
d_{1},d_{2}\right) \nonumber\\
&  =\mathfrak{m}\left(
\begin{array}
[c]{c}%
\underset{1}{0},d_{1},\underset{3}{0},...,\underset{20}{0},d_{2},\underset
{22}{0},d_{2},\underset{24}{0},-d_{2},\underset{26}{0},...,\underset{32}%
{0},-d_{1}d_{2},\underset{34}{0},d_{1}d_{2},\\
\underset{36}{0},...,\underset{45}{0},d_{1}d_{2},\underset{47}{0}%
,...,\underset{51}{0},-d_{1}d_{2},\underset{53}{0}%
\end{array}
\right)  \label{t3.2.27}%
\end{align}
(\ref{t3.2.24}) and (\ref{t3.2.27}) imply that
\begin{align}
&  \left(
\begin{array}
[c]{c}%
\left(  \left(  \gamma_{2143}\overset{\cdot}{\underset{21}{-}}\gamma
_{2134}\right)  \overset{\cdot}{\underset{1}{-}}\left(  \gamma_{2431}%
\overset{\cdot}{\underset{21}{-}}\gamma_{2341}\right)  \right)  \overset
{\cdot}{-}\\
\left(  \left(  \gamma_{1432}\overset{\cdot}{\underset{21}{-}}\gamma
_{1342}\right)  \overset{\cdot}{\underset{1}{-}}\left(  \gamma_{4312}%
\overset{\cdot}{\underset{21}{-}}\gamma_{3412}\right)  \right)
\end{array}
\right)  \left(  d\right) \nonumber\\
&  =\mathfrak{m}\left(  \underset{1}{0},...,\underset{32}{0},d,\underset
{34}{0},-d,-d,d,\underset{38}{0},d,\underset{40}{0},-d,\underset{42}%
{0},...,\underset{45}{0},-d,\underset{47}{0},...,\underset{51}{0}%
,d,\underset{53}{0}\right)  \label{t3.2.28}%
\end{align}

\item Since
\begin{align*}
&  \gamma_{2314}\left(  d_{1},d_{2},d_{3},d_{4}\right) \\
&  =\mathfrak{m}\left(
\begin{array}
[c]{c}%
d_{1},d_{2},d_{3},d_{4},d_{1}d_{2},d_{1}d_{3},\underset{7}{0},...,\underset
{12}{0},d_{1}d_{2}d_{3},0,0,0,d_{1}d_{2}d_{4},\underset{18}{0},...,\underset
{21}{0},d_{1}d_{3}d_{4},\\
\underset{23}{0},...,\underset{37}{0},d_{1}d_{2}d_{3}d_{4},\underset{39}%
{0},...,\underset{53}{0}%
\end{array}
\right)
\end{align*}
and
\begin{align*}
&  \gamma_{2341}\left(  d_{1},d_{2},d_{3},d_{4}\right) \\
&  =\mathfrak{m}\left(
\begin{array}
[c]{c}%
d_{1},d_{2},d_{3},d_{4},d_{1}d_{2},d_{1}d_{3},d_{1}d_{4},\underset{8}%
{0},...,\underset{12}{0},d_{1}d_{2}d_{3},\underset{14}{0},...,\underset{17}%
{0},d_{1}d_{2}d_{4},\underset{19}{0},...,\underset{22}{0},\\
d_{1}d_{3}d_{4},\underset{24}{0},...,\underset{38}{0},d_{1}d_{2}d_{3}%
d_{4},\underset{40}{0},...,\underset{53}{0}%
\end{array}
\right)
\end{align*}
we have
\begin{align}
&  \left(  \gamma_{2314}\overset{\cdot}{\underset{23}{-}}\gamma_{2341}\right)
\left(  d_{1},d_{2},d_{3}\right) \nonumber\\
&  =\mathfrak{m}\left(
\begin{array}
[c]{c}%
\underset{1}{0},d_{1},d_{2},\underset{4}{0},...,\underset{6}{0},-d_{3}%
,\underset{8}{0},...,\underset{16}{0},d_{1}d_{3},-d_{1}d_{3},\underset{19}%
{0},...,\underset{21}{0},d_{2}d_{3},-d_{2}d_{3},\\
\underset{24}{0},...,\underset{37}{0},d_{1}d_{2}d_{3},-d_{1}d_{2}%
d_{3},\underset{40}{0},...,\underset{53}{0}%
\end{array}
\right)  \label{t3.2.29}%
\end{align}
Since
\begin{align*}
&  \gamma_{2143}\left(  d_{1},d_{2},d_{3},d_{4}\right) \\
&  =\mathfrak{m}\left(
\begin{array}
[c]{c}%
d_{1},d_{2},d_{3},d_{4},d_{1}d_{2},\underset{6}{0},...,\underset{9}{0}%
,d_{3}d_{4},0,d_{1}d_{2}d_{3},\underset{13}{0},...,\underset{16}{0},d_{1}%
d_{2}d_{4},0,0,0,d_{1}d_{3}d_{4},\\
\underset{22}{0},...,\underset{25}{0},d_{2}d_{3}d_{4},\underset{27}%
{0},...,\underset{36}{0},d_{1}d_{2}d_{3}d_{4},\underset{38}{0},...,\underset
{53}{0}%
\end{array}
\right)
\end{align*}
and
\begin{align*}
&  \gamma_{2413}\left(  d_{1},d_{2},d_{3},d_{4}\right) \\
&  =\mathfrak{m}\left(
\begin{array}
[c]{c}%
d_{1},d_{2},d_{3},d_{4},d_{1}d_{2},\underset{6}{0},d_{1}d_{4},\underset{8}%
{0},\underset{9}{0},d_{3}d_{4},\underset{11}{0},d_{1}d_{2}d_{3},\underset
{13}{0},...,\underset{17}{0},d_{1}d_{2}d_{4},\underset{19}{0},...,\underset
{23}{0},\\
d_{1}d_{3}d_{4},\underset{25}{0},d_{2}d_{3}d_{4},\underset{27}{0}%
,...,\underset{39}{0},d_{1}d_{2}d_{3}d_{4},\underset{41}{0},...,\underset
{53}{0}%
\end{array}
\right)
\end{align*}
we have
\begin{align}
&  \left(  \gamma_{2143}\overset{\cdot}{\underset{23}{-}}\gamma_{2413}\right)
\left(  d_{1},d_{2},d_{3}\right) \nonumber\\
&  =\mathfrak{m}\left(
\begin{array}
[c]{c}%
\underset{1}{0},d_{1},d_{2},\underset{4}{0},...,\underset{6}{0},-d_{3}%
,\underset{8}{0},...,\underset{16}{0},d_{1}d_{3},-d_{1}d_{3},\underset{19}%
{0},\underset{20}{0},d_{2}d_{3},\underset{22}{0},\underset{23}{0},\\
-d_{2}d_{3},\underset{25}{0},...,\underset{36}{0},d_{1}d_{2}d_{3}%
,\underset{38}{0},\underset{39}{0},-d_{1}d_{2}d_{3},\underset{41}%
{0},...,\underset{53}{0}%
\end{array}
\right)  \label{t3.2.30}%
\end{align}
(\ref{t3.2.29}) and (\ref{t3.2.30}) imply that
\begin{align}
&  \left(  \left(  \gamma_{2314}\overset{\cdot}{\underset{23}{-}}\gamma
_{2341}\right)  \overset{\cdot}{\underset{1}{-}}\left(  \gamma_{2143}%
\overset{\cdot}{\underset{23}{-}}\gamma_{2413}\right)  \right)  \left(
d_{1},d_{2}\right) \nonumber\\
&  =\mathfrak{m}\left(  \underset{1}{0},d_{1},\underset{3}{0},...,\underset
{20}{0},-d_{2},d_{2},-d_{2},d_{2},\underset{25}{0},...,\underset{36}{0}%
,-d_{1}d_{2},d_{1}d_{2},-d_{1}d_{2},d_{1}d_{2},\underset{41}{0},...,\underset
{53}{0}\right)  \label{t3.2.31}%
\end{align}
Since
\begin{align*}
&  \gamma_{3142}\left(  d_{1},d_{2},d_{3},d_{4}\right) \\
&  =\mathfrak{m}\left(
\begin{array}
[c]{c}%
d_{1},d_{2},d_{3},d_{4},\underset{5}{0},d_{1}d_{3},\underset{7}{0},d_{2}%
d_{3},d_{2}d_{4},\underset{10}{0},...,\underset{13}{0},d_{1}d_{2}d_{3}%
,0,d_{1}d_{2}d_{4},\underset{17}{0},...,\underset{21}{0},d_{1}d_{3}d_{4},\\
\underset{23}{0},...,\underset{27}{0},d_{2}d_{3}d_{4},\underset{29}%
{0},...,\underset{42}{0},d_{1}d_{2}d_{3}d_{4},\underset{44}{0},...,\underset
{53}{0}%
\end{array}
\right)
\end{align*}
and
\begin{align*}
&  \gamma_{3412}\left(  d_{1},d_{2},d_{3},d_{4}\right) \\
&  =\mathfrak{m}\left(
\begin{array}
[c]{c}%
d_{1},d_{2},d_{3},d_{4},\underset{5}{0},d_{1}d_{3},d_{1}d_{4},d_{2}d_{3}%
,d_{2}d_{4},\underset{10}{0},...,\underset{13}{0},d_{1}d_{2}d_{3}%
,\underset{15}{0},...,\underset{18}{0},d_{1}d_{2}d_{4},\\
\underset{20}{0},...,\underset{22}{0},d_{1}d_{3}d_{4},\underset{24}%
{0},...,\underset{27}{0},d_{2}d_{3}d_{4},\underset{29}{0},...,\underset{45}%
{0},d_{1}d_{2}d_{3}d_{4},\underset{47}{0},...,\underset{53}{0}%
\end{array}
\right)
\end{align*}
we have
\begin{align}
&  \left(  \gamma_{3142}\overset{\cdot}{\underset{23}{-}}\gamma_{3412}\right)
\left(  d_{1},d_{2},d_{3}\right) \nonumber\\
&  =\mathfrak{m}\left(
\begin{array}
[c]{c}%
\underset{1}{0},d_{1},d_{2},\underset{4}{0},...,\underset{6}{0},-d_{3}%
,\underset{8}{0},...,\underset{15}{0},d_{1}d_{3},\underset{17}{0}%
,\underset{18}{0},-d_{1}d_{3},\underset{20}{0},\underset{21}{0},d_{2}%
d_{3},-d_{2}d_{3},\\
\underset{24}{0},...,\underset{42}{0},d_{1}d_{2}d_{3},\underset{44}%
{0},\underset{45}{0},-d_{1}d_{2}d_{3},\underset{47}{0},...,\underset{53}{0}%
\end{array}
\right)  \label{t3.2.32}%
\end{align}
Since
\begin{align*}
&  \gamma_{1432}\left(  d_{1},d_{2},d_{3},d_{4}\right) \\
&  =\mathfrak{m}\left(
\begin{array}
[c]{c}%
d_{1},d_{2},d_{3},d_{4},\underset{5}{0},...,\underset{7}{0},d_{2}d_{3}%
,d_{2}d_{4},d_{3}d_{4},d_{1}d_{2}d_{3},\underset{12}{0},...,\underset{15}%
{0},d_{1}d_{2}d_{4},\underset{17}{0},...,\underset{20}{0},\\
d_{1}d_{3}d_{4},\underset{22}{0},...,\underset{29}{0},d_{2}d_{3}%
d_{4},\underset{31}{0},...,\underset{34}{0},d_{1}d_{2}d_{3}d_{4},\underset
{36}{0},...,\underset{53}{0}%
\end{array}
\right)
\end{align*}
and
\begin{align*}
&  \gamma_{4132}\left(  d_{1},d_{2},d_{3},d_{4}\right) \\
&  =\mathfrak{m}\left(
\begin{array}
[c]{c}%
d_{1},d_{2},d_{3},d_{4},\underset{5}{0},\underset{6}{0},d_{1}d_{4},d_{2}%
d_{3},d_{2}d_{4},d_{3}d_{4},d_{1}d_{2}d_{3},\underset{12}{0},...,\underset
{18}{0},d_{1}d_{2}d_{4},\underset{20}{0},...,\underset{23}{0},\\
d_{1}d_{3}d_{4},\underset{25}{0},...,\underset{29}{0},d_{2}d_{3}%
d_{4},\underset{31}{0},...,\underset{48}{0},d_{1}d_{2}d_{3}d_{4},\underset
{50}{0},...,\underset{53}{0}%
\end{array}
\right)
\end{align*}
we have
\begin{align}
&  \left(  \gamma_{1432}\overset{\cdot}{\underset{23}{-}}\gamma_{4132}\right)
\left(  d_{1},d_{2},d_{3}\right) \nonumber\\
&  =\mathfrak{m}\left(
\begin{array}
[c]{c}%
\underset{1}{0},d_{1},d_{2},\underset{4}{0},...,\underset{6}{0},-d_{3}%
,\underset{8}{0},...,\underset{15}{0},d_{1}d_{3},\underset{17}{0}%
,\underset{18}{0},-d_{1}d_{3},\underset{20}{0},d_{2}d_{3},\underset{22}%
{0},\underset{23}{0},\\
-d_{2}d_{3},\underset{25}{0},...,\underset{34}{0},d_{1}d_{2}d_{3}%
,\underset{36}{0},...,\underset{48}{0},-d_{1}d_{2}d_{3},\underset{50}%
{0},...,\underset{53}{0}%
\end{array}
\right)  \label{t3.2.33}%
\end{align}
(\ref{t3.2.32}) and (\ref{t3.2.33}) imply that
\begin{align}
&  \left(  \left(  \gamma_{3142}\overset{\cdot}{\underset{23}{-}}\gamma
_{3412}\right)  \overset{\cdot}{\underset{1}{-}}\left(  \gamma_{1432}%
\overset{\cdot}{\underset{23}{-}}\gamma_{4132}\right)  \right)  \left(
d_{1},d_{2}\right) \nonumber\\
&  =\mathfrak{m}\left(
\begin{array}
[c]{c}%
\underset{1}{0},d_{1},\underset{3}{0},...,\underset{20}{0},-d_{2},d_{2}%
,-d_{2},d_{2},\underset{25}{0},...,\underset{34}{0},-d_{1}d_{2},\underset
{36}{0},...,\underset{42}{0},\\
d_{1}d_{2},\underset{44}{0},\underset{45}{0},-d_{1}d_{2},\underset{47}%
{0},\underset{48}{0},d_{1}d_{2},\underset{50}{0},...,\underset{53}{0}%
\end{array}
\right)  \label{t3.2.34}%
\end{align}
(\ref{t3.2.31}) and (\ref{t3.2.34}) imply that
\begin{align}
&  \left(
\begin{array}
[c]{c}%
\left(  \left(  \gamma_{2314}\overset{\cdot}{\underset{23}{-}}\gamma
_{2341}\right)  \overset{\cdot}{\underset{1}{-}}\left(  \gamma_{2143}%
\overset{\cdot}{\underset{23}{-}}\gamma_{2413}\right)  \right)  \overset
{\cdot}{-}\\
\left(  \left(  \gamma_{3142}\overset{\cdot}{\underset{23}{-}}\gamma
_{3412}\right)  \overset{\cdot}{\underset{1}{-}}\left(  \gamma_{1432}%
\overset{\cdot}{\underset{23}{-}}\gamma_{4132}\right)  \right)
\end{array}
\right)  \left(  d\right) \nonumber\\
&  =\mathfrak{m}\left(  \underset{1}{0},...,\underset{34}{0},d,\underset
{36}{0},-d,d,-d,d,\underset{41}{0},\underset{42}{0},-d,\underset{44}%
{0},\underset{45}{0},d,\underset{47}{0},\underset{48}{0},-d,\underset{50}%
{0},...,\underset{53}{0}\right)  \label{t3.2.35}%
\end{align}

\item Since
\begin{align*}
&  \gamma_{2431}\left(  d_{1},d_{2},d_{3},d_{4}\right) \\
&  =\mathfrak{m}\left(
\begin{array}
[c]{c}%
d_{1},d_{2},d_{3},d_{4},d_{1}d_{2},d_{1}d_{3},d_{1}d_{4},\underset{8}%
{0},\underset{9}{0},d_{3}d_{4},\underset{11}{0},\underset{12}{0},d_{1}%
d_{2}d_{3},\underset{14}{0},...,\underset{17}{0},d_{1}d_{2}d_{4},\\
\underset{19}{0},...,\underset{24}{0},d_{1}d_{3}d_{4},d_{2}d_{3}%
d_{4},\underset{27}{0},...,\underset{40}{0},d_{1}d_{2}d_{3}d_{4},\underset
{42}{0},...,\underset{53}{0}%
\end{array}
\right)
\end{align*}
and
\begin{align*}
&  \gamma_{2413}\left(  d_{1},d_{2},d_{3},d_{4}\right) \\
&  =\mathfrak{m}\left(
\begin{array}
[c]{c}%
d_{1},d_{2},d_{3},d_{4},d_{1}d_{2},\underset{6}{0},d_{1}d_{4},\underset{8}%
{0},\underset{9}{0},d_{3}d_{4},\underset{11}{0},d_{1}d_{2}d_{3},\underset
{13}{0},...,\underset{17}{0},d_{1}d_{2}d_{4},\underset{19}{0},...,\underset
{23}{0},\\
d_{1}d_{3}d_{4},0,d_{2}d_{3}d_{4},\underset{27}{0},...,\underset{39}{0}%
,d_{1}d_{2}d_{3}d_{4},\underset{41}{0},...,\underset{53}{0}%
\end{array}
\right)
\end{align*}
we have
\begin{align}
&  \left(  \gamma_{2431}\overset{\cdot}{\underset{24}{-}}\gamma_{2413}\right)
\left(  d_{1},d_{2},d_{3}\right) \nonumber\\
&  =\mathfrak{m}\left(
\begin{array}
[c]{c}%
\underset{1}{0},d_{1},\underset{3}{0},d_{2},\underset{5}{0},d_{3},\underset
{7}{0},...,\underset{11}{0},-d_{1}d_{3},d_{1}d_{3},\underset{14}%
{0},...,\underset{23}{0},\\
-d_{2}d_{3},d_{2}d_{3},\underset{26}{0},...,\underset{39}{0},-d_{1}d_{2}%
d_{3},d_{1}d_{2}d_{3},\underset{42}{0},...,\underset{53}{0}%
\end{array}
\right)  \label{t3.2.36}%
\end{align}
Since
\begin{align*}
&  \gamma_{2314}\left(  d_{1},d_{2},d_{3},d_{4}\right) \\
&  =\mathfrak{m}\left(
\begin{array}
[c]{c}%
d_{1},d_{2},d_{3},d_{4},d_{1}d_{2},d_{1}d_{3},\underset{7}{0},...,\underset
{12}{0},d_{1}d_{2}d_{3},0,0,0,d_{1}d_{2}d_{4},\underset{18}{0},...,\underset
{21}{0},d_{1}d_{3}d_{4},\\
\underset{23}{0},...,\underset{37}{0},d_{1}d_{2}d_{3}d_{4},\underset{39}%
{0},...,\underset{53}{0}%
\end{array}
\right)
\end{align*}
and
\begin{align*}
&  \gamma_{2134}\left(  d_{1},d_{2},d_{3},d_{4}\right) \\
&  =\left(  d_{1},d_{2},d_{3},d_{4},d_{1}d_{2},\underset{6}{0},...,\underset
{11}{0},d_{1}d_{2}d_{3},\underset{13}{0},...,\underset{16}{0},d_{1}d_{2}%
d_{4},\underset{18}{0},...,\underset{35}{0},d_{1}d_{2}d_{3}d_{4},\underset
{37}{0},...,\underset{53}{0}\right)
\end{align*}
we have
\begin{align}
&  \left(  \gamma_{2314}\overset{\cdot}{\underset{24}{-}}\gamma_{2134}\right)
\left(  d_{1},d_{2},d_{3}\right) \nonumber\\
&  =\mathfrak{m}\left(
\begin{array}
[c]{c}%
\underset{1}{0},d_{1},\underset{3}{0},d_{2},\underset{5}{0},d_{3},\underset
{7}{0},...,\underset{11}{0},-d_{1}d_{3},d_{1}d_{3},\underset{14}%
{0},...,\underset{21}{0},d_{2}d_{3},\\
\underset{23}{0},...,\underset{35}{0},-d_{1}d_{2}d_{3},\underset{37}{0}%
,d_{1}d_{2}d_{3},\underset{39}{0},...,\underset{53}{0}%
\end{array}
\right)  \label{t3.2.37}%
\end{align}
(\ref{t3.2.36}) and (\ref{t3.2.37}) imply that
\begin{align}
&  \left(  \left(  \gamma_{2431}\overset{\cdot}{\underset{24}{-}}\gamma
_{2413}\right)  \overset{\cdot}{\underset{1}{-}}\left(  \gamma_{2314}%
\overset{\cdot}{\underset{24}{-}}\gamma_{2134}\right)  \right)  \left(
d_{1},d_{2}\right) \nonumber\\
&  =\mathfrak{m}\left(  \underset{1}{0},d_{1},\underset{3}{0},...,\underset
{21}{0},-d_{2},\underset{23}{0},-d_{2},d_{2},\underset{26}{0},...,\underset
{35}{0},d_{1}d_{2},\underset{37}{0},-d_{1}d_{2},\underset{39}{0},-d_{1}%
d_{2},d_{1}d_{2},\underset{42}{0},...,\underset{53}{0}\right)  \label{t3.2.38}%
\end{align}
Since
\begin{align*}
&  \gamma_{4312}\left(  d_{1},d_{2},d_{3},d_{4}\right) \\
&  =\mathfrak{m}\left(
\begin{array}
[c]{c}%
d_{1},d_{2},d_{3},d_{4},\underset{5}{0},d_{1}d_{3},d_{1}d_{4},d_{2}d_{3}%
,d_{2}d_{4},d_{3}d_{4},\underset{11}{0},...,\underset{13}{0},d_{1}d_{2}%
d_{3},\underset{15}{0},...,\underset{18}{0},\\
d_{1}d_{2}d_{4},\underset{20}{0},...,\underset{24}{0},d_{1}d_{3}%
d_{4},\underset{26}{0},...,\underset{29}{0},d_{2}d_{3}d_{4},\underset{31}%
{0},...,\underset{51}{0},d_{1}d_{2}d_{3}d_{4},\underset{53}{0}%
\end{array}
\right)
\end{align*}
and
\begin{align*}
&  \gamma_{4132}\left(  d_{1},d_{2},d_{3},d_{4}\right) \\
&  =\mathfrak{m}\left(
\begin{array}
[c]{c}%
d_{1},d_{2},d_{3},d_{4},\underset{5}{0},\underset{6}{0},d_{1}d_{4},d_{2}%
d_{3},d_{2}d_{4},d_{3}d_{4},d_{1}d_{2}d_{3},\underset{12}{0},...,\underset
{18}{0},d_{1}d_{2}d_{4},\underset{20}{0},...,\underset{23}{0},\\
d_{1}d_{3}d_{4},\underset{25}{0},...,\underset{29}{0},d_{2}d_{3}%
d_{4},\underset{31}{0},...,\underset{48}{0},d_{1}d_{2}d_{3}d_{4},\underset
{50}{0},...,\underset{53}{0}%
\end{array}
\right)
\end{align*}
we have
\begin{align}
&  \left(  \gamma_{4312}\overset{\cdot}{\underset{24}{-}}\gamma_{4132}\right)
\left(  d_{1},d_{2},d_{3}\right) \nonumber\\
&  =\mathfrak{m}\left(
\begin{array}
[c]{c}%
\underset{1}{0},d_{1},\underset{3}{0},d_{2},\underset{5}{0},d_{3},\underset
{7}{0},...,\underset{10}{0},-d_{1}d_{3},\underset{12}{0},\underset{13}%
{0},d_{1}d_{3},\underset{15}{0},...,\underset{23}{0},\\
-d_{2}d_{3},d_{2}d_{3},\underset{26}{0},...,\underset{48}{0},-d_{1}d_{2}%
d_{3},\underset{50}{0},\underset{51}{0},d_{1}d_{2}d_{3},\underset{53}{0}%
\end{array}
\right)  \label{t3.2.39}%
\end{align}
Since
\begin{align*}
&  \gamma_{3142}\left(  d_{1},d_{2},d_{3},d_{4}\right) \\
&  =\mathfrak{m}\left(
\begin{array}
[c]{c}%
d_{1},d_{2},d_{3},d_{4},\underset{5}{0},d_{1}d_{3},\underset{7}{0},d_{2}%
d_{3},d_{2}d_{4},\underset{10}{0},...,\underset{13}{0},d_{1}d_{2}d_{3}%
,0,d_{1}d_{2}d_{4},\underset{17}{0},...,\underset{21}{0},d_{1}d_{3}d_{4},\\
\underset{23}{0},...,\underset{27}{0},d_{2}d_{3}d_{4},\underset{29}%
{0},...,\underset{42}{0},d_{1}d_{2}d_{3}d_{4},\underset{44}{0},...,\underset
{53}{0}%
\end{array}
\right)
\end{align*}
and
\begin{align*}
&  \gamma_{1342}\left(  d_{1},d_{2},d_{3},d_{4}\right) \\
&  =\mathfrak{m}\left(
\begin{array}
[c]{c}%
d_{1},d_{2},d_{3},d_{4},\underset{5}{0},...,\underset{7}{0},d_{2}d_{3}%
,d_{2}d_{4},\underset{10}{0},d_{1}d_{2}d_{3},\underset{12}{0},...,\underset
{15}{0},d_{1}d_{2}d_{4},\underset{17}{0},...,\underset{27}{0},\\
d_{2}d_{3}d_{4},0,0,0,0,d_{1}d_{2}d_{3}d_{4},\underset{34}{0},...,\underset
{53}{0}%
\end{array}
\right)
\end{align*}
we have
\begin{align}
&  \left(  \gamma_{3142}\overset{\cdot}{\underset{24}{-}}\gamma_{1342}\right)
\left(  d_{1},d_{2},d_{3}\right) \nonumber\\
&  =\mathfrak{m}\left(
\begin{array}
[c]{c}%
\underset{1}{0},d_{1},\underset{3}{0},d_{2},\underset{5}{0},d_{3},\underset
{7}{0},...,\underset{10}{0},-d_{1}d_{3},\underset{12}{0},\underset{13}%
{0},d_{1}d_{3},\underset{15}{0},...,\underset{21}{0},d_{2}d_{3},\\
\underset{23}{0},...,\underset{32}{0},-d_{1}d_{2}d_{3},\underset{34}%
{0},...,\underset{42}{0},d_{1}d_{2}d_{3},\underset{44}{0},...,\underset{53}{0}%
\end{array}
\right)  \label{t3.2.40}%
\end{align}
(\ref{t3.2.39}) and (\ref{t3.2.40}) imply that
\begin{align}
&  \left(  \left(  \gamma_{4312}\overset{\cdot}{\underset{24}{-}}\gamma
_{4132}\right)  \overset{\cdot}{\underset{1}{-}}\left(  \gamma_{3142}%
\overset{\cdot}{\underset{24}{-}}\gamma_{1342}\right)  \right)  \left(
d_{1},d_{2}\right) \nonumber\\
&  =\mathfrak{m}\left(
\begin{array}
[c]{c}%
\underset{1}{0},d_{1},\underset{3}{0},...,\underset{21}{0},-d_{2}%
,\underset{23}{0},-d_{2},d_{2},\underset{26}{0},...,\underset{32}{0}%
,d_{1}d_{2},\underset{34}{0},...,\underset{42}{0},\\
-d_{1}d_{2},\underset{44}{0},...,\underset{48}{0},-d_{1}d_{2},\underset{50}%
{0},\underset{51}{0},d_{1}d_{2},\underset{53}{0}%
\end{array}
\right)  \label{t3.2.41}%
\end{align}
(\ref{t3.2.38}) and (\ref{t3.2.41}) imply that
\begin{align}
&  \left(
\begin{array}
[c]{c}%
\left(  \left(  \gamma_{2431}\overset{\cdot}{\underset{24}{-}}\gamma
_{2413}\right)  \overset{\cdot}{\underset{1}{-}}\left(  \gamma_{2314}%
\overset{\cdot}{\underset{24}{-}}\gamma_{2134}\right)  \right)  \overset
{\cdot}{-}\\
\left(  \left(  \gamma_{4312}\overset{\cdot}{\underset{24}{-}}\gamma
_{4132}\right)  \overset{\cdot}{\underset{1}{-}}\left(  \gamma_{3142}%
\overset{\cdot}{\underset{24}{-}}\gamma_{1342}\right)  \right)
\end{array}
\right)  \left(  d\right) \nonumber\\
&  =\mathfrak{m}\left(  \underset{1}{0},...,\underset{32}{0},-d,\underset
{34}{0},\underset{35}{0},d,\underset{37}{0},-d,\underset{39}{0},-d,d,\underset
{42}{0},d,\underset{44}{0},...,\underset{48}{0},d,\underset{50}{0}%
,\underset{51}{0},-d,\underset{53}{0}\right)  \label{t3.2.42}%
\end{align}

\item Since
\begin{align*}
&  \gamma_{3124}\left(  d_{1},d_{2},d_{3},d_{4}\right) \\
&  =\mathfrak{m}\left(
\begin{array}
[c]{c}%
d_{1},d_{2},d_{3},d_{4},\underset{5}{0},d_{1}d_{3},\underset{7}{0},d_{2}%
d_{3},\underset{9}{0},...,\underset{13}{0},d_{1}d_{2}d_{3},\underset{15}%
{0},...,\underset{21}{0},d_{1}d_{3}d_{4},\underset{23}{0},...,\underset{26}%
{0},\\
d_{2}d_{3}d_{4},\underset{28}{0},...,\underset{41}{0},d_{1}d_{2}d_{3}%
d_{4},\underset{43}{0},...,\underset{53}{0}%
\end{array}
\right)
\end{align*}
and
\begin{align*}
&  \gamma_{3142}\left(  d_{1},d_{2},d_{3},d_{4}\right) \\
&  =\mathfrak{m}\left(
\begin{array}
[c]{c}%
d_{1},d_{2},d_{3},d_{4},\underset{5}{0},d_{1}d_{3},\underset{7}{0},d_{2}%
d_{3},d_{2}d_{4},\underset{10}{0},...,\underset{13}{0},d_{1}d_{2}d_{3}%
,0,d_{1}d_{2}d_{4},\underset{17}{0},...,\underset{21}{0},\\
d_{1}d_{3}d_{4},\underset{23}{0},...,\underset{27}{0},d_{2}d_{3}%
d_{4},\underset{29}{0},...,\underset{42}{0},d_{1}d_{2}d_{3}d_{4},\underset
{44}{0},...,\underset{53}{0}%
\end{array}
\right)
\end{align*}
we have
\begin{align}
&  \left(  \gamma_{3124}\overset{\cdot}{\underset{31}{-}}\gamma_{3142}\right)
\left(  d_{1},d_{2},d_{3}\right) \nonumber\\
&  =\mathfrak{m}\left(
\begin{array}
[c]{c}%
d_{2},\underset{2}{0},d_{1},\underset{4}{0},...,\underset{8}{0},-d_{3}%
,\underset{10}{0},...,\underset{15}{0},-d_{2}d_{3},\underset{17}%
{0},...,\underset{26}{0},d_{1}d_{3},-d_{1}d_{3},\\
\underset{29}{0},...,\underset{41}{0},d_{1}d_{2}d_{3},-d_{1}d_{2}%
d_{3},\underset{44}{0},...,\underset{53}{0}%
\end{array}
\right)  \label{t3.2.43}%
\end{align}
Since
\begin{align*}
&  \gamma_{3241}\left(  d_{1},d_{2},d_{3},d_{4}\right) \\
&  =\mathfrak{m}\left(
\begin{array}
[c]{c}%
d_{1},d_{2},d_{3},d_{4},d_{1}d_{2},d_{1}d_{3},d_{1}d_{4},d_{2}d_{3}%
,\underset{9}{0},...,\underset{14}{0},d_{1}d_{2}d_{3},\underset{16}%
{0},\underset{17}{0},d_{1}d_{2}d_{4},\underset{19}{0},...,\underset{22}{0},\\
d_{1}d_{3}d_{4},\underset{24}{0},...,\underset{26}{0},d_{2}d_{3}%
d_{4},\underset{28}{0},...,\underset{44}{0},d_{1}d_{2}d_{3}d_{4},\underset
{46}{0},...,\underset{53}{0}%
\end{array}
\right)
\end{align*}
and
\begin{align*}
&  \gamma_{3421}\left(  d_{1},d_{2},d_{3},d_{4}\right) \\
&  =\mathfrak{m}\left(
\begin{array}
[c]{c}%
d_{1},d_{2},d_{3},d_{4},d_{1}d_{2},d_{1}d_{3},d_{1}d_{4},d_{2}d_{3},d_{2}%
d_{4},\underset{10}{0},...,\underset{14}{0},d_{1}d_{2}d_{3},\underset{16}%
{0},...,\underset{19}{0},d_{1}d_{2}d_{4},\\
\underset{21}{0},0,d_{1}d_{3}d_{4},\underset{24}{0},...,\underset{27}{0}%
,d_{2}d_{3}d_{4},\underset{29}{0},...,\underset{46}{0},d_{1}d_{2}d_{3}%
d_{4},\underset{48}{0},...,\underset{53}{0}%
\end{array}
\right)
\end{align*}
we have
\begin{align}
&  \left(  \gamma_{3241}\overset{\cdot}{\underset{31}{-}}\gamma_{3421}\right)
\left(  d_{1},d_{2},d_{3}\right) \nonumber\\
&  =\mathfrak{m}\left(
\begin{array}
[c]{c}%
d_{2},\underset{2}{0},d_{1},\underset{4}{0},...,\underset{8}{0},-d_{3}%
,\underset{10}{0},...,\underset{17}{0},d_{2}d_{3},\underset{19}{0},-d_{2}%
d_{3},\underset{21}{0},...,\underset{26}{0},\\
d_{1}d_{3},-d_{1}d_{3},\underset{29}{0},...,\underset{44}{0},d_{1}d_{2}%
d_{3},\underset{46}{0},-d_{1}d_{2}d_{3},\underset{48}{0},...,\underset{53}{0}%
\end{array}
\right)  \label{t3.2.44}%
\end{align}
(\ref{t3.2.43}) and (\ref{t3.2.44}) imply that
\begin{align}
&  \left(  \left(  \gamma_{3124}\overset{\cdot}{\underset{31}{-}}\gamma
_{3142}\right)  \overset{\cdot}{\underset{1}{-}}\left(  \gamma_{3241}%
\overset{\cdot}{\underset{31}{-}}\gamma_{3421}\right)  \right)  \left(
d_{1},d_{2}\right) \nonumber\\
&  =\mathfrak{m}\left(
\begin{array}
[c]{c}%
\underset{1}{0},\underset{2}{0},d_{1},\underset{4}{0},...,\underset{15}%
{0},-d_{2},\underset{17}{0},-d_{2},\underset{19}{0},d_{2},\underset{21}%
{0},...,\underset{41}{0},\\
d_{1}d_{2},-d_{1}d_{2},\underset{44}{0},-d_{1}d_{2},\underset{46}{0}%
,d_{1}d_{2},\underset{48}{0},...,\underset{53}{0}%
\end{array}
\right)  \label{t3.2.45}%
\end{align}
Since
\begin{align*}
&  \gamma_{1243}\left(  d_{1},d_{2},d_{3},d_{4}\right) \\
&  =\mathfrak{m}\left(
\begin{array}
[c]{c}%
d_{1},d_{2},d_{3},d_{4},\underset{5}{0},...,\underset{9}{0},\underset
{10}{d_{3}d_{4}},\underset{11}{0},...,\underset{20}{0},d_{1}d_{3}%
d_{4},\underset{22}{0},...,\underset{25}{0},d_{2}d_{3}d_{4},\underset{27}%
{0},...,\underset{30}{0},\\
d_{1}d_{2}d_{3}d_{4},\underset{32}{0},...,\underset{53}{0}%
\end{array}
\right)
\end{align*}
and
\begin{align*}
&  \gamma_{1423}\left(  d_{1},d_{2},d_{3},d_{4}\right) \\
&  =\mathfrak{m}\left(
\begin{array}
[c]{c}%
d_{1},d_{2},d_{3},d_{4},\underset{5}{0},...,\underset{8}{0},d_{2}d_{4}%
,d_{3}d_{4},\underset{11}{0},...,\underset{15}{0},d_{1}d_{2}d_{4}%
,\underset{17}{0},...,\underset{20}{0},d_{1}d_{3}d_{4},\\
\underset{22}{0},...,\underset{28}{0},d_{2}d_{3}d_{4},\underset{30}%
{0},...,\underset{33}{0},d_{1}d_{2}d_{3}d_{4},\underset{35}{0},...,\underset
{53}{0}%
\end{array}
\right)
\end{align*}
we have
\begin{align}
&  \left(  \gamma_{1243}\overset{\cdot}{\underset{31}{-}}\gamma_{1423}\right)
\left(  d_{1},d_{2},d_{3}\right) \nonumber\\
&  =\mathfrak{m}\left(
\begin{array}
[c]{c}%
d_{2},\underset{2}{0},d_{1},\underset{4}{0},...,\underset{8}{0},-d_{3}%
,\underset{10}{0},...,\underset{15}{0},-d_{2}d_{3},\underset{17}%
{0},...,\underset{25}{0},d_{1}d_{3},\\
\underset{27}{0},\underset{28}{0},-d_{1}d_{3},\underset{30}{0},d_{1}d_{2}%
d_{3},\underset{32}{0},\underset{33}{0},-d_{1}d_{2}d_{3},\underset{35}%
{0},...,\underset{53}{0}%
\end{array}
\right)  \label{t3.2.46}%
\end{align}
Since
\begin{align*}
&  \gamma_{2413}\left(  d_{1},d_{2},d_{3},d_{4}\right) \\
&  =\mathfrak{m}\left(
\begin{array}
[c]{c}%
d_{1},d_{2},d_{3},d_{4},d_{1}d_{2},\underset{6}{0},d_{1}d_{4},\underset{8}%
{0},\underset{9}{0},d_{3}d_{4},\underset{11}{0},d_{1}d_{2}d_{3},\underset
{13}{0},...,\underset{17}{0},d_{1}d_{2}d_{4},\underset{19}{0},...,\underset
{23}{0},\\
d_{1}d_{3}d_{4},0,d_{2}d_{3}d_{4},\underset{27}{0},...,\underset{39}{0}%
,d_{1}d_{2}d_{3}d_{4},\underset{41}{0},...,\underset{53}{0}%
\end{array}
\right)
\end{align*}
and
\begin{align*}
&  \gamma_{4213}\left(  d_{1},d_{2},d_{3},d_{4}\right) \\
&  =\mathfrak{m}\left(
\begin{array}
[c]{c}%
d_{1},d_{2},d_{3},d_{4},d_{1}d_{2},\underset{6}{0},d_{1}d_{4},\underset{8}%
{0},d_{2}d_{4},d_{3}d_{4},\underset{11}{0},d_{1}d_{2}d_{3},\underset{13}%
{0},...,\underset{19}{0},d_{1}d_{2}d_{4},\\
\underset{21}{0},...,\underset{23}{0},d_{1}d_{3}d_{4},\underset{25}%
{0},...,\underset{28}{0},d_{2}d_{3}d_{4},\underset{30}{0},...,\underset{49}%
{0},d_{1}d_{2}d_{3}d_{4},\underset{51}{0},...,\underset{53}{0}%
\end{array}
\right)
\end{align*}
we have
\begin{align}
&  \left(  \gamma_{2413}\overset{\cdot}{\underset{31}{-}}\gamma_{4213}\right)
\left(  d_{1},d_{2},d_{3}\right) \nonumber\\
&  =\mathfrak{m}\left(
\begin{array}
[c]{c}%
d_{2},\underset{2}{0},d_{1},\underset{4}{0},...,\underset{8}{0},-d_{3}%
,\underset{10}{0},...,\underset{17}{0},d_{2}d_{3},\underset{19}{0},-d_{2}%
d_{3},\underset{21}{0},...,\underset{25}{0},d_{1}d_{3},\\
\underset{27}{0},\underset{28}{0},-d_{1}d_{3},\underset{30}{0},...,\underset
{39}{0},d_{1}d_{2}d_{3},\underset{41}{0},...,\underset{49}{0},-d_{1}d_{2}%
d_{3},\underset{51}{0},...,\underset{53}{0}%
\end{array}
\right)  \label{t3.2.47}%
\end{align}
(\ref{t3.2.46}) and (\ref{t3.2.47}) imply that
\begin{align}
&  \left(  \left(  \gamma_{1243}\overset{\cdot}{\underset{31}{-}}\gamma
_{1423}\right)  \overset{\cdot}{\underset{1}{-}}\left(  \gamma_{2413}%
\overset{\cdot}{\underset{31}{-}}\gamma_{4213}\right)  \right)  \left(
d_{1},d_{2}\right) \nonumber\\
&  =\mathfrak{m}\left(
\begin{array}
[c]{c}%
\underset{1}{0},\underset{2}{0},d_{1},\underset{4}{0},...,\underset{15}%
{0},-d_{2},\underset{17}{0},-d_{2},\underset{19}{0},d_{2},\underset{21}%
{0},...,\underset{30}{0},d_{1}d_{2},\underset{32}{0},\underset{33}{0}%
,-d_{1}d_{2},\\
\underset{35}{0},...,\underset{39}{0},-d_{1}d_{2},\underset{41}{0}%
,...,\underset{49}{0},d_{1}d_{2},\underset{51}{0},...,\underset{53}{0}%
\end{array}
\right)  \label{t3.2.48}%
\end{align}
(\ref{t3.2.45}) and (\ref{t3.2.48}) imply that
\begin{align}
&  \left(
\begin{array}
[c]{c}%
\left(  \left(  \gamma_{3124}\overset{\cdot}{\underset{31}{-}}\gamma
_{3142}\right)  \overset{\cdot}{\underset{1}{-}}\left(  \gamma_{3241}%
\overset{\cdot}{\underset{31}{-}}\gamma_{3421}\right)  \right)  \overset
{\cdot}{-}\\
\left(  \left(  \gamma_{1243}\overset{\cdot}{\underset{31}{-}}\gamma
_{1423}\right)  \overset{\cdot}{\underset{1}{-}}\left(  \gamma_{2413}%
\overset{\cdot}{\underset{31}{-}}\gamma_{4213}\right)  \right)
\end{array}
\right)  \left(  d\right) \nonumber\\
&  =\mathfrak{m}\left(  \underset{1}{0},...,\underset{30}{0},-d,\underset
{32}{0},\underset{33}{0},d,\underset{35}{0},...,\underset{39}{0}%
,d,\underset{41}{0},d,-d,\underset{44}{0},-d,\underset{46}{0},d,\underset
{48}{0},\underset{49}{0},-d,\underset{51}{0},...,\underset{53}{0}\right)
\label{t3.2.49}%
\end{align}

\item Since
\begin{align*}
&  \gamma_{3241}\left(  d_{1},d_{2},d_{3},d_{4}\right) \\
&  =\mathfrak{m}\left(
\begin{array}
[c]{c}%
d_{1},d_{2},d_{3},d_{4},d_{1}d_{2},d_{1}d_{3},d_{1}d_{4},d_{2}d_{3}%
,\underset{9}{0},...,\underset{14}{0},d_{1}d_{2}d_{3},\underset{16}%
{0},\underset{17}{0},d_{1}d_{2}d_{4},\underset{19}{0},...,\underset{22}{0},\\
d_{1}d_{3}d_{4},\underset{24}{0},...,\underset{26}{0},d_{2}d_{3}%
d_{4},\underset{28}{0},...,\underset{44}{0},d_{1}d_{2}d_{3}d_{4},\underset
{46}{0},...,\underset{53}{0}%
\end{array}
\right)
\end{align*}
and
\begin{align*}
&  \gamma_{3214}\left(  d_{1},d_{2},d_{3},d_{4}\right) \\
&  =\mathfrak{m}\left(
\begin{array}
[c]{c}%
d_{1},d_{2},d_{3},d_{4},d_{1}d_{2},d_{1}d_{3},0,d_{2}d_{3},\underset{9}%
{0},...,\underset{14}{0},d_{1}d_{2}d_{3},0,d_{1}d_{2}d_{4},\underset{18}%
{0},...,\underset{21}{0},d_{1}d_{3}d_{4},\\
\underset{23}{0},...,\underset{26}{0},d_{2}d_{3}d_{4},\underset{28}%
{0},...,\underset{43}{0},d_{1}d_{2}d_{3}d_{4},\underset{45}{0},...,\underset
{53}{0}%
\end{array}
\right)
\end{align*}
we have
\begin{align}
&  \left(  \gamma_{3241}\overset{\cdot}{\underset{32}{-}}\gamma_{3214}\right)
\left(  d_{1},d_{2},d_{3}\right) \nonumber\\
&  =\mathfrak{m}\left(
\begin{array}
[c]{c}%
\underset{1}{0},d_{2},d_{1},\underset{4}{0},...,\underset{6}{0},d_{3}%
,\underset{8}{0},...,\underset{16}{0},-d_{2}d_{3},d_{2}d_{3},\underset{19}%
{0},...,\underset{21}{0},-d_{1}d_{3},d_{1}d_{3},\\
\underset{24}{0},...,\underset{43}{0},-d_{1}d_{2}d_{3},d_{1}d_{2}%
d_{3},\underset{46}{0},...,\underset{53}{0}%
\end{array}
\right)  \label{t3.2.50}%
\end{align}
Since
\begin{align*}
&  \gamma_{3412}\left(  d_{1},d_{2},d_{3},d_{4}\right) \\
&  =\mathfrak{m}\left(
\begin{array}
[c]{c}%
d_{1},d_{2},d_{3},d_{4},\underset{5}{0},d_{1}d_{3},d_{1}d_{4},d_{2}d_{3}%
,d_{2}d_{4},\underset{10}{0},...,\underset{13}{0},d_{1}d_{2}d_{3}%
,\underset{15}{0},...,\underset{18}{0},d_{1}d_{2}d_{4},\\
\underset{20}{0},...,\underset{22}{0},d_{1}d_{3}d_{4},\underset{24}%
{0},...,\underset{27}{0},d_{2}d_{3}d_{4},\underset{29}{0},...,\underset{45}%
{0},d_{1}d_{2}d_{3}d_{4},\underset{47}{0},...,\underset{53}{0}%
\end{array}
\right)
\end{align*}
and
\begin{align*}
&  \gamma_{3142}\left(  d_{1},d_{2},d_{3},d_{4}\right) \\
&  =\mathfrak{m}\left(
\begin{array}
[c]{c}%
d_{1},d_{2},d_{3},d_{4},\underset{5}{0},d_{1}d_{3},\underset{7}{0},d_{2}%
d_{3},d_{2}d_{4},\underset{10}{0},...,\underset{13}{0},d_{1}d_{2}d_{3}%
,0,d_{1}d_{2}d_{4},\underset{17}{0},...,\underset{21}{0},d_{1}d_{3}d_{4},\\
\underset{23}{0},...,\underset{27}{0},d_{2}d_{3}d_{4},\underset{29}%
{0},...,\underset{42}{0},d_{1}d_{2}d_{3}d_{4},\underset{44}{0},...,\underset
{53}{0}%
\end{array}
\right)
\end{align*}
we have
\begin{align}
&  \left(  \gamma_{3412}\overset{\cdot}{\underset{32}{-}}\gamma_{3142}\right)
\left(  d_{1},d_{2},d_{3}\right) \nonumber\\
&  =\mathfrak{m}\left(
\begin{array}
[c]{c}%
\underset{1}{0},d_{2},d_{1},\underset{4}{0},...,\underset{6}{0},d_{3}%
,\underset{8}{0},...,\underset{15}{0},-d_{2}d_{3},\underset{17}{0}%
,\underset{18}{0},d_{2}d_{3},\underset{20}{0},\underset{21}{0},\\
-d_{1}d_{3},d_{1}d_{3},\underset{24}{0},...,\underset{42}{0},-d_{1}d_{2}%
d_{3},\underset{44}{0},\underset{45}{0},d_{1}d_{2}d_{3},\underset{47}%
{0},...,\underset{53}{0}%
\end{array}
\right)  \label{t3.2.51}%
\end{align}
(\ref{t3.2.50}) and (\ref{t3.2.51}) imply that
\begin{align}
&  \left(  \left(  \gamma_{3241}\overset{\cdot}{\underset{32}{-}}\gamma
_{3214}\right)  \overset{\cdot}{\underset{1}{-}}\left(  \gamma_{3412}%
\overset{\cdot}{\underset{32}{-}}\gamma_{3142}\right)  \right)  \left(
d_{1},d_{2}\right) \nonumber\\
&  =\mathfrak{m}\left(  \underset{1}{0},\underset{2}{0},d_{1},\underset{4}%
{0},...,\underset{15}{0},d_{2},-d_{2},d_{2},-d_{2},\underset{20}%
{0},...,\underset{42}{0},d_{1}d_{2},-d_{1}d_{2},d_{1}d_{2},-d_{1}%
d_{2},\underset{47}{0},...,\underset{53}{0}\right)  \label{t3.2.52}%
\end{align}
Since
\begin{align*}
&  \gamma_{2413}\left(  d_{1},d_{2},d_{3},d_{4}\right) \\
&  =\mathfrak{m}\left(
\begin{array}
[c]{c}%
d_{1},d_{2},d_{3},d_{4},d_{1}d_{2},\underset{6}{0},d_{1}d_{4},\underset{8}%
{0},\underset{9}{0},d_{3}d_{4},\underset{11}{0},d_{1}d_{2}d_{3},\underset
{13}{0},...,\underset{17}{0},d_{1}d_{2}d_{4},\underset{19}{0},...,\underset
{23}{0},\\
d_{1}d_{3}d_{4},0,d_{2}d_{3}d_{4},\underset{27}{0},...,\underset{39}{0}%
,d_{1}d_{2}d_{3}d_{4},\underset{41}{0},...,\underset{53}{0}%
\end{array}
\right)
\end{align*}
and
\begin{align*}
&  \gamma_{2143}\left(  d_{1},d_{2},d_{3},d_{4}\right) \\
&  =\mathfrak{m}\left(
\begin{array}
[c]{c}%
d_{1},d_{2},d_{3},d_{4},d_{1}d_{2},\underset{6}{0},...,\underset{9}{0}%
,d_{3}d_{4},0,d_{1}d_{2}d_{3},\underset{13}{0},...,\underset{16}{0},d_{1}%
d_{2}d_{4},0,0,0,d_{1}d_{3}d_{4},\\
\underset{22}{0},...,\underset{25}{0},d_{2}d_{3}d_{4},\underset{27}%
{0},...,\underset{36}{0},d_{1}d_{2}d_{3}d_{4},\underset{38}{0},...,\underset
{53}{0}%
\end{array}
\right)
\end{align*}
we have
\begin{align}
&  \left(  \gamma_{2413}\overset{\cdot}{\underset{32}{-}}\gamma_{2143}\right)
\left(  d_{1},d_{2},d_{3}\right) \nonumber\\
&  =\mathfrak{m}\left(
\begin{array}
[c]{c}%
\underset{1}{0},d_{2},d_{1},\underset{4}{0},...,\underset{6}{0},d_{3}%
,\underset{8}{0},...,\underset{16}{0},-d_{2}d_{3},d_{2}d_{3},\underset{19}%
{0},\underset{20}{0},-d_{1}d_{3},\\
\underset{22}{0},\underset{23}{0},d_{1}d_{3},\underset{25}{0},...,\underset
{36}{0},-d_{1}d_{2}d_{3},\underset{38}{0},\underset{39}{0},d_{1}d_{2}%
d_{3},\underset{41}{0},...,\underset{53}{0}%
\end{array}
\right)  \label{t3.2.53}%
\end{align}
Since
\begin{align*}
&  \gamma_{4123}\left(  d_{1},d_{2},d_{3},d_{4}\right) \\
&  =\mathfrak{m}\left(
\begin{array}
[c]{c}%
d_{1},d_{2},d_{3},d_{4},\underset{5}{0},\underset{6}{0},d_{1}d_{4}%
,\underset{8}{0},d_{2}d_{4},d_{3}d_{4},\underset{11}{0},...,\underset{18}%
{0},d_{1}d_{2}d_{4},\underset{20}{0},...,\underset{23}{0},d_{1}d_{3}d_{4},\\
\underset{25}{0},...,\underset{28}{0},d_{2}d_{3}d_{4},\underset{30}%
{0},...,\underset{47}{0},d_{1}d_{2}d_{3}d_{4},\underset{49}{0},...,\underset
{53}{0}%
\end{array}
\right)
\end{align*}
and
\begin{align*}
&  \gamma_{1423}\left(  d_{1},d_{2},d_{3},d_{4}\right) \\
&  =\mathfrak{m}\left(
\begin{array}
[c]{c}%
d_{1},d_{2},d_{3},d_{4},\underset{5}{0},...,\underset{8}{0},d_{2}d_{4}%
,d_{3}d_{4},\underset{11}{0},...,\underset{15}{0},d_{1}d_{2}d_{4}%
,\underset{17}{0},...,\underset{20}{0},d_{1}d_{3}d_{4},\\
\underset{22}{0},...,\underset{28}{0},d_{2}d_{3}d_{4},\underset{30}%
{0},...,\underset{33}{0},d_{1}d_{2}d_{3}d_{4},\underset{35}{0},...,\underset
{53}{0}%
\end{array}
\right)
\end{align*}
we have
\begin{align}
&  \left(  \gamma_{4123}\overset{\cdot}{\underset{32}{-}}\gamma_{1423}\right)
\left(  d_{1},d_{2},d_{3}\right) \nonumber\\
&  =\mathfrak{m}\left(
\begin{array}
[c]{c}%
\underset{1}{0},d_{2},d_{1},\underset{4}{0},...,\underset{6}{0},d_{3}%
,\underset{8}{0},...,\underset{15}{0},-d_{2}d_{3},\underset{17}{0}%
,\underset{18}{0},d_{2}d_{3},\underset{20}{0},-d_{1}d_{3},\\
\underset{22}{0},\underset{23}{0},d_{1}d_{3},\underset{25}{0},...,\underset
{33}{0},-d_{1}d_{2}d_{3},\underset{35}{0},...,\underset{47}{0},d_{1}d_{2}%
d_{3},\underset{49}{0},...,\underset{53}{0}%
\end{array}
\right)  \label{t3.2.54}%
\end{align}
(\ref{3.2.53}) and (\ref{t3.2.54}) imply that
\begin{align}
&  \left(  \left(  \gamma_{2413}\overset{\cdot}{\underset{32}{-}}\gamma
_{2143}\right)  \overset{\cdot}{\underset{1}{-}}\left(  \gamma_{4123}%
\overset{\cdot}{\underset{32}{-}}\gamma_{1423}\right)  \right)  \left(
d_{1},d_{2}\right) \nonumber\\
&  =\mathfrak{m}\left(
\begin{array}
[c]{c}%
\underset{1}{0},\underset{2}{0},d_{1},\underset{4}{0},...,\underset{15}%
{0},d_{2},-d_{2},d_{2},-d_{2},\underset{20}{0},...,\underset{33}{0},d_{1}%
d_{2},\underset{35}{0},\underset{36}{0},-d_{1}d_{2},\\
\underset{38}{0},\underset{39}{0},d_{1}d_{2},\underset{41}{0},...,\underset
{47}{0},-d_{1}d_{2},\underset{49}{0},...,\underset{53}{0}%
\end{array}
\right)  \label{t3.2.55}%
\end{align}
(\ref{t3.2.52}) and (\ref{t3.2.55}) imply that
\begin{align}
&  \left(
\begin{array}
[c]{c}%
\left(  \left(  \gamma_{3241}\overset{\cdot}{\underset{32}{-}}\gamma
_{3214}\right)  \overset{\cdot}{\underset{1}{-}}\left(  \gamma_{3412}%
\overset{\cdot}{\underset{32}{-}}\gamma_{3142}\right)  \right)  \overset
{\cdot}{-}\\
\left(  \left(  \gamma_{2413}\overset{\cdot}{\underset{32}{-}}\gamma
_{2143}\right)  \overset{\cdot}{\underset{1}{-}}\left(  \gamma_{4123}%
\overset{\cdot}{\underset{32}{-}}\gamma_{1423}\right)  \right)
\end{array}
\right)  \left(  d\right) \nonumber\\
&  =\mathfrak{m}\left(  \underset{1}{0},...,\underset{33}{0},-d,\underset
{35}{0},\underset{36}{0},d,\underset{38}{0},\underset{39}{0},-d,\underset
{41}{0},\underset{42}{0},d,-d,d,-d,\underset{47}{0},d,\underset{49}%
{0},...,\underset{53}{0}\right)  \label{t3.2.56}%
\end{align}

\item Since
\begin{align*}
&  \gamma_{3412}\left(  d_{1},d_{2},d_{3},d_{4}\right) \\
&  =\mathfrak{m}\left(
\begin{array}
[c]{c}%
d_{1},d_{2},d_{3},d_{4},\underset{5}{0},d_{1}d_{3},d_{1}d_{4},d_{2}d_{3}%
,d_{2}d_{4},\underset{10}{0},...,\underset{13}{0},d_{1}d_{2}d_{3}%
,\underset{15}{0},...,\underset{18}{0},d_{1}d_{2}d_{4},\\
\underset{20}{0},...,\underset{22}{0},d_{1}d_{3}d_{4},\underset{24}%
{0},...,\underset{27}{0},d_{2}d_{3}d_{4},\underset{29}{0},...,\underset{45}%
{0},d_{1}d_{2}d_{3}d_{4},\underset{47}{0},...,\underset{53}{0}%
\end{array}
\right)
\end{align*}
and
\begin{align*}
&  \gamma_{3421}\left(  d_{1},d_{2},d_{3},d_{4}\right) \\
&  =\mathfrak{m}\left(
\begin{array}
[c]{c}%
d_{1},d_{2},d_{3},d_{4},d_{1}d_{2},d_{1}d_{3},d_{1}d_{4},d_{2}d_{3},d_{2}%
d_{4},\underset{10}{0},...,\underset{14}{0},d_{1}d_{2}d_{3},\underset{16}%
{0},...,\underset{19}{0},d_{1}d_{2}d_{4},\\
\underset{21}{0},\underset{22}{0},d_{1}d_{3}d_{4},\underset{24}{0}%
,...,\underset{27}{0},d_{2}d_{3}d_{4},\underset{29}{0},...,\underset{46}%
{0},d_{1}d_{2}d_{3}d_{4},\underset{48}{0},...,\underset{53}{0}%
\end{array}
\right)
\end{align*}
we have
\begin{align}
&  \left(  \gamma_{3412}\overset{\cdot}{\underset{34}{-}}\gamma_{3421}\right)
\left(  d_{1},d_{2},d_{3}\right) \nonumber\\
&  =\mathfrak{m}\left(
\begin{array}
[c]{c}%
\underset{1}{0},\underset{2}{0},d_{1},d_{2},-d_{3},\underset{6}{0}%
,...,\underset{13}{0},d_{1}d_{3},-d_{1}d_{3},\underset{16}{0},...,\underset
{18}{0},d_{2}d_{3},-d_{2}d_{3},\\
\underset{21}{0},...,\underset{45}{0},d_{1}d_{2}d_{3},-d_{1}d_{2}%
d_{3},\underset{48}{0},...,\underset{53}{0}%
\end{array}
\right)  \label{t3.2.57}%
\end{align}
Since
\begin{align*}
&  \gamma_{3124}\left(  d_{1},d_{2},d_{3},d_{4}\right) \\
&  =\mathfrak{m}\left(
\begin{array}
[c]{c}%
d_{1},d_{2},d_{3},d_{4},\underset{5}{0},d_{1}d_{3},\underset{7}{0},d_{2}%
d_{3},\underset{9}{0},...,\underset{13}{0},d_{1}d_{2}d_{3},\underset{15}%
{0},...,\underset{21}{0},d_{1}d_{3}d_{4},\underset{23}{0},...,\underset{26}%
{0},\\
d_{2}d_{3}d_{4},\underset{28}{0},...,\underset{41}{0},d_{1}d_{2}d_{3}%
d_{4},\underset{43}{0},...,\underset{53}{0}%
\end{array}
\right)
\end{align*}
and
\begin{align*}
&  \gamma_{3214}\left(  d_{1},d_{2},d_{3},d_{4}\right) \\
&  =\mathfrak{m}\left(
\begin{array}
[c]{c}%
d_{1},d_{2},d_{3},d_{4},d_{1}d_{2},d_{1}d_{3},\underset{7}{0},d_{2}%
d_{3},\underset{9}{0},...,\underset{14}{0},d_{1}d_{2}d_{3},0,d_{1}d_{2}%
d_{4},\underset{18}{0},...,\underset{21}{0},d_{1}d_{3}d_{4},\\
\underset{23}{0},...,\underset{26}{0},d_{2}d_{3}d_{4},\underset{28}%
{0},...,\underset{43}{0},d_{1}d_{2}d_{3}d_{4},\underset{45}{0},...,\underset
{53}{0}%
\end{array}
\right)
\end{align*}
we have
\begin{align}
&  \left(  \gamma_{3124}\overset{\cdot}{\underset{34}{-}}\gamma_{3214}\right)
\left(  d_{1},d_{2},d_{3}\right) \nonumber\\
&  =\mathfrak{m}\left(
\begin{array}
[c]{c}%
\underset{1}{0},\underset{2}{0},d_{1},d_{2},-d_{3},\underset{6}{0}%
,...,\underset{13}{0},d_{1}d_{3},-d_{1}d_{3},\underset{16}{0},-d_{2}%
d_{3},\underset{18}{0},...,\underset{41}{0},\\
d_{1}d_{2}d_{3},\underset{43}{0},-d_{1}d_{2}d_{3},\underset{45}{0}%
,...,\underset{53}{0}%
\end{array}
\right)  \label{t3.2.58}%
\end{align}
(\ref{t3.2.57}) and (\ref{t3.2.58}) imply that
\begin{align}
&  \left(  \left(  \gamma_{3412}\overset{\cdot}{\underset{34}{-}}\gamma
_{3421}\right)  \overset{\cdot}{\underset{1}{-}}\left(  \gamma_{3124}%
\overset{\cdot}{\underset{34}{-}}\gamma_{3214}\right)  \right)  \left(
d_{1},d_{2}\right) \nonumber\\
&  =\mathfrak{m}\left(
\begin{array}
[c]{c}%
\underset{1}{0},\underset{2}{0},d_{1},\underset{4}{0},...,\underset{16}%
{0},d_{2},\underset{18}{0},d_{2},-d_{2},\underset{21}{0},...\underset{41}%
{0},-d_{1}d_{2},\underset{43}{0},d_{1}d_{2},\\
\underset{45}{0},d_{1}d_{2},-d_{1}d_{2},\underset{48}{0},...,\underset{53}{0}%
\end{array}
\right)  \label{t3.2.59}%
\end{align}
Since
\begin{align*}
&  \gamma_{4123}\left(  d_{1},d_{2},d_{3},d_{4}\right) \\
&  =\mathfrak{m}\left(
\begin{array}
[c]{c}%
d_{1},d_{2},d_{3},d_{4},\underset{5}{0},\underset{6}{0},d_{1}d_{4}%
,\underset{8}{0},d_{2}d_{4},d_{3}d_{4},\underset{11}{0},...,\underset{18}%
{0},d_{1}d_{2}d_{4},\underset{20}{0},...,\underset{23}{0},d_{1}d_{3}d_{4},\\
\underset{25}{0},...,\underset{28}{0},d_{2}d_{3}d_{4},\underset{30}%
{0},...,\underset{47}{0},d_{1}d_{2}d_{3}d_{4},\underset{49}{0},...,\underset
{53}{0}%
\end{array}
\right)
\end{align*}
and
\begin{align*}
&  \gamma_{4213}\left(  d_{1},d_{2},d_{3},d_{4}\right) \\
&  =\mathfrak{m}\left(
\begin{array}
[c]{c}%
d_{1},d_{2},d_{3},d_{4},d_{1}d_{2},0,d_{1}d_{4},0,d_{2}d_{4},d_{3}%
d_{4},0,d_{1}d_{2}d_{3},\underset{13}{0},...,\underset{19}{0},d_{1}d_{2}%
d_{4},\\
\underset{21}{0},...,\underset{23}{0},d_{1}d_{3}d_{4},\underset{25}%
{0},...,\underset{28}{0},d_{2}d_{3}d_{4},\underset{30}{0},...,\underset{49}%
{0},d_{1}d_{2}d_{3}d_{4},\underset{51}{0},...,\underset{53}{0}%
\end{array}
\right)
\end{align*}
we have
\begin{align}
&  \left(  \gamma_{4123}\overset{\cdot}{\underset{34}{-}}\gamma_{4213}\right)
\left(  d_{1},d_{2},d_{3}\right) \nonumber\\
&  =\mathfrak{m}\left(
\begin{array}
[c]{c}%
\underset{1}{0},\underset{2}{0},d_{1},d_{2},-d_{3},\underset{6}{0}%
,...,\underset{11}{0},-d_{1}d_{3},\underset{13}{0},...,\underset{18}{0}%
,d_{2}d_{3},-d_{2}d_{3},\\
\underset{21}{0},...,\underset{47}{0},d_{1}d_{2}d_{3},\underset{49}{0}%
,-d_{1}d_{2}d_{3},\underset{51}{0},...,\underset{53}{0}%
\end{array}
\right)  \label{t3.2.60}%
\end{align}
Since
\begin{align*}
&  \gamma_{1243}\left(  d_{1},d_{2},d_{3},d_{4}\right) \\
&  =\mathfrak{m}\left(
\begin{array}
[c]{c}%
d_{1},d_{2},d_{3},d_{4},\underset{5}{0},...,\underset{9}{0},\underset
{10}{d_{3}d_{4}},\underset{11}{0},...,\underset{20}{0},d_{1}d_{3}%
d_{4},\underset{22}{0},...,\underset{25}{0},d_{2}d_{3}d_{4},\underset{27}%
{0},...,\underset{30}{0},\\
d_{1}d_{2}d_{3}d_{4},\underset{32}{0},...,\underset{53}{0}%
\end{array}
\right)
\end{align*}
and
\begin{align*}
&  \gamma_{2143}\left(  d_{1},d_{2},d_{3},d_{4}\right) \\
&  =\mathfrak{m}\left(
\begin{array}
[c]{c}%
d_{1},d_{2},d_{3},d_{4},d_{1}d_{2},\underset{6}{0},...,\underset{9}{0}%
,d_{3}d_{4},0,d_{1}d_{2}d_{3},\underset{13}{0},...,\underset{16}{0},d_{1}%
d_{2}d_{4},\underset{18}{0},...,\underset{20}{0},d_{1}d_{3}d_{4},\\
\underset{22}{0},...,\underset{25}{0},d_{2}d_{3}d_{4},\underset{27}%
{0},...,\underset{36}{0},d_{1}d_{2}d_{3}d_{4},\underset{38}{0},...,\underset
{53}{0}%
\end{array}
\right)
\end{align*}
we have
\begin{align}
&  \left(  \gamma_{1243}\overset{\cdot}{\underset{34}{-}}\gamma_{2143}\right)
\left(  d_{1},d_{2},d_{3}\right) \nonumber\\
&  =\mathfrak{m}\left(
\begin{array}
[c]{c}%
\underset{1}{0},\underset{2}{0},d_{1},d_{2},-d_{3},\underset{6}{0}%
,...,\underset{11}{0},-d_{1}d_{3},\underset{13}{0},...,\underset{16}{0}%
,-d_{2}d_{3},\\
\underset{18}{0},...,\underset{30}{0},d_{1}d_{2}d_{3},\underset{32}%
{0},...,\underset{36}{0},-d_{1}d_{2}d_{3},\underset{38}{0},...,\underset
{53}{0}%
\end{array}
\right)  \label{t3.2.61}%
\end{align}
(\ref{t3.2.60}) and (\ref{t3.2.61}) imply that
\begin{align}
&  \left(  \left(  \gamma_{4123}\overset{\cdot}{\underset{34}{-}}\gamma
_{4213}\right)  \overset{\cdot}{\underset{1}{-}}\left(  \gamma_{1243}%
\overset{\cdot}{\underset{34}{-}}\gamma_{2143}\right)  \right)  \left(
d_{1},d_{2}\right) \nonumber\\
&  =\mathfrak{m}\left(
\begin{array}
[c]{c}%
\underset{1}{0},\underset{2}{0},d_{1},\underset{4}{0},...,\underset{16}%
{0},d_{2},\underset{18}{0},d_{2},-d_{2},\underset{21}{0},...,\underset{30}%
{0},-d_{1}d_{2},\underset{32}{0},...,\underset{36}{0},\\
d_{1}d_{2},\underset{38}{0},...,\underset{47}{0},d_{1}d_{2},\underset{49}%
{0},-d_{1}d_{2},\underset{51}{0},...,\underset{53}{0}%
\end{array}
\right)  \label{t3.2.62}%
\end{align}
(\ref{t3.2.59}) and (\ref{t3.2.62}) imply that
\begin{align}
&  \left(
\begin{array}
[c]{c}%
\left(  \left(  \gamma_{3412}\overset{\cdot}{\underset{34}{-}}\gamma
_{3421}\right)  \overset{\cdot}{\underset{1}{-}}\left(  \gamma_{3124}%
\overset{\cdot}{\underset{34}{-}}\gamma_{3214}\right)  \right)  \overset
{\cdot}{-}\\
\left(  \left(  \gamma_{4123}\overset{\cdot}{\underset{34}{-}}\gamma
_{4213}\right)  \overset{\cdot}{\underset{1}{-}}\left(  \gamma_{1243}%
\overset{\cdot}{\underset{34}{-}}\gamma_{2143}\right)  \right)
\end{array}
\right)  \left(  d\right) \nonumber\\
&  =\mathfrak{m}\left(  \underset{1}{0},...,\underset{30}{0},d,\underset
{32}{0},...,\underset{36}{0},-d,\underset{38}{0},...,\underset{41}%
{0},-d,\underset{43}{0},d,\underset{45}{0},d,-d,-d,\underset{49}%
{0},d,\underset{51}{0},...,\underset{53}{0}\right)  \label{t3.2.63}%
\end{align}

\item Since
\begin{align*}
&  \gamma_{4132}\left(  d_{1},d_{2},d_{3},d_{4}\right) \\
&  =\mathfrak{m}\left(
\begin{array}
[c]{c}%
d_{1},d_{2},d_{3},d_{4},\underset{5}{0},\underset{6}{0},d_{1}d_{4},d_{2}%
d_{3},d_{2}d_{4},d_{3}d_{4},d_{1}d_{2}d_{3},\underset{12}{0},...,\underset
{18}{0},d_{1}d_{2}d_{4},\underset{20}{0},...,\underset{23}{0},\\
d_{1}d_{3}d_{4},\underset{25}{0},...,\underset{29}{0},d_{2}d_{3}%
d_{4},\underset{31}{0},...,\underset{48}{0},d_{1}d_{2}d_{3}d_{4},\underset
{50}{0},...,\underset{53}{0}%
\end{array}
\right)
\end{align*}
and
\begin{align*}
&  \gamma_{4123}\left(  d_{1},d_{2},d_{3},d_{4}\right) \\
&  =\mathfrak{m}\left(
\begin{array}
[c]{c}%
d_{1},d_{2},d_{3},d_{4},\underset{5}{0},\underset{6}{0},d_{1}d_{4}%
,\underset{8}{0},d_{2}d_{4},d_{3}d_{4},\underset{11}{0},...,\underset{18}%
{0},d_{1}d_{2}d_{4},\underset{20}{0},...,\underset{23}{0,}d_{1}d_{3}d_{4},\\
\underset{25}{0},...,\underset{28}{0},d_{2}d_{3}d_{4},\underset{30}%
{0},...,\underset{47}{0},d_{1}d_{2}d_{3}d_{4},\underset{49}{0},...,\underset
{53}{0}%
\end{array}
\right)
\end{align*}
we have
\begin{align}
&  \left(  \gamma_{4132}\overset{\cdot}{\underset{41}{-}}\gamma_{4123}\right)
\left(  d_{1},d_{2},d_{3}\right) \nonumber\\
&  =\mathfrak{m}\left(
\begin{array}
[c]{c}%
d_{2},\underset{2}{0},\underset{3}{0},d_{1},\underset{5}{0},...,\underset
{7}{0},d_{3},\underset{9}{0},...,\underset{10}{0},d_{2}d_{3},\underset{12}%
{0},...,\underset{28}{0},-d_{1}d_{3},d_{1}d_{3},\\
\underset{31}{0},...,\underset{47}{0},-d_{1}d_{2}d_{3},d_{1}d_{2}%
d_{3},\underset{50}{0},...,\underset{53}{0}%
\end{array}
\right)  \label{t3.2.64}%
\end{align}
Since
\begin{align*}
&  \gamma_{4321}\left(  d_{1},d_{2},d_{3},d_{4}\right) \\
&  =\mathfrak{m}\left(
\begin{array}
[c]{c}%
d_{1},d_{2},d_{3},d_{4},d_{1}d_{2},d_{1}d_{3},d_{1}d_{4},d_{2}d_{3},d_{2}%
d_{4},d_{3}d_{4},\underset{11}{0},...,\underset{14}{0},d_{1}d_{2}%
d_{3},\underset{16}{0},...,\underset{19}{0},\\
d_{1}d_{2}d_{4},\underset{21}{0},...,\underset{24}{0},d_{1}d_{3}%
d_{4},\underset{26}{0},...,\underset{29}{0},d_{2}d_{3}d_{4},\underset{31}%
{0},...,\underset{52}{0},d_{1}d_{2}d_{3}d_{4}%
\end{array}
\right)
\end{align*}
and
\begin{align*}
&  \gamma_{4231}\left(  d_{1},d_{2},d_{3},d_{4}\right) \\
&  =\mathfrak{m}\left(
\begin{array}
[c]{c}%
d_{1},d_{2},d_{3},d_{4},d_{1}d_{2},d_{1}d_{3},d_{1}d_{4},0,d_{2}d_{4}%
,d_{3}d_{4},0,0,d_{1}d_{2}d_{3},\underset{14}{0},...,\underset{19}{0},\\
d_{1}d_{2}d_{4},\underset{21}{0},...,\underset{24}{0},d_{1}d_{3}%
d_{4},\underset{26}{0},...,\underset{28}{0},d_{2}d_{3}d_{4},\underset{30}%
{0},...,\underset{50}{0},d_{1}d_{2}d_{3}d_{4},0,\underset{53}{0}%
\end{array}
\right)
\end{align*}
we have
\begin{align}
&  \left(  \gamma_{4321}\overset{\cdot}{\underset{41}{-}}\gamma_{4231}\right)
\left(  d_{1},d_{2},d_{3}\right) \nonumber\\
&  =\mathfrak{m}\left(
\begin{array}
[c]{c}%
d_{2},\underset{2}{0},\underset{3}{0},d_{1},\underset{5}{0},...,\underset
{7}{0},d_{3},\underset{9}{0},...,\underset{12}{0},-d_{2}d_{3},0,d_{2}%
d_{3},\underset{16}{0},...,\underset{28}{0},\\
-d_{1}d_{3},d_{1}d_{3},\underset{31}{0},...,\underset{50}{0},-d_{1}d_{2}%
d_{3},\underset{52}{0},d_{1}d_{2}d_{3}%
\end{array}
\right)  \label{t3..2.65}%
\end{align}
(\ref{t3.2.64}) and (\ref{t3.2.65}) imply that
\begin{align}
&  \left(  \left(  \gamma_{4132}\overset{\cdot}{\underset{41}{-}}\gamma
_{4123}\right)  \overset{\cdot}{\underset{1}{-}}\left(  \gamma_{4321}%
\overset{\cdot}{\underset{41}{-}}\gamma_{4231}\right)  \right)  \left(
d_{1},d_{2}\right) \nonumber\\
&  =\mathfrak{m}\left(  \underset{1}{0},...,\underset{3}{0},d_{1},\underset
{5}{0},...,\underset{10}{0},d_{2},\underset{12}{0},d_{2},\underset{14}%
{0},-d_{2},\underset{16}{0},...,\underset{47}{0},-d_{1}d_{2},d_{1}%
d_{2},\underset{50}{0},d_{1}d_{2},\underset{52}{0},-d_{1}d_{2}\right)
\label{t3.2.66}%
\end{align}
Since
\begin{align*}
&  \gamma_{1324}\left(  d_{1},d_{2},d_{3},d_{4}\right) \\
&  =\mathfrak{m}\left(
\begin{array}
[c]{c}%
d_{1},d_{2},d_{3},d_{4},\underset{5}{0},\underset{6}{0},\underset{7}{0}%
,d_{2}d_{3},\underset{9}{0},\underset{10}{0},d_{1}d_{2}d_{3},\underset{12}%
{0},...,\underset{26}{0},d_{2}d_{3}d_{4},\underset{28}{0},...,\underset{31}%
{0},\\
d_{1}d_{2}d_{3}d_{4},\underset{33}{0},...,\underset{53}{0}%
\end{array}
\right)
\end{align*}
and
\begin{align*}
&  \gamma_{1234}\left(  d_{1},d_{2},d_{3},d_{4}\right) \\
&  =\mathfrak{m}\left(  d_{1},d_{2},d_{3},d_{4},\underset{5}{0},...,\underset
{53}{0}\right)
\end{align*}
we have
\begin{align}
&  \left(  \gamma_{1324}\overset{\cdot}{\underset{41}{-}}\gamma_{1234}\right)
\left(  d_{1},d_{2},d_{3}\right) \nonumber\\
&  =\mathfrak{m}\left(  d_{2},\underset{2}{0},\underset{3}{0},d_{1}%
,\underset{5}{0},...,\underset{7}{0},d_{3},\underset{9}{0},\underset{10}%
{0},d_{2}d_{3},\underset{12}{0},...,\underset{26}{0},d_{1}d_{3},\underset
{28}{0},...,\underset{31}{0},d_{1}d_{2}d_{3},\underset{33}{0},...,\underset
{53}{0}\right)  \label{t3.2.67}%
\end{align}
Since
\begin{align*}
&  \gamma_{3214}\left(  d_{1},d_{2},d_{3},d_{4}\right) \\
&  =\mathfrak{m}\left(
\begin{array}
[c]{c}%
d_{1},d_{2},d_{3},d_{4},d_{1}d_{2},d_{1}d_{3},\underset{7}{0},d_{2}%
d_{3},\underset{9}{0},...,\underset{14}{0},d_{1}d_{2}d_{3},0,d_{1}d_{2}%
d_{4},\underset{18}{0},...,\underset{21}{0},d_{1}d_{3}d_{4},\\
\underset{23}{0},...,\underset{26}{0},d_{2}d_{3}d_{4},\underset{28}%
{0},...,\underset{43}{0},d_{1}d_{2}d_{3}d_{4},\underset{45}{0},...,\underset
{53}{0}%
\end{array}
\right)
\end{align*}
and
\begin{align*}
&  \gamma_{2314}\left(  d_{1},d_{2},d_{3},d_{4}\right) \\
&  =\mathfrak{m}\left(
\begin{array}
[c]{c}%
d_{1},d_{2},d_{3},d_{4},d_{1}d_{2},d_{1}d_{3},\underset{7}{0},...,\underset
{12}{0},d_{1}d_{2}d_{3},0,0,0,d_{1}d_{2}d_{4},\underset{18}{0},...,\underset
{21}{0},d_{1}d_{3}d_{4},\\
\underset{23}{0},...,\underset{37}{0},d_{1}d_{2}d_{3}d_{4},\underset{39}%
{0},...,\underset{53}{0}%
\end{array}
\right)
\end{align*}
we have
\begin{align}
&  \left(  \gamma_{3214}\overset{\cdot}{\underset{41}{-}}\gamma_{2314}\right)
\left(  d_{1},d_{2},d_{3}\right) \nonumber\\
&  =\mathfrak{m}\left(
\begin{array}
[c]{c}%
d_{2},\underset{2}{0},\underset{3}{0},d_{1},\underset{5}{0},...,\underset
{7}{0},d_{3},\underset{9}{0},...,\underset{12}{0},-d_{2}d_{3},\underset{14}%
{0},d_{2}d_{3},\underset{16}{0},...,\underset{26}{0},d_{1}d_{3},\\
\underset{28}{0},...,\underset{37}{0},-d_{1}d_{2}d_{3},\underset{39}%
{0},...,\underset{43}{0},d_{1}d_{2}d_{3},\underset{45}{0},...,\underset{53}{0}%
\end{array}
\right)  \label{t3.2.68}%
\end{align}
(\ref{t3.2.67}) and (\ref{t3.2.68}) imply that
\begin{align}
&  \left(  \left(  \gamma_{1324}\overset{\cdot}{\underset{41}{-}}\gamma
_{1234}\right)  \overset{\cdot}{\underset{1}{-}}\left(  \gamma_{3214}%
\overset{\cdot}{\underset{41}{-}}\gamma_{2314}\right)  \right)  \left(
d_{1},d_{2}\right) \nonumber\\
&  =\mathfrak{m}\left(
\begin{array}
[c]{c}%
\underset{1}{0},...,\underset{3}{0},d_{1},\underset{5}{0},...,\underset{10}%
{0},d_{2},\underset{12}{0},d_{2},\underset{14}{0},-d_{2},\underset{16}%
{0},...,\underset{31}{0},d_{1}d_{2},\\
\underset{33}{0},...,\underset{37}{0},d_{1}d_{2},\underset{39}{0}%
,...,\underset{43}{0},-d_{1}d_{2},\underset{45}{0},...,\underset{53}{0}%
\end{array}
\right)  \label{t3.2.69}%
\end{align}
(\ref{t3.2.66}) and (\ref{t3.2.69}) imply that
\begin{align}
&  \left(
\begin{array}
[c]{c}%
\left(  \left(  \gamma_{4132}\overset{\cdot}{\underset{41}{-}}\gamma
_{4123}\right)  \overset{\cdot}{\underset{1}{-}}\left(  \gamma_{4321}%
\overset{\cdot}{\underset{41}{-}}\gamma_{4231}\right)  \right)  \overset
{\cdot}{-}\\
\left(  \left(  \gamma_{1324}\overset{\cdot}{\underset{41}{-}}\gamma
_{1234}\right)  \overset{\cdot}{\underset{1}{-}}\left(  \gamma_{3214}%
\overset{\cdot}{\underset{41}{-}}\gamma_{2314}\right)  \right)
\end{array}
\right)  \left(  d\right) \nonumber\\
&  =\mathfrak{m}\left(  \underset{1}{0},...,\underset{31}{0},-d,\underset
{33}{0},...,\underset{37}{0},-d,\underset{39}{0},...,\underset{43}%
{0},d,\underset{45}{0},...,\underset{47}{0},-d,d,\underset{50}{0}%
,d,\underset{52}{0},\underset{53}{-d}\right)  \label{t3.2.70}%
\end{align}

\item Since
\begin{align*}
&  \gamma_{4213}\left(  d_{1},d_{2},d_{3},d_{4}\right) \\
&  =\mathfrak{m}\left(
\begin{array}
[c]{c}%
d_{1},d_{2},d_{3},d_{4},d_{1}d_{2},\underset{6}{0},d_{1}d_{4},\underset{8}%
{0},d_{2}d_{4},d_{3}d_{4},\underset{11}{0},d_{1}d_{2}d_{3},\underset{13}%
{0},...,\underset{19}{0},d_{1}d_{2}d_{4},\\
\underset{21}{0},...,\underset{23}{0},d_{1}d_{3}d_{4},\underset{25}%
{0},...,\underset{28}{0},d_{2}d_{3}d_{4},\underset{30}{0},...,\underset{49}%
{0},d_{1}d_{2}d_{3}d_{4},\underset{51}{0},...,\underset{53}{0}%
\end{array}
\right)
\end{align*}
and
\begin{align*}
&  \gamma_{4231}\left(  d_{1},d_{2},d_{3},d_{4}\right) \\
&  =\mathfrak{m}\left(
\begin{array}
[c]{c}%
d_{1},d_{2},d_{3},d_{4},d_{1}d_{2},d_{1}d_{3},d_{1}d_{4},\underset{8}{0}%
,d_{2}d_{4},d_{3}d_{4},\underset{11}{0},\underset{12}{0},d_{1}d_{2}%
d_{3},\underset{14}{0},...,\underset{19}{0},\\
d_{1}d_{2}d_{4},\underset{21}{0},...,\underset{24}{0},d_{1}d_{3}%
d_{4},\underset{26}{0},...,\underset{28}{0},d_{2}d_{3}d_{4},\underset{30}%
{0},...,\underset{50}{0},d_{1}d_{2}d_{3}d_{4},\underset{52}{0},\underset
{53}{0}%
\end{array}
\right)
\end{align*}
we have
\begin{align}
&  \left(  \gamma_{4213}\overset{\cdot}{\underset{42}{-}}\gamma_{4231}\right)
\left(  d_{1},d_{2},d_{3}\right) \nonumber\\
&  =\mathfrak{m}\left(
\begin{array}
[c]{c}%
\underset{1}{0},d_{2},\underset{3}{0},d_{1},\underset{5}{0},-d_{3}%
,\underset{7}{0},...,\underset{11}{0},d_{2}d_{3},-d_{2}d_{3},\underset{14}%
{0},...,\underset{23}{0},\\
d_{1}d_{3},-d_{1}d_{3},\underset{26}{0},...,\underset{49}{0},d_{1}d_{2}%
d_{3},\underset{51}{-d_{1}d_{2}d_{3}},0,0
\end{array}
\right)  \label{t3.2.71}%
\end{align}
Since
\begin{align*}
&  \gamma_{4132}\left(  d_{1},d_{2},d_{3},d_{4}\right) \\
&  =\mathfrak{m}\left(
\begin{array}
[c]{c}%
d_{1},d_{2},d_{3},d_{4},\underset{5}{0},\underset{6}{0},d_{1}d_{4},d_{2}%
d_{3},d_{2}d_{4},d_{3}d_{4},d_{1}d_{2}d_{3},\underset{12}{0},...,\underset
{18}{0},d_{1}d_{2}d_{4},\underset{20}{0},...,\underset{23}{0},\\
d_{1}d_{3}d_{4},\underset{25}{0},...,\underset{29}{0},d_{2}d_{3}%
d_{4},\underset{31}{0},...,\underset{48}{0},d_{1}d_{2}d_{3}d_{4},\underset
{50}{0},...,\underset{53}{0}%
\end{array}
\right)
\end{align*}
and
\begin{align*}
&  \gamma_{4312}\left(  d_{1},d_{2},d_{3},d_{4}\right) \\
&  =\mathfrak{m}\left(
\begin{array}
[c]{c}%
d_{1},d_{2},d_{3},d_{4},0,d_{1}d_{3},d_{1}d_{4},d_{2}d_{3},d_{2}d_{4}%
,d_{3}d_{4},\underset{11}{0},...,\underset{13}{0},d_{1}d_{2}d_{3}%
,\underset{15}{0},...,\underset{18}{0},\\
d_{1}d_{2}d_{4},\underset{20}{0},...,\underset{24}{0},d_{1}d_{3}%
d_{4},\underset{26}{0},...,\underset{29}{0},d_{2}d_{3}d_{4},\underset{31}%
{0},...,\underset{51}{0},d_{1}d_{2}d_{3}d_{4},\underset{53}{0}%
\end{array}
\right)
\end{align*}
we have
\begin{align}
&  \left(  \gamma_{4132}\overset{\cdot}{\underset{42}{-}}\gamma_{4312}\right)
\left(  d_{1},d_{2},d_{3}\right) \nonumber\\
&  =\mathfrak{m}\left(
\begin{array}
[c]{c}%
\underset{1}{0},d_{2},\underset{3}{0},d_{1},\underset{5}{0},-d_{3}%
,\underset{7}{0},...,\underset{10}{0},d_{2}d_{3},\underset{12}{0}%
,\underset{13}{0},-d_{2}d_{3},\underset{15}{0},...,\underset{23}{0},d_{1}%
d_{3},-d_{1}d_{3},\\
\underset{26}{0},...,\underset{48}{0},d_{1}d_{2}d_{3},\underset{50}%
{0},\underset{51}{0},-d_{1}d_{2}d_{3},\underset{53}{0}%
\end{array}
\right)  \label{t3.2.72}%
\end{align}
(\ref{t3.2.71}) and (\ref{t3.2.72}) imply that
\begin{align}
&  \left(  \left(  \gamma_{4213}\overset{\cdot}{\underset{42}{-}}\gamma
_{4231}\right)  \overset{\cdot}{\underset{1}{-}}\left(  \gamma_{4132}%
\overset{\cdot}{\underset{42}{-}}\gamma_{4312}\right)  \right)  \left(
d_{1},d_{2}\right) \nonumber\\
&  =\mathfrak{m}\left(  \underset{1}{0},...,\underset{3}{0},d_{1},\underset
{5}{0},...,\underset{10}{0},-d_{2},d_{2},-d_{2},d_{2}\underset{15}%
{,0},...,\underset{48}{0},-d_{1}d_{2},d_{1}d_{2},-d_{1}d_{2},d_{1}%
d_{2},\underset{53}{0}\right)  \label{t3.2.73}%
\end{align}
Since
\begin{align*}
&  \gamma_{2134}\left(  d_{1},d_{2},d_{3},d_{4}\right) \\
&  =\mathfrak{m}\left(  d_{1},d_{2},d_{3},d_{4},d_{1}d_{2},\underset{6}%
{0},...,\underset{11}{0},d_{1}d_{2}d_{3},\underset{13}{0},...,\underset{16}%
{0},d_{1}d_{2}d_{4},\underset{18}{0},...,\underset{35}{0},d_{1}d_{2}d_{3}%
d_{4},\underset{37}{0},...,\underset{53}{0}\right)
\end{align*}
and
\begin{align*}
&  \gamma_{2314}\left(  d_{1},d_{2},d_{3},d_{4}\right) \\
&  =\mathfrak{m}\left(
\begin{array}
[c]{c}%
d_{1},d_{2},d_{3},d_{4},d_{1}d_{2},d_{1}d_{3},\underset{7}{0},...,\underset
{12}{0},d_{1}d_{2}d_{3},0,0,0,d_{1}d_{2}d_{4},\underset{18}{0},...,\underset
{21}{0},d_{1}d_{3}d_{4},\\
\underset{23}{0},...,\underset{37}{0},d_{1}d_{2}d_{3}d_{4},\underset{39}%
{0},...,\underset{53}{0}%
\end{array}
\right)
\end{align*}
we have
\begin{align}
&  \left(  \gamma_{2134}\overset{\cdot}{\underset{42}{-}}\gamma_{2314}\right)
\left(  d_{1},d_{2},d_{3}\right) \nonumber\\
&  =\mathfrak{m}\left(
\begin{array}
[c]{c}%
\underset{1}{0},d_{2},\underset{3}{0},d_{1},\underset{5}{0},-d_{3}%
,\underset{7}{0},...,\underset{11}{0},d_{2}d_{3},-d_{2}d_{3},\underset{14}%
{0},...,\underset{21}{0},-d_{1}d_{3},\\
\underset{23}{0},...,\underset{35}{0},d_{1}d_{2}d_{3},0,-d_{1}d_{2}%
d_{3},\underset{39}{0},...,\underset{53}{0}%
\end{array}
\right)  \label{t3.2.74}%
\end{align}
Since
\begin{align*}
&  \gamma_{1324}\left(  d_{1},d_{2},d_{3},d_{4}\right) \\
&  =\mathfrak{m}\left(
\begin{array}
[c]{c}%
d_{1},d_{2},d_{3},d_{4},\underset{5}{0},\underset{6}{0},\underset{7}{0}%
,d_{2}d_{3},\underset{9}{0},\underset{10}{0},d_{1}d_{2}d_{3},\underset{12}%
{0},...,\underset{26}{0},d_{2}d_{3}d_{4},0,0,\underset{30}{0},\\
0,d_{1}d_{2}d_{3}d_{4},\underset{33}{0},...,\underset{53}{0}%
\end{array}
\right)
\end{align*}
and
\begin{align*}
&  \gamma_{3124}\left(  d_{1},d_{2},d_{3},d_{4}\right) \\
&  =\mathfrak{m}\left(
\begin{array}
[c]{c}%
d_{1},d_{2},d_{3},d_{4},\underset{5}{0},d_{1}d_{3},\underset{7}{0},d_{2}%
d_{3},\underset{9}{0},...,\underset{13}{0},d_{1}d_{2}d_{3},\underset{15}%
{0},...,\underset{21}{0},d_{1}d_{3}d_{4},\underset{23}{0},...,\underset{26}%
{0},\\
d_{2}d_{3}d_{4},\underset{28}{0},...,\underset{41}{0},d_{1}d_{2}d_{3}%
d_{4},\underset{43}{0},...,\underset{53}{0}%
\end{array}
\right)
\end{align*}
we have
\begin{align}
&  \left(  \gamma_{1324}\overset{\cdot}{\underset{42}{-}}\gamma_{3124}\right)
\left(  d_{1},d_{2},d_{3}\right) \nonumber\\
&  =\mathfrak{m}\left(
\begin{array}
[c]{c}%
\underset{1}{0},d_{2},\underset{3}{0},d_{1},\underset{5}{0},-d_{3}%
,\underset{7}{0},...,\underset{10}{0},d_{2}d_{3},0,0,-d_{2}d_{3},\underset
{15}{0},...,\underset{21}{0},\\
-d_{1}d_{3},\underset{23}{0},...,\underset{31}{0},d_{1}d_{2}d_{3}%
,\underset{33}{0},...,\underset{41}{0},-d_{1}d_{2}d_{3},\underset{43}%
{0},...,\underset{53}{0}%
\end{array}
\right)  \label{t3.2.75}%
\end{align}
(\ref{t3.2.74}) and (\ref{t3.2.75}) imply that
\begin{align}
&  \left(  \left(  \gamma_{2134}\overset{\cdot}{\underset{42}{-}}\gamma
_{2314}\right)  \overset{\cdot}{\underset{1}{-}}\left(  \gamma_{1324}%
\overset{\cdot}{\underset{42}{-}}\gamma_{3124}\right)  \right)  \left(
d_{1},d_{2}\right) \nonumber\\
&  =\mathfrak{m}\left(
\begin{array}
[c]{c}%
\underset{1}{0},...,\underset{3}{0},d_{1},\underset{5}{0},...,\underset{10}%
{0},-d_{2},d_{2},-d_{2},d_{2},\underset{15}{0},...,\underset{31}{0}%
,-d_{1}d_{2},\underset{33}{0},...,\underset{35}{0},\\
d_{1}d_{2},0,-d_{1}d_{2},\underset{39}{0},...,\underset{41}{0},d_{1}%
d_{2},\underset{43}{0},...,\underset{53}{0}%
\end{array}
\right)  \label{t3.2.76}%
\end{align}
(\ref{t3.2.73}) and (\ref{t3.2.76}) imply that
\begin{align}
&  \left(
\begin{array}
[c]{c}%
\left(  \left(  \gamma_{4213}\overset{\cdot}{\underset{42}{-}}\gamma
_{4231}\right)  \overset{\cdot}{\underset{1}{-}}\left(  \gamma_{4132}%
\overset{\cdot}{\underset{42}{-}}\gamma_{4312}\right)  \right)  \overset
{\cdot}{-}\\
\left(  \left(  \gamma_{2134}\overset{\cdot}{\underset{42}{-}}\gamma
_{2314}\right)  \overset{\cdot}{\underset{1}{-}}\left(  \gamma_{1324}%
\overset{\cdot}{\underset{42}{-}}\gamma_{3124}\right)  \right)
\end{array}
\right)  \left(  d\right) \nonumber\\
&  =\mathfrak{m}\left(  \underset{1}{0},...,\underset{31}{0},d,\underset
{33}{0},...,\underset{35}{0},-d,0,d,\underset{39}{0},...,\underset{41}%
{0},-d,\underset{43}{0},...,\underset{48}{0},-d,d,-d,d,\underset{53}%
{0}\right)  \label{t3.2.77}%
\end{align}

\item Since
\begin{align*}
&  \gamma_{4321}\left(  d_{1},d_{2},d_{3},d_{4}\right) \\
&  =\mathfrak{m}\left(
\begin{array}
[c]{c}%
d_{1},d_{2},d_{3},d_{4},d_{1}d_{2},d_{1}d_{3},d_{1}d_{4},d_{2}d_{3},d_{2}%
d_{4},d_{3}d_{4},\underset{11}{0},...,\underset{14}{0},d_{1}d_{2}%
d_{3},\underset{16}{0},...,\underset{19}{0},\\
d_{1}d_{2}d_{4},\underset{21}{0},...,\underset{24}{0},d_{1}d_{3}%
d_{4},\underset{26}{0},...,\underset{29}{0},d_{2}d_{3}d_{4},\underset{31}%
{0},...,\underset{52}{0},d_{1}d_{2}d_{3}d_{4}%
\end{array}
\right)
\end{align*}
and
\begin{align*}
&  \gamma_{4312}\left(  d_{1},d_{2},d_{3},d_{4}\right) \\
&  =\mathfrak{m}\left(
\begin{array}
[c]{c}%
d_{1},d_{2},d_{3},d_{4},\underset{5}{0},d_{1}d_{3},d_{1}d_{4},d_{2}d_{3}%
,d_{2}d_{4},d_{3}d_{4},\underset{11}{0},...,\underset{13}{0},d_{1}d_{2}%
d_{3},\underset{15}{0},...,\underset{18}{0},\\
d_{1}d_{2}d_{4},\underset{20}{0},...,\underset{24}{0},d_{1}d_{3}%
d_{4},\underset{26}{0},...,\underset{29}{0},d_{2}d_{3}d_{4},\underset{31}%
{0},...,\underset{51}{0},d_{1}d_{2}d_{3}d_{4},\underset{53}{0}%
\end{array}
\right)
\end{align*}
we have
\begin{align}
&  \left(  \gamma_{4321}\overset{\cdot}{\underset{43}{-}}\gamma_{4312}\right)
\left(  d_{1},d_{2},d_{3}\right) \nonumber\\
&  =\mathfrak{m}\left(
\begin{array}
[c]{c}%
\underset{1}{0},\underset{2}{0},d_{2},d_{1},d_{3},\underset{6}{0}%
,...,\underset{13}{0},-d_{2}d_{3},d_{2}d_{3},\underset{16}{0},...,\underset
{18}{0},-d_{1}d_{3},d_{1}d_{3},\\
\underset{21}{0},...,\underset{51}{0},-d_{1}d_{2}d_{3},d_{1}d_{2}d_{3}%
\end{array}
\right)  \label{t3.2.78}%
\end{align}
Since
\begin{align*}
&  \gamma_{4213}\left(  d_{1},d_{2},d_{3},d_{4}\right) \\
&  =\mathfrak{m}\left(
\begin{array}
[c]{c}%
d_{1},d_{2},d_{3},d_{4},d_{1}d_{2},\underset{6}{0},d_{1}d_{4},\underset{8}%
{0},d_{2}d_{4},d_{3}d_{4},\underset{11}{0},d_{1}d_{2}d_{3},\underset{13}%
{0},...,\underset{19}{0},d_{1}d_{2}d_{4},\\
\underset{21}{0},...,\underset{23}{0},d_{1}d_{3}d_{4},\underset{25}%
{0},...,\underset{28}{0},d_{2}d_{3}d_{4},\underset{30}{0},...,\underset{49}%
{0},d_{1}d_{2}d_{3}d_{4},\underset{51}{0},...,\underset{53}{0}%
\end{array}
\right)
\end{align*}
and
\begin{align*}
&  \gamma_{4123}\left(  d_{1},d_{2},d_{3},d_{4}\right) \\
&  =\mathfrak{m}\left(
\begin{array}
[c]{c}%
d_{1},d_{2},d_{3},d_{4},0,0,d_{1}d_{4},0,0,d_{2}d_{4},d_{3}d_{4},\underset
{11}{0},...,\underset{18}{0},d_{1}d_{2}d_{4},\underset{20}{0},...,\underset
{23}{0,}d_{1}d_{3}d_{4},\\
\underset{25}{0},...,\underset{28}{0},d_{2}d_{3}d_{4},\underset{30}%
{0},...,\underset{47}{0},d_{1}d_{2}d_{3}d_{4},\underset{49}{0},...,\underset
{53}{0}%
\end{array}
\right)
\end{align*}
we have
\begin{align}
&  \left(  \gamma_{4213}\overset{\cdot}{\underset{43}{-}}\gamma_{4123}\right)
\left(  d_{1},d_{2},d_{3}\right) \nonumber\\
&  =\mathfrak{m}\left(
\begin{array}
[c]{c}%
\underset{1}{0},\underset{2}{0},d_{2},d_{1},d_{3},\underset{6}{0}%
,...,\underset{11}{0},d_{2}d_{3},\underset{13}{0},...,\underset{18}{0}%
,-d_{1}d_{3},d_{1}d_{3},\underset{21}{0},...,\underset{47}{0},-d_{1}d_{2}%
d_{3},\underset{49}{0},\\
d_{1}d_{2}d_{3},\underset{51}{0},...,\underset{53}{0}%
\end{array}
\right)  \label{t3.2.79}%
\end{align}
(\ref{t3.2.78}) and (\ref{t3.2.79}) imply that
\begin{align}
&  \left(  \left(  \gamma_{4321}\overset{\cdot}{\underset{43}{-}}\gamma
_{4312}\right)  \overset{\cdot}{\underset{1}{-}}\left(  \gamma_{4213}%
\overset{\cdot}{\underset{43}{-}}\gamma_{4123}\right)  \right)  \left(
d_{1},d_{2}\right) \nonumber\\
&  =\mathfrak{m}\left(  \underset{1}{0},...,\underset{3}{0},d_{1},\underset
{5}{0},...,\underset{11}{0},-d_{2},0,-d_{2},d_{2},\underset{16}{0}%
,...,\underset{47}{0},d_{1}d_{2},\underset{49}{0},-d_{1}d_{2},\underset{51}%
{0},-d_{1}d_{2},d_{1}d_{2}\right)  \label{t3.2.80}%
\end{align}
Since
\begin{align*}
&  \gamma_{3214}\left(  d_{1},d_{2},d_{3},d_{4}\right) \\
&  =\mathfrak{m}\left(
\begin{array}
[c]{c}%
d_{1},d_{2},d_{3},d_{4},d_{1}d_{2},d_{1}d_{3},0,d_{2}d_{3},\underset{9}%
{0},...,\underset{14}{0},d_{1}d_{2}d_{3},0,d_{1}d_{2}d_{4},\underset{18}%
{0},...,\underset{21}{0},d_{1}d_{3}d_{4},\\
\underset{23}{0},...,\underset{26}{0},d_{2}d_{3}d_{4},\underset{28}%
{0},...,\underset{43}{0},d_{1}d_{2}d_{3}d_{4},\underset{45}{0},...,\underset
{53}{0}%
\end{array}
\right)
\end{align*}
and
\begin{align*}
&  \gamma_{3124}\left(  d_{1},d_{2},d_{3},d_{4}\right) \\
&  =\mathfrak{m}\left(
\begin{array}
[c]{c}%
d_{1},d_{2},d_{3},d_{4},\underset{5}{0},d_{1}d_{3},\underset{7}{0},d_{2}%
d_{3},\underset{9}{0},...,\underset{13}{0},d_{1}d_{2}d_{3},\underset{15}%
{0},...,\underset{21}{0},d_{1}d_{3}d_{4},\underset{23}{0},...,\underset{26}%
{0},\\
d_{2}d_{3}d_{4},\underset{28}{0},...,\underset{41}{0},d_{1}d_{2}d_{3}%
d_{4},\underset{43}{0},...,\underset{53}{0}%
\end{array}
\right)
\end{align*}
we have
\begin{align}
&  \left(  \gamma_{3214}\overset{\cdot}{\underset{43}{-}}\gamma_{3124}\right)
\left(  d_{1},d_{2},d_{3}\right) \nonumber\\
&  =\mathfrak{m}\left(
\begin{array}
[c]{c}%
\underset{1}{0},\underset{2}{0},d_{2},d_{1},d_{3},\underset{6}{0}%
,...,\underset{13}{0},-d_{2}d_{3},d_{2}d_{3},0,d_{1}d_{3},\\
\underset{18}{0},...,\underset{41}{0},-d_{1}d_{2}d_{3},0,d_{1}d_{2}%
d_{3},\underset{45}{0},...,\underset{53}{0}%
\end{array}
\right)  \label{t3.2.81}%
\end{align}
Since
\begin{align*}
&  \gamma_{2134}\left(  d_{1},d_{2},d_{3},d_{4}\right) \\
&  =\mathfrak{m}\left(  d_{1},d_{2},d_{3},d_{4},d_{1}d_{2},\underset{6}%
{0},...,\underset{11}{0},d_{1}d_{2}d_{3},\underset{13}{0},...,\underset{16}%
{0},d_{1}d_{2}d_{4},\underset{18}{0},...,\underset{35}{0},d_{1}d_{2}d_{3}%
d_{4},\underset{37}{0},...,\underset{53}{0}\right)
\end{align*}
and
\begin{align*}
&  \gamma_{1234}\left(  d_{1},d_{2},d_{3},d_{4}\right) \\
&  =\mathfrak{m}\left(  d_{1},d_{2},d_{3},d_{4},\underset{5}{0},...,\underset
{53}{0}\right)
\end{align*}
we have
\begin{align}
&  \left(  \gamma_{2134}\overset{\cdot}{\underset{43}{-}}\gamma_{1234}\right)
\left(  d_{1},d_{2},d_{3}\right) \nonumber\\
&  =\mathfrak{m}\left(  \underset{1}{0},\underset{2}{0},d_{2},d_{1}%
,d_{3},\underset{6}{0},...,\underset{11}{0},d_{2}d_{3},\underset{13}%
{0},...,\underset{16}{0},d_{1}d_{3},\underset{18}{0},...,\underset{35}%
{0},d_{1}d_{2}d_{3},\underset{37}{0},...,\underset{53}{0}\right)
\label{t3.2.82}%
\end{align}
(\ref{t3.2.81}) and (\ref{t3.2.82}) imply that
\begin{align}
&  \left(  \left(  \gamma_{3214}\overset{\cdot}{\underset{43}{-}}\gamma
_{3124}\right)  \overset{\cdot}{\underset{1}{-}}\left(  \gamma_{2134}%
\overset{\cdot}{\underset{43}{-}}\gamma_{1234}\right)  \right)  \left(
d_{1},d_{2}\right) \nonumber\\
&  =\mathfrak{m}\left(
\begin{array}
[c]{c}%
\underset{1}{0},...,\underset{3}{0},d_{1},\underset{5}{0},...,\underset{11}%
{0},-d_{2},0,-d_{2},d_{2},\underset{16}{0},...,\underset{35}{0},-d_{1}d_{2},\\
\underset{37}{0},...,\underset{41}{0},-d_{1}d_{2},\underset{43}{0},d_{1}%
d_{2},\underset{45}{0},...,\underset{53}{0}%
\end{array}
\right)  \label{t3.2.83}%
\end{align}
(\ref{t3.2.80}) and (\ref{t3.2.83}) imply that
\begin{align}
&  \left(
\begin{array}
[c]{c}%
\left(  \left(  \gamma_{4321}\overset{\cdot}{\underset{43}{-}}\gamma
_{4312}\right)  \overset{\cdot}{\underset{1}{-}}\left(  \gamma_{4213}%
\overset{\cdot}{\underset{43}{-}}\gamma_{4123}\right)  \right)  \overset
{\cdot}{-}\\
\left(  \left(  \gamma_{3214}\overset{\cdot}{\underset{43}{-}}\gamma
_{3124}\right)  \overset{\cdot}{\underset{1}{-}}\left(  \gamma_{2134}%
\overset{\cdot}{\underset{43}{-}}\gamma_{1234}\right)  \right)
\end{array}
\right)  \left(  d\right) \nonumber\\
&  =\mathfrak{m}\left(  \underset{1}{0},...,\underset{35}{0},d,\underset
{37}{0},...,\underset{41}{0},d,0,-d,\underset{45}{0},...,\underset{47}%
{0},d,\underset{49}{0},-d,\underset{51}{0},-d,d\right)  \label{t3.2.84}%
\end{align}

\item By (\ref{t.8.2.7}), (\ref{t8.2.14}), (\ref{t8.2.21}), (\ref{t8.2.28}),
(\ref{t3.2.35}), (\ref{t3.2.42}), (\ref{t3.2.49}), (\ref{t3.2.56}),
(\ref{t3.2.63}), (\ref{t3.2.70}), (\ref{t3.2.77}) and (\ref{t3.2.84}), we have%
\begin{align}
&  \text{the left-hand side of the equation (\ref{t3.2.0})}\nonumber\\
&  =\left(  \lambda d.\mathfrak{m}\left(  \underset{1}{0},...,\underset{30}%
{0},-d,\underset{32}{0},-d,\underset{34}{0},d,\underset{36}{0},...,\underset
{38}{0},-d,0,d,\underset{42}{0},...,\underset{46}{0},d,\underset{48}%
{0},...,\underset{52}{0},-d\right)  \right)  +\nonumber\\
&  \left(  \lambda d.\mathfrak{m}\left(  \underset{1}{0},...,\underset{30}%
{0},d,-d,d,-d,\underset{35}{0},...,\underset{40}{0},-d,\underset{42}%
{0},...\underset{44}{0},d,\underset{46}{0},-d,\underset{48}{0},...,\underset
{50}{0},d,\underset{52}{0},\underset{53}{0}\right)  \right)  +\nonumber\\
&  \left(  \lambda d.\mathfrak{m}\left(  \underset{1}{0},...,\underset{31}%
{0},d,\underset{33}{0},d,-d,\underset{36}{0},...,\underset{38}{0}%
,d,\underset{40}{0},...,\underset{44}{0},-d,\underset{46}{0},...,\underset
{50}{0},-d,\underset{52}{0},d\right)  \right)  +\nonumber\\
&  \left(  \lambda d.\mathfrak{m}\left(  \underset{1}{0},...,\underset{32}%
{0},d,\underset{34}{0},-d,-d,d,\underset{38}{0},d,\underset{40}{0}%
,-d,\underset{42}{0},...,\underset{45}{0},-d,\underset{47}{0},...,\underset
{51}{0},d,\underset{53}{0}\right)  \right)  +\nonumber\\
&  \left(  \lambda d.\mathfrak{m}\left(  \underset{1}{0},...,\underset{34}%
{0},d,\underset{36}{0},-d,d,-d,d,\underset{41}{0},\underset{42}{0}%
,-d,\underset{44}{0},\underset{45}{0},d,\underset{47}{0},\underset{48}%
{0},-d,\underset{50}{0},...,\underset{53}{0}\right)  \right)  +\nonumber\\
&  \left(  \lambda d.\mathfrak{m}\left(  \underset{1}{0},...,\underset{32}%
{0},-d,\underset{34}{0},\underset{35}{0},d,\underset{37}{0},-d,\underset
{39}{0},-d,d,\underset{42}{0},d,\underset{44}{0},...,\underset{48}%
{0},d,\underset{50}{0},\underset{51}{0},-d,\underset{53}{0}\right)  \right)
+\nonumber\\
&  \left(  \lambda d.\mathfrak{m}\left(  \underset{1}{0},...,\underset{30}%
{0},-d,\underset{32}{0},\underset{33}{0},d,\underset{35}{0},...,\underset
{39}{0},d,\underset{41}{0},d,-d,\underset{44}{0},-d,\underset{46}%
{0},d,\underset{48}{0},\underset{49}{0},-d,\underset{51}{0},...,\underset
{53}{0}\right)  \right)  +\nonumber\\
&  \left(  \lambda d.\mathfrak{m}\left(  \underset{1}{0},...,\underset{30}%
{0},-d,\underset{32}{0},\underset{33}{0},d,\underset{35}{0},...,\underset
{39}{0},d,\underset{41}{0},d,-d,\underset{44}{0},-d,\underset{46}%
{0},d,\underset{48}{0},\underset{49}{0},-d,\underset{51}{0},...,\underset
{53}{0}\right)  \right)  +\nonumber\\
&  \left(  \lambda d.\mathfrak{m}\left(  \underset{1}{0},...,\underset{33}%
{0},-d,\underset{35}{0},\underset{36}{0},d,\underset{38}{0},\underset{39}%
{0},-d,\underset{41}{0},\underset{42}{0},d,-d,d,-d,\underset{47}%
{0},d,\underset{49}{0},...,\underset{53}{0}\right)  \right)  +\nonumber\\
&  \left(  \lambda d.\mathfrak{m}\left(  \underset{1}{0},...,\underset{30}%
{0},d,\underset{32}{0},...,\underset{36}{0},-d,\underset{38}{0},...,\underset
{41}{0},-d,\underset{43}{0},d,\underset{45}{0},d,-d,-d,\underset{49}%
{0},d,\underset{51}{0},...,\underset{53}{0}\right)  \right)  +\nonumber\\
&  \left(  \lambda d.\mathfrak{m}\left(  \underset{1}{0},...,\underset{31}%
{0},-d,\underset{33}{0},...,\underset{37}{0},-d,\underset{39}{0}%
,...,\underset{43}{0},d,\underset{45}{0},...,\underset{47}{0},-d,d,\underset
{50}{0},d,0,\underset{53}{-d}\right)  \right)  +\nonumber\\
&  \left(  \lambda d.\mathfrak{m}\left(  \underset{1}{0},...,\underset{31}%
{0},d,\underset{33}{0},...,\underset{35}{0},-d,0,d,\underset{39}%
{0},...,\underset{41}{0},-d,\underset{43}{0},...,\underset{48}{0}%
,-d,d,-d,d,0\right)  \right)  +\nonumber\\
&  \left(  \lambda d.\mathfrak{m}\left(  \underset{1}{0},...,\underset{35}%
{0},d,\underset{37}{0},...,\underset{41}{0},d,0,-d,\underset{45}%
{0},...,\underset{47}{0},d,0,-d,0,-d,d\right)  \right) \nonumber\\
&  =\lambda d.\mathfrak{m}\left(
\begin{array}
[c]{c}%
\underset{1}{0},...,\underset{30}{0},\underset{31}{-d+d-d+d},\underset
{32}{-d+d-d+d},\underset{33}{-d+d+d-d},\\
\underset{34}{-d+d+d-d},\underset{35}{d-d-d+d},\underset{36}{-d+d-d+d},\\
\underset{37}{d-d+d-d},\underset{38}{d-d-d+d},\underset{39}{-d+d+d-d},\\
\underset{40}{d-d+d-d},\underset{41}{d-d-d+d},\underset{42}{d-d-d+d},\\
\underset{43}{-d+d-d+d},\underset{44}{-d+d+d-d},\underset{45}{d-d-d+d},\\
\underset{46}{-d+d-d+d},\underset{47}{d-d+d-d},\underset{48}{d-d-d+d},\\
\underset{49}{-d+d+d-d},\underset{50}{-d+d+d-d},\underset{51}{d-d+d-d},\\
\underset{52}{d-d+d-d},\underset{53}{-d+d-d+d}%
\end{array}
\right) \nonumber\\
&  =\lambda d.\mathfrak{m}\left(  \underset{1}{0},...,\underset{53}{0}\right)
\label{t3.2.85}%
\end{align}

\end{enumerate}
\end{proof}

\begin{remark}
For our convenience we display the positions $31-53$ in (\ref{t3.2.7}),
(\ref{t3.2.14}), (\ref{t3.2.21}), (\ref{t3.2.28}), (\ref{t3.2.35}),
(\ref{t3.2.42}), (\ref{t3.2.49}), (\ref{t3.2.56}), (\ref{t3.2.63}),
(\ref{t3.2.70}), (\ref{t3.2.77}) and (\ref{t3.2.84}) as a table:%
\begin{equation}%
\begin{array}
[c]{ccccccccccccc}
& 1/12 & 2/13 & 3/14 & 4/21 & 5/23 & 6/24 & 7/31 & 8/32 & 9/34 & 10/41 &
11/42 & 12/43\\
31/1243 & -d & d &  &  &  &  & -d &  & d &  &  & \\
32/1324 &  & -d & d &  &  &  &  &  &  & -d & d & \\
33/1342 & -d & d &  & d &  & -d &  &  &  &  &  & \\
34/1423 &  & -d & d &  &  &  & d & -d &  &  &  & \\
35/1432 & d &  & -d & -d & d &  &  &  &  &  &  & \\
36/2134 &  &  &  & -d &  & d &  &  &  &  & -d & d\\
37/2143 &  &  &  & d & -d &  &  & d & -d &  &  & \\
38/2314 &  &  &  &  & d & -d &  &  &  & -d & d & \\
39/2341 & -d &  & d & d & -d &  &  &  &  &  &  & \\
40/2413 &  &  &  &  & d & -d & d & -d &  &  &  & \\
41/2431 & d & -d &  & -d &  & d &  &  &  &  &  & \\
42/3124 &  &  &  &  &  &  & d &  & -d &  & -d & d\\
43/3142 &  &  &  &  & -d & d & -d & d &  &  &  & \\
44/3214 &  &  &  &  &  &  &  & -d & d & d &  & -d\\
45/3241 &  & d & -d &  &  &  & -d & d &  &  &  & \\
46/3412 &  &  &  & -d & d &  &  & -d & d &  &  & \\
47/3421 & d & -d &  &  &  &  & d &  & -d &  &  & \\
48/4123 &  &  &  &  &  &  &  & d & -d & -d &  & d\\
49/4132 &  &  &  &  & -d & d &  &  &  & d & -d & \\
50/4213 &  &  &  &  &  &  & -d &  & d &  & d & -d\\
51/4231 &  & d & -d &  &  &  &  &  &  & d & -d & \\
52/4312 &  &  &  & d &  & -d &  &  &  &  & d & -d\\
53/4321 & -d &  & d &  &  &  &  &  &  & -d &  & d
\end{array}
\label{Figure 3}%
\end{equation}

\end{remark}

\begin{corollary}
Let $M$\ be a microlinear space with
\[
X_{1},X_{2},X_{3},X_{4}\in\mathfrak{X}\left(  M\right)
\]
Then we have
\begin{align*}
&  \left[  X_{1},\left[  X_{2},\left[  X_{3},X_{4}\right]  \right]  \right]
+\left[  X_{1},\left[  X_{3},\left[  X_{4},X_{2}\right]  \right]  \right]
+\left[  X_{1},\left[  X_{4},\left[  X_{2},X_{3}\right]  \right]  \right]  +\\
&  \left[  X_{2},\left[  X_{1},\left[  X_{4},X_{3}\right]  \right]  \right]
+\left[  X_{2},\left[  X_{3},\left[  X_{1},X_{4}\right]  \right]  \right]
+\left[  X_{2},\left[  X_{4},\left[  X_{3},X_{1}\right]  \right]  \right]  +\\
&  \left[  X_{3},\left[  X_{1},\left[  X_{2},X_{4}\right]  \right]  \right]
+\left[  X_{3},\left[  X_{2},\left[  X_{4},X_{1}\right]  \right]  \right]
+\left[  X_{3},\left[  X_{4},\left[  X_{1},X_{2}\right]  \right]  \right]  +\\
&  \left[  X_{4},\left[  X_{1},\left[  X_{3},X_{2}\right]  \right]  \right]
+\left[  X_{4},\left[  X_{2},\left[  X_{1},X_{3}\right]  \right]  \right]
+\left[  X_{4},\left[  X_{3},\left[  X_{2},X_{1}\right]  \right]  \right] \\
&  =0
\end{align*}

\end{corollary}

\begin{proof}
Let%
\begin{align*}
\gamma_{1234}  &  =X_{4}\ast X_{3}\ast X_{2}\ast X_{1},\gamma_{1243}=\left(
X_{3}\ast X_{4}\ast X_{2}\ast X_{1}\right)  ^{\sigma_{1243}},\\
\gamma_{1324}  &  =\left(  X_{4}\ast X_{2}\ast X_{3}\ast X_{1}\right)
^{\sigma_{1324}},\gamma_{1342}=\left(  X_{2}\ast X_{4}\ast X_{3}\ast
X_{1}\right)  ^{\sigma_{1342}},\\
\gamma_{1423}  &  =\left(  X_{3}\ast X_{2}\ast X_{4}\ast X_{1}\right)
^{\sigma_{1423}},\gamma_{1432}=\left(  X_{2}\ast X_{3}\ast X_{4}\ast
X_{1}\right)  ^{\sigma_{1432}},\\
\gamma_{2134}  &  =\left(  X_{4}\ast X_{3}\ast X_{1}\ast X_{2}\right)
^{\sigma_{2134}},\gamma_{2143}=\left(  X_{3}\ast X_{4}\ast X_{1}\ast
X_{2}\right)  ^{\sigma_{2143}},\\
\gamma_{2314}  &  =\left(  X_{4}\ast X_{1}\ast X_{3}\ast X_{2}\right)
^{\sigma_{2314}},\gamma_{2341}=\left(  X_{1}\ast X_{4}\ast X_{3}\ast
X_{2}\right)  ^{\sigma_{2341}},\\
\gamma_{2413}  &  =\left(  X_{3}\ast X_{1}\ast X_{4}\ast X_{2}\right)
^{\sigma_{2413}},\gamma_{2431}=\left(  X_{1}\ast X_{3}\ast X_{4}\ast
X_{2}\right)  ^{\sigma_{2431}},\\
\gamma_{3124}  &  =\left(  X_{4}\ast X_{2}\ast X_{1}\ast X_{3}\right)
^{\sigma_{3124}},\gamma_{3142}=\left(  X_{2}\ast X_{4}\ast X_{1}\ast
X_{3}\right)  ^{\sigma_{3142}},\\
\gamma_{3214}  &  =\left(  X_{4}\ast X_{1}\ast X_{2}\ast X_{3}\right)
^{\sigma_{3214}},\gamma_{3241}=\left(  X_{1}\ast X_{4}\ast X_{2}\ast
X_{3}\right)  ^{\sigma_{3241}},\\
\gamma_{3412}  &  =\left(  X_{2}\ast X_{1}\ast X_{4}\ast X_{3}\right)
^{\sigma_{3412}},\gamma_{3421}=\left(  X_{1}\ast X_{2}\ast X_{4}\ast
X_{3}\right)  ^{\sigma_{3421}},\\
\gamma_{4123}  &  =\left(  X_{3}\ast X_{2}\ast X_{1}\ast X_{4}\right)
^{\sigma_{4123}},\gamma_{4132}=\left(  X_{2}\ast X_{3}\ast X_{1}\ast
X_{4}\right)  ^{\sigma_{4132}},\\
\gamma_{4213}  &  =\left(  X_{3}\ast X_{1}\ast X_{2}\ast X_{4}\right)
^{\sigma_{4213}},\gamma_{4231}=\left(  X_{1}\ast X_{3}\ast X_{2}\ast
X_{4}\right)  ^{\sigma_{4231}},\\
\gamma_{4312}  &  =\left(  X_{2}\ast X_{1}\ast X_{3}\ast X_{4}\right)
^{\sigma_{4312}},\gamma_{4321}=\left(  X_{1}\ast X_{2}\ast X_{3}\ast
X_{4}\right)  ^{\sigma_{4321}}%
\end{align*}
with%
\begin{align*}
\sigma_{1243}  &  =\left(
\begin{array}
[c]{c}%
1234\\
1243
\end{array}
\right)  ,\sigma_{1324}=\left(
\begin{array}
[c]{c}%
1234\\
1324
\end{array}
\right)  ,\sigma_{1342}=\left(
\begin{array}
[c]{c}%
1234\\
1423
\end{array}
\right)  ,\sigma_{1423}=\left(
\begin{array}
[c]{c}%
1234\\
1342
\end{array}
\right)  ,\sigma_{1432}=\left(
\begin{array}
[c]{c}%
1234\\
1432
\end{array}
\right)  ,\\
\sigma_{2134}  &  =\left(
\begin{array}
[c]{c}%
1234\\
2134
\end{array}
\right)  ,\sigma_{2143}=\left(
\begin{array}
[c]{c}%
1234\\
2143
\end{array}
\right)  ,\sigma_{2314}=\left(
\begin{array}
[c]{c}%
1234\\
3124
\end{array}
\right)  ,\sigma_{2341}=\left(
\begin{array}
[c]{c}%
1234\\
4123
\end{array}
\right)  ,\sigma_{2413}=\left(
\begin{array}
[c]{c}%
1234\\
3142
\end{array}
\right)  ,\\
\sigma_{2431}  &  =\left(
\begin{array}
[c]{c}%
1234\\
4132
\end{array}
\right)  ,\sigma_{3124}=\left(
\begin{array}
[c]{c}%
1234\\
2314
\end{array}
\right)  ,\sigma_{3142}=\left(
\begin{array}
[c]{c}%
1234\\
2413
\end{array}
\right)  ,\sigma_{3214}=\left(
\begin{array}
[c]{c}%
1234\\
3214
\end{array}
\right)  ,\sigma_{3241}=\left(
\begin{array}
[c]{c}%
1234\\
4213
\end{array}
\right)  ,\\
\sigma_{3412}  &  =\left(
\begin{array}
[c]{c}%
1234\\
3412
\end{array}
\right)  ,\sigma_{3421}=\left(
\begin{array}
[c]{c}%
1234\\
4312
\end{array}
\right)  ,\sigma_{4123}=\left(
\begin{array}
[c]{c}%
1234\\
2341
\end{array}
\right)  ,\sigma_{4132}=\left(
\begin{array}
[c]{c}%
1234\\
2431
\end{array}
\right)  ,\sigma_{4213}=\left(
\begin{array}
[c]{c}%
1234\\
3241
\end{array}
\right)  ,\\
\sigma_{4231}  &  =\left(
\begin{array}
[c]{c}%
1234\\
4231
\end{array}
\right)  ,\sigma_{4312}=\left(
\begin{array}
[c]{c}%
1234\\
3421
\end{array}
\right)  ,\sigma_{4321}=\left(
\begin{array}
[c]{c}%
1234\\
4321
\end{array}
\right)
\end{align*}
Then it is easy to see that%
\begin{align*}
&  \left[  X_{1},\left[  X_{2},\left[  X_{3},X_{4}\right]  \right]  \right] \\
&  =\left(  \left(  \gamma_{1234}\overset{\cdot}{\underset{12}{-}}%
\gamma_{1243}\right)  \overset{\cdot}{\underset{1}{-}}\left(  \gamma
_{1342}\overset{\cdot}{\underset{12}{-}}\gamma_{1432}\right)  \right)
\overset{\cdot}{-}\left(  \left(  \gamma_{2341}\overset{\cdot}{\underset
{12}{-}}\gamma_{2431}\right)  \overset{\cdot}{\underset{1}{-}}\left(
\gamma_{3421}\overset{\cdot}{\underset{12}{-}}\gamma_{4321}\right)  \right)
\end{align*}%
\begin{align*}
&  \left[  X_{1},\left[  X_{3},\left[  X_{4},X_{2}\right]  \right]  \right] \\
&  =\left(  \left(  \gamma_{1342}\overset{\cdot}{\underset{13}{-}}%
\gamma_{1324}\right)  \overset{\cdot}{\underset{1}{-}}\left(  \gamma
_{1423}\overset{\cdot}{\underset{13}{-}}\gamma_{1243}\right)  \right)
\overset{\cdot}{-}\left(  \left(  \gamma_{3421}\overset{\cdot}{\underset
{13}{-}}\gamma_{3241}\right)  \overset{\cdot}{\underset{1}{-}}\left(
\gamma_{4231}\overset{\cdot}{\underset{13}{-}}\gamma_{2431}\right)  \right)
\end{align*}%
\begin{align*}
&  \left[  X_{1},\left[  X_{4},\left[  X_{2},X_{3}\right]  \right]  \right] \\
&  =\left(  \left(  \gamma_{1423}\overset{\cdot}{\underset{14}{-}}%
\gamma_{1432}\right)  \overset{\cdot}{\underset{1}{-}}\left(  \gamma
_{1234}\overset{\cdot}{\underset{14}{-}}\gamma_{1324}\right)  \right)
\overset{\cdot}{-}\left(  \left(  \gamma_{4231}\overset{\cdot}{\underset
{14}{-}}\gamma_{4321}\right)  \overset{\cdot}{\underset{1}{-}}\left(
\gamma_{2341}\overset{\cdot}{\underset{14}{-}}\gamma_{3241}\right)  \right)
\end{align*}%
\begin{align*}
&  \left[  X_{2},\left[  X_{1},\left[  X_{4},X_{3}\right]  \right]  \right] \\
&  =\left(  \left(  \gamma_{2143}\overset{\cdot}{\underset{21}{-}}%
\gamma_{2134}\right)  \overset{\cdot}{\underset{1}{-}}\left(  \gamma
_{2431}\overset{\cdot}{\underset{21}{-}}\gamma_{2341}\right)  \right)
\overset{\cdot}{-}\left(  \left(  \gamma_{1432}\overset{\cdot}{\underset
{21}{-}}\gamma_{1342}\right)  \overset{\cdot}{\underset{1}{-}}\left(
\gamma_{4312}\overset{\cdot}{\underset{21}{-}}\gamma_{3412}\right)  \right)
\end{align*}%
\begin{align*}
&  \left[  X_{2},\left[  X_{3},\left[  X_{1},X_{4}\right]  \right]  \right] \\
&  =\left(  \left(  \gamma_{2314}\overset{\cdot}{\underset{23}{-}}%
\gamma_{2341}\right)  \overset{\cdot}{\underset{1}{-}}\left(  \gamma
_{2143}\overset{\cdot}{\underset{23}{-}}\gamma_{2413}\right)  \right)
\overset{\cdot}{-}\left(  \left(  \gamma_{3142}\overset{\cdot}{\underset
{23}{-}}\gamma_{3412}\right)  \overset{\cdot}{\underset{1}{-}}\left(
\gamma_{1432}\overset{\cdot}{\underset{23}{-}}\gamma_{4132}\right)  \right)
\end{align*}%
\begin{align*}
&  \left[  X_{2},\left[  X_{4},\left[  X_{3},X_{1}\right]  \right]  \right] \\
&  =\left(  \left(  \gamma_{2431}\overset{\cdot}{\underset{24}{-}}%
\gamma_{2413}\right)  \overset{\cdot}{\underset{1}{-}}\left(  \gamma
_{2314}\overset{\cdot}{\underset{24}{-}}\gamma_{2134}\right)  \right)
\overset{\cdot}{-}\left(  \left(  \gamma_{4312}\overset{\cdot}{\underset
{24}{-}}\gamma_{4132}\right)  \overset{\cdot}{\underset{1}{-}}\left(
\gamma_{3142}\overset{\cdot}{\underset{24}{-}}\gamma_{1342}\right)  \right)
\end{align*}%
\begin{align*}
&  \left[  X_{3},\left[  X_{1},\left[  X_{2},X_{4}\right]  \right]  \right] \\
&  =\left(  \left(  \gamma_{3124}\overset{\cdot}{\underset{31}{-}}%
\gamma_{3142}\right)  \overset{\cdot}{\underset{1}{-}}\left(  \gamma
_{3241}\overset{\cdot}{\underset{31}{-}}\gamma_{3421}\right)  \right)
\overset{\cdot}{-}\left(  \left(  \gamma_{1243}\overset{\cdot}{\underset
{31}{-}}\gamma_{1423}\right)  \overset{\cdot}{\underset{1}{-}}\left(
\gamma_{2413}\overset{\cdot}{\underset{31}{-}}\gamma_{4213}\right)  \right)
\end{align*}%
\begin{align*}
&  \left[  X_{3},\left[  X_{2},\left[  X_{4},X_{1}\right]  \right]  \right] \\
&  =\left(  \left(  \gamma_{3241}\overset{\cdot}{\underset{32}{-}}%
\gamma_{3214}\right)  \overset{\cdot}{\underset{1}{-}}\left(  \gamma
_{3412}\overset{\cdot}{\underset{32}{-}}\gamma_{3142}\right)  \right)
\overset{\cdot}{-}\left(  \left(  \gamma_{2413}\overset{\cdot}{\underset
{32}{-}}\gamma_{2143}\right)  \overset{\cdot}{\underset{1}{-}}\left(
\gamma_{4123}\overset{\cdot}{\underset{32}{-}}\gamma_{1423}\right)  \right)
\end{align*}%
\begin{align*}
&  \left[  X_{3},\left[  X_{4},\left[  X_{1},X_{2}\right]  \right]  \right] \\
&  =\left(  \left(  \gamma_{3412}\overset{\cdot}{\underset{34}{-}}%
\gamma_{3421}\right)  \overset{\cdot}{\underset{1}{-}}\left(  \gamma
_{3124}\overset{\cdot}{\underset{34}{-}}\gamma_{3214}\right)  \right)
\overset{\cdot}{-}\left(  \left(  \gamma_{4123}\overset{\cdot}{\underset
{34}{-}}\gamma_{4213}\right)  \overset{\cdot}{\underset{1}{-}}\left(
\gamma_{1243}\overset{\cdot}{\underset{34}{-}}\gamma_{2143}\right)  \right)
\end{align*}%
\begin{align*}
&  \left[  X_{4},\left[  X_{1},\left[  X_{3},X_{2}\right]  \right]  \right] \\
&  =\left(  \left(  \gamma_{4132}\overset{\cdot}{\underset{41}{-}}%
\gamma_{4123}\right)  \overset{\cdot}{\underset{1}{-}}\left(  \gamma
_{4321}\overset{\cdot}{\underset{41}{-}}\gamma_{4231}\right)  \right)
\overset{\cdot}{-}\left(  \left(  \gamma_{1324}\overset{\cdot}{\underset
{41}{-}}\gamma_{1234}\right)  \overset{\cdot}{\underset{1}{-}}\left(
\gamma_{3214}\overset{\cdot}{\underset{41}{-}}\gamma_{2314}\right)  \right)
\end{align*}%
\begin{align*}
&  \left[  X_{4},\left[  X_{2},\left[  X_{1},X_{3}\right]  \right]  \right] \\
&  =\left(  \left(  \gamma_{4213}\overset{\cdot}{\underset{42}{-}}%
\gamma_{4231}\right)  \overset{\cdot}{\underset{1}{-}}\left(  \gamma
_{4132}\overset{\cdot}{\underset{42}{-}}\gamma_{4312}\right)  \right)
\overset{\cdot}{-}\left(  \left(  \gamma_{2134}\overset{\cdot}{\underset
{42}{-}}\gamma_{2314}\right)  \overset{\cdot}{\underset{1}{-}}\left(
\gamma_{1324}\overset{\cdot}{\underset{42}{-}}\gamma_{3124}\right)  \right)
\end{align*}%
\begin{align*}
&  \left[  X_{4},\left[  X_{3},\left[  X_{2},X_{1}\right]  \right]  \right] \\
&  =\left(  \left(  \gamma_{4321}\overset{\cdot}{\underset{43}{-}}%
\gamma_{4312}\right)  \overset{\cdot}{\underset{1}{-}}\left(  \gamma
_{4213}\overset{\cdot}{\underset{43}{-}}\gamma_{4123}\right)  \right)
\overset{\cdot}{-}\left(  \left(  \gamma_{3214}\overset{\cdot}{\underset
{43}{-}}\gamma_{3124}\right)  \overset{\cdot}{\underset{1}{-}}\left(
\gamma_{2134}\overset{\cdot}{\underset{43}{-}}\gamma_{1234}\right)  \right)
\end{align*}

\end{proof}


\begin{thebibliography}{99}                                                                                               %


\bibitem {bh}D. Blessenohl and H. Laue, Generalized Jacobi identities, Note di
Matematica \textbf{8} (1988), 111-121.

\bibitem {ko}A. Kock, Synthetic Differential Geometry (2nd ed.), Cambridge
University Press 2006.

\bibitem {kola}A. Kock and R. Lavendhomme, Strong infinitesimal linearity,
with applications to strong difference and affine connections, Cahiers de
Topologie et G\'{e}om\'{e}trie Diff. Cat\'{e}goriques \textbf{28} (1987), 311-324.

\bibitem {la}R. Lavendhomme, Basic Concepts of Synthetic Differential
Geometry, Kluwer Academic Publishers 1996.

\bibitem {ma}K. Mackenzie, Proving the Jacobi identity the hard way, in
Geometric Methods in Physics (Trends Math.), 357-366, Birkh\"{a}user/Springer, 2013.

\bibitem {ni1}H. Nishimura, Theory of microcubes, International Journal of
Theoretical Physics \textbf{36} (1997), 1099-1131.

\bibitem {ni2}H. Nishimura, General Jacobi identity revisited, International
Journal of Theoretical Physics \textbf{38} (1999), 2163-2174.

\bibitem {ni3}H. Nishimura and T. Osoekawa, General Jacobi identity revisited
again, International Journal of Theoretical Physics \textbf{46} (2007), 2843-2862

\bibitem {ni4}H. Nishimura, The Jacobi identity beyond Lie algebras, Far East
Journal of Mathematical Sciences \textbf{35} (2009), 33-48.

\bibitem {ni5}H. Nishimura, Synthetic differential geometry within homotopy
type theory, arXiv math CT/1593662 (2016).

\bibitem {re}C. Reutenauer, Free Lie Algebras, Oxford University Press 1993.

\bibitem {we}F. Wever, \"{U}ber Invaianten in Lie'schen Ringen, Math. Ann.
\textbf{120} (1949), 563-580.
\end{thebibliography}
\end{document}